\documentclass[11pt]{amsart}

\usepackage{amssymb}
\usepackage{amsmath}
\usepackage{amsthm}
\usepackage{color}
\usepackage{euscript}
\usepackage[linktoc=page,colorlinks,citecolor = green!50!black,linkcolor= blue,urlcolor = blue]{hyperref}
\usepackage{enumitem}
\usepackage{tikz}
\usepackage{mathrsfs}
\usepackage[all]{xy}
\usepackage{verbatim}
\usepackage{comment}
\usepackage[margin=1in]{geometry}
\usepackage{cleveref}

\theoremstyle{plain}
\newtheorem{thm}{Theorem}[section]
\newtheorem{lem}[thm]{Lemma}

\newtheorem{cor}[thm]{Corollary}
\newtheorem{prop}[thm]{Proposition}

\newtheorem{assump}[thm]{Assumption}

\theoremstyle{definition}
\newtheorem{defn}[thm]{Definition}
\newtheorem{exmp}[thm]{Example}

\theoremstyle{definition}
\newtheorem{note}[thm]{Note}
\newtheorem{rk}[thm]{Remark}
\newtheorem*{notation}{Notation}

\newtheorem{fact}[thm]{Fact}

\newcommand{\bA}{{\mathbb A}}

\newcommand{\bC}{{\mathbb C}}

\newcommand{\bG}{{\mathbb G}}

\newcommand{\bN}{{\mathbb N}}

\newcommand{\bR}{{\mathbb R}}

\newcommand{\bZ}{{\mathbb Z}}

\newcommand{\cB}{{\mathcal B}}
\newcommand{\cC}{{\mathcal C}}
\newcommand{\cD}{{\mathcal D}}

\newcommand{\cF}{{\mathcal F}}

\newcommand{\cH}{{\mathcal H}}

\newcommand{\cK}{{\mathcal K}}
\newcommand{\cL}{{\mathcal L}}
\newcommand{\cM}{{\mathcal M}}
\newcommand{\cN}{{\mathcal N}}
\newcommand{\cO}{{\mathcal O}}

\newcommand{\cW}{{\mathcal W}}

\newcommand{\scrF}{\EuScript F}

\newcommand{\A}{\mathcal{A}}

\newcommand{\B}{\mathcal{B}}

\newcommand{\fM}{\mathfrak{M}}
\newcommand{\fN}{\mathfrak{N}}

\newcommand{\fstop}{\mathfrak{f}}
\newcommand{\sheafhom}{\mathcal{H} \kern -.5pt \mathit{om}}
\newcommand{\derivedsheafhom}{\mathcal{R}\mathcal{H} \kern -.5pt \mathit{om}}

\newcommand{\frlag}{\mathfrak{L}\mathfrak{a}\mathfrak{g}}
\newcommand{\red}[1]{\textcolor{red}{#1}}

\numberwithin{equation}{section}
\numberwithin{figure}{section}
\title[Algebraic sheaves of Floer homology groups]{Algebraic sheaves of Floer homology groups via algebraic torus actions on the Fukaya category}
\author[Yusuf Bar\i\c{s} Kartal]{Yusuf Bar\i\c{s} Kartal}
\address{Department of Mathematics, Princeton University, Princeton, NJ 08544, USA}
\email{ykartal@math.princeton.edu}
\date{}
\begin{document}
	
\maketitle	
\begin{abstract}
Let $(M,\omega_M)$ be a monotone or negatively monotone symplectic manifold, or a Weinstein manifold. One can construct an ``action'' of $H^1(M,\bG_m)$ on the Fukaya category (wrapped Fukaya category in the exact case) that reflects the action of $Symp^0(M,\omega_M)$ on the set of Lagrangian branes. A priori this action is only analytic. The purpose of this work is to show the algebraicity of this action under some assumptions. 

We use this to prove a tameness result for the sheaf of Lagrangian Floer homology groups obtained by moving one of the Lagrangians via global symplectic isotopies. We also show the algebraicity of the locus of $z\in H^1(M,\bG_m)$ that fix a Lagrangian brane in the Fukaya category. The latter has applications to Lagrangian flux. Finally, we prove a statement in mirror symmetry: in the Weinstein case, assume that $M$ is mirror to an affine or projective variety $X$, that there exists an exact Lagrangian torus $L\subset M$ such that $H^1(M)\to H^1(L)$ is surjective, and that $L$ is sent to a smooth point of $x\in X$ under the mirror equivalence. Then we construct a Zariski chart of $X$ containing $x$, that is isomorphic to $H^1(L,\bG_m)$, and such that other points of this chart correspond to non-exact deformations of $L$ (possibly equipped with unitary local systems). In particular, this implies rationality of the irreducible component containing $x$; however, it is stronger. 

Under our assumptions, one can construct an algebraic action of $H^1(M,\bG_m)$, namely the action by non-unitary local systems. By combining techniques from family Floer homology and non-commutative geometry, we prove that this action coincides with the geometric action mentioned in the first paragraph. We use this to deduce the theorems above.
\end{abstract}

\setcounter{tocdepth}{1}
\tableofcontents
\parskip1em
\parindent0em	

\section{Introduction}\label{sec:intro}
\subsection{Motivation}
Let $(M,\omega_M)$ be a closed symplectic manifold. Then, there exists a natural group homomorphism
\begin{equation}\label{eq:sympoverham}
Flux:\widetilde{Symp^0}/\widetilde{Ham^0}\to H^1(M,\bR)
\end{equation}
called the \textbf{flux map}. Here, $\widetilde{Symp^0}$ denote the universal cover of the identity component of $Symp(M,\omega_M)$, and $\widetilde{Ham^0}$ is the preimage of Hamiltonian diffeomorphisms under the covering map. This map is an isomorphism by Banyaga's theorem. Let $L$ be a Lagrangian submanifold. One can define a family of Lagrangians 
\begin{equation}
\{\phi_1(L): [\{\phi_t\}]\in \widetilde{Symp^0}\text{ lifts } \phi_1\in Symp^0(M,\omega_M) \}	
\end{equation}
parametrized by $\widetilde{Symp^0}$; hence, a family of Lagrangians parametrized by $H^1(M,\bR)$, which is well-defined only up to Hamiltonian isotopy. The purpose of this article is to explore properties of this family over $H^1(M,\bR)$, when $M$ is monotone, negatively monotone or exact.

Let $\frlag$ denote the set of closed Lagrangians up to Hamiltonian isotopy. $\frlag$ carries a natural topology, and it can be endowed with a (local) integral affine structure, where the coordinates are given by action coordinates (see \cite{kosoaffine} for a special case). Moreover, $Symp^0$ action on $\frlag$ induces an action of $\widetilde{Symp^0}$, and thus of $H^1(M,\bR)$. The family above can be seen as the orbit map of $L\in\frlag$, and this is an affine map $H^1(M,\bR)\to \frlag$. Similarly, one can also consider the moduli of Lagrangians endowed with unitary local systems, denoted by $\frlag^+$. By \cite{fukayamirror}, this moduli space carries a local analytic structure. Then, the extended symplectomorphism group (i.e. the group of symplectomorphisms and unitary local systems) modulo Hamiltonians act on $\frlag^+$, inducing an analogous $H^1(M,\bG_m)$-action on $\frlag^+$. Denote the action of $z\in H^1(M,\bG_m)$ by $\phi_z$. This action is (locally) analytic. Similar to above, one can define a family of Lagrangians endowed with local systems parametrized by $H^1(M,\bG_m)$, which can be thought as an analytic family.

The questions this article tries to answer are about the complexity of the families over $H^1(M,\bR)$ and $H^1(M,\bG_m)$. These families are a priori affine/analytic, and our main results can be seen as algebraicity of these families. As we will see, it will have implications to growth of Floer homology groups, generalizing the results of \cite{owniteratespadic}, as well as to the flux group of Lagrangians in monotone or negatively monotone symplectic manifolds. We also use it to prove a result in homological mirror symmetry for exact symplectic manifolds.

Even thought the local structures of $\frlag$ and $\frlag^+$ are easy to describe, these spaces are hard to work with. To our knowledge, they do not have nice constructions analogous to moduli of sheaves in algebraic geometry, and they do not form Artin stacks. Therefore, we prefer to replace them by ``moduli of objects'' in the Fukaya category. These moduli can be described as stacks as in \cite{toenvaquiemoduli}, and under some assumptions on the category, they are relatively tame. However, we will avoid the use of their language and treat ``moduli of objects'' as heuristics, and refer to its ``functor of points''. More precisely, we will use families of objects/modules over the Fukaya category, and use this as an algebraic replacement for families of Lagrangians. The approach via Fukaya categories will allow us to apply techniques of homological algebra and non-commutative geometry. 

Therefore, we study the continuous dynamics on the Fukaya category. One expects an analytic action of $H^1(M,\bG_m)$ on the Fukaya category, by family Floer homology \cite{abouzaidicm}. However, this action is often not ``algebraic''. For instance, if $M=E$ is an elliptic curve, then its Fukaya category is derived equivalent to $D^bCoh(E^\vee)$, where $E^\vee$ is the dual elliptic curve and the action of one of the cocharacters of $H^1(M,\bG_m)=\bG_m^2$ translates as the the action of $\bG_m$ via uniformization. This action is clearly non-algebraic. Fix $x\in E^\vee$ and consider the sheaf $RHom(\cO_{z.x},\cO_x), z\in \bG_m$ over $\bG_m$. This sheaf has $0$ dimensional, but infinite support in $\bG_m$; hence, it is only an analytic coherent sheaf. The result translates to the mirror dual torus analogously. Geometrically, the subgroup $\bR\subset \bG_m$ acts by rotation in fixed direction (which is extended to a $\bG_m$-action by unitary local systems), and the integral points $\bZ\subset\bR\subset \bG_m$ correspond to iterates of the full rotation. Therefore, if $L=L'$ is a simple non-separating curve orthogonal to the direction of the rotation, then the rank of $HF(\phi_z(L),L'), z\in\bG_m\subset H^1(E,\bG_m)$ is $0$ outside the subgroup $\bZ\subset \bG_m$, and it is $2$ at the points of this subgroup. 

\subsection{Algebraicity of the sheaf of Floer homologies}
Our first goal is to show that this does not happen in the monotone, negatively monotone, or exact case under some assumptions. More precisely, 
\begin{thm}\label{thm:algsheaf}
Let $L,L'\subset M$ be tautologically unobstructed closed Lagrangian branes. Then, there exists a finite complex of algebraic coherent sheaves over $H^1(M,\bG_m)$ whose restriction at $z\in H^1(M,\bG_m)$ has cohomology isomorphic to $HF(L,\phi_z(L') )$. 
\end{thm}
\begin{assump}\label{assump:monotoneexact}
	If $M$ is closed, then it is monotone or negatively monotone, non-degenerate, and its Fukaya category $\cF(M)$ is generated by a set of Bohr-Sommerfeld monotone Lagrangians. If $M$ is open, then it is Weinstein. 	
\end{assump}
If we do not specify whether $M$ is closed, resp. exact, then the phrase ``Fukaya category'' refers to compact Fukaya category, resp. wrapped Fukaya category (or generating subcategories). We define Fukaya category and Floer homology over the Novikov field $\Lambda=\bC((T^\bR))$. The notation $\bG_m$ also refers to the multiplicative group over $\Lambda$. Any $z\in H^1(M,\bG_m)$ splits as $``z=T^{val_T z}z_0$'', where $val_T z\in H^1(M,\bR)$ is obtained by componentwise application of $T$-adic valuation and $z_0\in H^1(M,U_\Lambda)$, the set of ``unitary classes''. Let $\alpha$ be a closed $1$-form representing $val_T z$ such that the corresponding symplectic vector field $X_\alpha$ (defined by $\iota_{X_\alpha}\omega_M=-\alpha$ in our conventions) is complete with flow $\phi_\alpha^t$. Let $\xi_{z_0}$ denote the unitary local system corresponding to $z_0$. We define the action of $\phi_z$ on the brane $(L,\xi)$ by $\phi_z(L,\xi)=(\phi_\alpha^1(L),\xi\otimes\xi_{z_0})$. The only ambiguity is in the choice of $1$-form $\alpha$, but different choices result in Hamiltonian isotopic branes. In particular, the Floer homology is well-defined. 

Theorem \ref{thm:algsheaf} can be seen as a tameness result for the family Floer sheaf (where one restricts to Lagrangians deformed by some $\phi_z$). In particular, it implies,
\begin{cor}\label{cor:zariskiopenrank}
There exists a non-empty Zariski open subset of $H^1(M,\bG_m)$ on which the rank of $HF(L,\phi_z(L') )$ is constant. 
\end{cor}
More generally, 
\begin{cor}\label{cor:rankstratification}
The set of $z\in H^1(M,\bG_m)$ at which the rank of $HF(L,\phi_z(L') )$ is at least $k$ forms a closed algebraic subvariety of $H^1(M,\bG_m)$, for any $k$. Therefore, constant rank loci of $HF(L,\phi_z(L') )$ form an algebraic stratification of $H^1(M,\bG_m)$.
\end{cor}
Theorem \ref{thm:algsheaf} can be used to strengthen one of the main results of \cite{owniteratespadic}. \cite[Theorem 1.1]{owniteratespadic} states that under Assumption \ref{assump:monotoneexact}, for any $\phi\in Symp^0(M,\omega)$, the rank $HF(L,\phi^k(L'))$ is constant in $k\in\bZ$ with finitely many possible exceptions. On the other hand, Theorem \ref{thm:algsheaf} implies:
\begin{cor}\label{cor:realiterates}
Let $\phi^t_\alpha$	denote the flow of a closed $1$-form $\alpha$. Then the rank of $HF(L,\phi^t_\alpha(L'))$ is constant in $t\in \bR$ with finitely many possible exceptions. 
\end{cor}
\begin{rk}
\cite[Theorem 1.1]{owniteratespadic} uses ``the $p$-adic method'' inspired by \cite{skolemmahler} and \cite{dynmordellforcoherent}. Even though its implication is weaker than Corollary \ref{cor:realiterates}, we also apply that technique in the non-monotone case, under the assumption that $\phi$ is generic (see \cite[Theorem 1.5]{owniteratespadic}). The ``$p$-adic method" also has the promise of being applicable without genericity of $\phi$ (in the non-monotone case). In this case, we expect $rk(HF(L,\phi^k(L')))$ to differ from an arithmetic sequence at finitely many $k\in\bZ$. 
\end{rk}
\subsection{Algebraic stabilizers and the flux groups of Lagrangians}
As we have explained, one already expects to have an ``analytic action of $H^1(M,\bG_m)$'' on the Fukaya category, and we show algebraicity of this action. As a result, one also expects ``algebraic stabilizers''. Our second main result makes this precise. Call two objects $L,L'$ of the Fukaya category \textbf{stably isomorphic}, if $L$ is quasi-isomorphic to a direct summand of $L'^{\oplus q}$ for some $q\gg 0$, and vice versa. 
\begin{thm}\label{thm:stabilizer}
Let $L$ be a tautologically unobstructed, closed Lagrangian brane in $M$. Then the set of $z\in H^1(M,\bG_m)$ such that $L$ is stably isomorphic to $\phi_z(L)$ form a subtorus of $H^1(M,\bG_m)$ whose Lie algebra is given by the kernel of the map $H^1(M,\Lambda)\to H^1(L,\Lambda)$.
\end{thm}
Observe that the first claim is stronger than just being an algebraic subgroup, we also claim this subgroup is connected. Theorem \ref{thm:stabilizer} implies:
\begin{cor}\label{cor:vanishingflux}
Assume $M$ is closed. Given $\phi\in Symp^0(M,\omega_M)$, if $\phi(L)$ is Hamiltonian isotopic to $L$, then the flux of an isotopy from $1$ to $\phi$ restricts to $0$ on $L$. 
\end{cor}
\begin{rk}
Corollary \ref{cor:vanishingflux} is stated for closed case due to subtleties related to flux on a non-compact manifold. One often uses compactly generated symplectic isotopies to define this; however, this does not suffice for our purposes. 
\end{rk}
Recall that the flux group of a monotone symplectic manifold vanishes (see \cite{lalondemcduffpolterovichonflux}, \cite{luptonoprea}). The following can be seen as a Lagrangian version of the same statement.
\begin{cor}\label{cor:lagrvanishingflux}
Assume the map $H^1(M,\bR)\to H^1(L,\bR)$ is surjective. Under the assumptions above, the flux of a Lagrangian isotopy from $L$ to itself vanishes (this is an element of $H^1(L,\bR)$).
\end{cor}
This result may be well known, but we are not aware of it in the literature. Contrary to Corollary \ref{cor:vanishingflux}, it is valid in the non-compact case as well.

Before turning to the next theorem, we would like to mention the possibility that Theorem \ref{thm:stabilizer} can be used to reprove \cite[Theorem 1.3]{onoflux}, which states that if the map $H^1(M,\bR)\to H^1(L,\bR)$ is surjective, then the component of $L$ within $\frlag$ is Hausdorff. 

\subsection{Mirror symmetry, affine torus charts and rationality}
We also study the implications in the mirror symmetry. Namely, to construct the sheaf mentioned in Theorem \ref{thm:algsheaf}, we actually construct families of ``quasi-functors'' (a.k.a. bimodules) of the Fukaya category. The same family can be used to construct ``algebraic families of objects'' corresponding to deformations of $L$ by global symplectomorphisms. In other words, this can be considered as a map from $H^1(M,\bG_m)$ to the ``moduli of objects'', as we mentioned above. Assume that the map $H^1(M,\bR)\to H^1(L,\bR)$ is surjective. By Theorem \ref{thm:stabilizer}, one can construct the map from $H^1(L,\bG_m)$ to the moduli (more precisely, a family parametrized by $H^1(L,\bG_m)$), and it is an injective map. Hence, this can be seen as a construction of torus charts in the moduli of objects. Putting the heuristics aside, one can prove the following:
\begin{thm}\label{thm:rationalmirror}
Assume $M$ is Weinstein, $L$ is as above and a Lagrangian torus, and $H^1(M,\bR)\to H^1(L,\bR)$ is surjective. Assume $M$ is mirror dual to a projective or affine variety $X$ over $\Lambda$, in the sense that the wrapped Fukaya category is $\bZ$-graded and derived equivalent to $D^b Coh(X)$. Further assume, the equivalence maps $L$ to a sky-scraper sheaf of a smooth point $x\in X$. Then $x$ lies in a Zariski chart isomorphic to $H^1(L,\bG_m)\cong \bG_m^{b_1(L)}$. In particular, its irreducible component is rational. Under the given equivalence, other points in this chart correspond to Lagrangian tori deforming $L$ (possibly equipped with unitary local systems) inside $M$.	
\end{thm}
Our proof actually implies this claim for Fukaya-Seidel categories of Lefschetz fibrations. Note that we need $\bZ$-grading and $\bZ$-graded equivalence. 
\begin{note}
As mentioned, for the proof of \Cref{thm:rationalmirror}, we essentially construct a map from $H^1(L,\bG_m)$ to the moduli of objects, which on the mirror gives a map to the moduli of objects of $D^bCoh(X)$. The condition that $L$ is mirror to a point means this map hits the component containing skyscraper sheaves; hence, it is basically a map to $X$. On the other hand, this condition can be dropped in some cases. Assume $X$ is smooth, and admits a smooth compactification $\bar X$ such that $D=\bar X\setminus X$ is an ample divisor. Also assume the pair $(\bar X,D)$ is mirror to a pair $(M,\fstop)$, where $\fstop$ is a nice stop. In other words, $D^bCoh(\bar X)$ is also equivalent to the partially wrapped Fukaya category $\cW(M,\fstop)$, and the restriction to $X$ on the $B$-side corresponds to the stop removal functor on the $A$-side. Then, if $L$ is compact and exact, \cite[Thm 6.27, Cor 8.2]{ownwithcotefiltrationgrowth} implies it corresponds to a complex of sheaves on $X$ with $0$-dimensional cohomological support. The length of the support is given by $b_0(L)$. Presumably, from the family above, one obtains a map from $H^1(L,\bG_m)$ to the Hilbert scheme of $X$, and under further topological assumptions on $L$ (at least connectedness, and possibly that $L$ is a homology torus), an actual map to $X$ itself. 
\end{note}
\begin{note}
As long as $L$ is exact, one can construct an algebraic family of modules parametrized by $H^1(L,\bG_m)$ (analogous to $h_L^{alg}$ that appears later, and essentially corresponding to twists by non-unitary local systems). Therefore, under the conditions of \Cref{thm:rationalmirror}, one still has a map from $H^1(L,\bG_m)$ to the moduli of objects of $D^bCoh(X)$, without the surjectivity assumption on cohomology. As remarked to us by Mohammed Abouzaid, in the exact case, the injectivity of this map is easy and does not require surjectivity assumption on the cohomology. The major point of \Cref{thm:rationalmirror} is that the deformations by non-unitary local systems are realized by flux. In other words, the other points in the chart are realized by the Lagrangian tori deforming $L$ (with unitary local systems). This theorem can have applications in SYZ mirror symmetry as well. 
\end{note}
\begin{note}
Under the assumptions of \Cref{thm:rationalmirror}, one has not only an $H^1(L,\bG_m)$-chart, but actually an $H^1(L,\bG_m)$-action on the category, where the chart is a single orbit. Up to some technical details, this likely implies that the component of the chart is toric.
\end{note}
\subsection{Algebraic torus actions, geometricity}
The key to prove \Cref{thm:algsheaf} is to construct an algebraic action of $H^1(M,\bG_m)$, and to show that it coincides with the expected geometric action by symplectomorphisms. 

One can extend the Fukaya category by pairs $(L,\xi|_{L})$, where $L$ is a Lagrangian and $\xi$ is a unitary local system on $M$. The coefficients of the $A_\infty$-products are sums over the same moduli of discs; however, each summand is twisted by a period of $\xi$. This does not extend beyond the unitary case, as for a non-unitary $\Lambda^*$-local system $\xi$, the periods are elements of $\Lambda^*$ that may have non-zero $T$-adic valuation, and Gromov compactness no longer shows the convergence of these sums. However, the key implication of \Cref{assump:monotoneexact} is that the Fukaya category is split generated by a finite collection $\{L_i\}$ of Lagrangians such that the coefficients of the $A_\infty$-maps are finite sums in $T^E\in \Lambda$. As a result, one can twist by non-unitary local systems too. By extending the Fukaya category, by pairs $(L_i,\xi|_{L_i})$, where $\xi$ is a $\Lambda^*$-local system, one obtains an action of the abstract group $H^1(M,\Lambda^*)$. If for $z\in H^1(M,\bG_m)$, $\xi_z$ denotes the corresponding local system, the action is by $(L_i,\xi|_{L_i}) \mapsto (L_i,\xi|_{L_i}\otimes\xi_z|_{L_i} )$.

First thing to note here is that, this is an action in the derived sense (or in the Morita sense to be precise). More precisely, as observed, if $\tilde L$ fails to be Bohr-Sommerfeld monotone, resp. exact, $(\tilde L,\xi|_{\tilde L})$ cannot be added to the Fukaya category. On the other hand, we have an action on the span of $\{(L_i,\xi)\}$ and these generate the Fukaya category. In other words, any other $\tilde L$ can be represented as a complex of these generators, and the action on them extends to derived Fukaya category. To make this more precise, one uses the language of derived Morita theory: let $\cF(M)$ denote the span of $\{L_i\}$ and $\widetilde{\cF}(M)$ denote the span of $\{(L_i,\xi)\}$. One has an action by auto-equivalences on the latter. Each auto-equivalence correspond to a ``graph bimodule'' (c.f. Fourier-Mukai kernels). As objects of $\widetilde{\cF}(M)$ can be represented as complexes of objects of $\cF(M)$, the restriction of this bimodule to the latter is not loss of information. Similarly, one can now act on other Lagrangians, roughly by representing them as complexes of $\cF(M)$ again. It is important to note that non monotone/exact Lagrangians are unavoidable: even the non-exact deformations of a Bohr-Sommerfeld/exact Lagrangians under global isotopies fails this property. 

Let $\fM^M|_z$ denote the bimodule corresponding to $z\in H^1(M,\bG_m)$. That this is an action can be stated as $\fM^M|_{z=1}$ is the diagonal bimodule, and $\fM^M|_{z_2}\otimes_{\cF(M)} \fM^M|_{z_1}\simeq \fM^M|_{z_1z_2}$. Recall that the convolution over $\cF(M)$ is what corresponds to composition in the language of bimodules. 

We show that this action is geometric, i.e. $\fM^M|_z$ acts the same way as $\phi_z$ for all $z\in H^1(M,\bG_m)$. Since we use actions on the derived category, it is convenient to state this using the language of Yoneda modules. Let $h_{\tilde L}$ denote the right Yoneda module of $\tilde L$. Then we have
\begin{thm}[Main abstract theorem]\label{thm:mainabstract}
Given tautologically unobstructed compact Lagrangian brane $\tilde{L}$ 
\begin{equation}
		h_{\tilde{L}}\otimes_{\cF(M)} \fM^M|_z\simeq h_{\phi_z(\tilde{L})}
\end{equation}	
for all $z\in H^1(M,\bG_m)$. 
\end{thm}
When $z$ is close to identity, we use Fukaya's trick to prove this. To conclude \Cref{thm:mainabstract} for general $z$, we use the action property.

To establish algebraicity, we need to construct $\{\fM^M|_z:z\in H^1(M,\bG_m)\}$ as an algebraic family of bimodules. Heuristically, one constructs an algebraic action by acting on Bohr-Sommerfeld monotone/exact Lagrangians via ``the universal rank-$1$ local system''. To make this more precise, we use the notion of algebraic family of bimodules borrowed from \cite{flux}, and construct one $\fM^M$. The formulae defining $\fM^M$ are in the ring of algebraic functions of $H^1(M,\bG_m)$, and algebraic by definition.

We would like to note that extension of \Cref{thm:mainabstract} (and hence \Cref{thm:algsheaf}) to the wrapped case meets difficulties. We use the fact that the constructed family of bimodules is proper, which fails in the wrapped case. To overcome this difficulty, we endow $M$ with the structure of a Lefschetz fibration. Let $\cW(M,\fstop)$ denote the corresponding Fukaya-Seidel category. The statements such as \Cref{thm:mainabstract} hold over $\cW(M,\fstop)$, and we show that this property descends to $\cW(M)$ under stop removal functor $\cW(M,\fstop)\to\cW(M)$. We overcome a similar technical difficulty in showing group-action property in this way as well.

\subsection{Future work}
In work in progress, we aim to extend this construction to Lagrangian isotopies (that does not necessarily come from global symplectic isotopies) by combining what is established in this paper and the Viterbo restriction. 

More precisely, assume $M$ is Weinstein and $L$ is exact. One can construct an algebraic family of right modules $h_L^{alg}$ parametrized by $H^1(L,\bG_m)$ in a very similar way. This is essentially obtained by deforming $L$ by the universal non-unitary local system. As remarked above, this leads to an injective map from $H^1(L,\bG_m)$ to the moduli of objects; however, the argument above does not prove that this deformation is realized by Lagrangian isotopies unless $H^1(M,\Lambda)\to H^1(L,\Lambda)$ is surjective. In other words, the family is not a priori ``geometric''. 
We aim to prove the geometricity of the family over Weinstein neighborhoods that are also Liouville subdomains, and use this to conclude geometricity over $M$. Note the following heuristic picture from mirror symmetry: if $X$ is a variety, $x\in X$ is a smooth point, then often there is not a vector field moving $x$ in every tangent direction. However, there are always Zariski local vector fields defined near $x$ and moving it in every direction. Restriction to sufficiently large Weinstein neighborhoods to make use of the continuous symmetries can be seen as the mirror analogue of this. 

We hope to extend this to the compact case by using other restriction functors, as in \cite{leehmsforopen}, or as will appear in \cite{localcats}. Note on the other hand in the compact case, there are restrictions on the size of Weinstein neighborhoods, and often on the flux of Lagrangian isotopies of $L$.
\subsection{Outline of the paper}
In Section \ref{sec:background}, we give background on Fukaya categories and related homological algebra. In particular, we recall notions such as families of modules and bimodules. In Section \ref{sec:algfamilydefclass}, we first restrict to compact $M$, define the algebraic family $\fM^M$, and show that it is essentially a group action. We also show how to extend these results to the non-proper (wrapped) case. In Section \ref{sec:algfloersheaf}, we use ``the algebraic torus action'' established in the previous section to prove \Cref{thm:mainabstract}, \Cref{thm:algsheaf}, and their corollaries. In \Cref{sec:stabilizer}, we prove \Cref{thm:stabilizer}, and use this to deduce corollaries about flux. In \Cref{sec:mirror}, we prove \Cref{thm:rationalmirror} by essentially constructing affine torus chart in the moduli of objects of the derived category that is geometrically realized by symplectic isotopies on the A-side. 
\subsection*{Acknowledgments}
We would like to thank Sheel Ganatra for telling us about ``smooth categorical compactifications'' of wrapped Fukaya categories, 
as well as for reference suggestions and useful conversations. We would also like to thank to Mohammed Abouzaid, Ivan Smith, Conan Leung, J\'{a}nos Koll\'{a}r, and John Sheridan for useful conversations. 

\section{Background on Fukaya categories and related homological algebra}\label{sec:background}
\subsection{Reminders on $A_\infty$-categories and modules}
In this section, we collect some generalities on $A_\infty$-categories and modules. Let $\cB$ be an $A_\infty$-category over $\Lambda=\bC((T^\bR))$. Given $L\in ob(\cB)$, we denote corresponding right, resp. left Yoneda module by $h_L$, resp. $h^L$. These are defined by $h_L=\cB(\cdot,L)$, resp. $h^L=\cB(L,\cdot)$, with the module structures induced by the $A_\infty$-structure of $\cB$. Given $L,L'$, one can define the corresponding Yoneda bimodule by $h^L\otimes_\Lambda h_{L'}$ (we will often omit the subscript of tensor products, when it is the base field). The underlying graded vector space for the Yoneda bimodule is given by 
\begin{equation}
	(L_0',L_0)\mapsto h^L(L_0)\otimes h_{L'}(L_0')=\cB(L,L_0)\otimes \cB(L_0',L')
\end{equation}
and the structure maps are given by 
\begin{equation}
	(x_1,\dots,x_k|m\otimes m'|x_l',\dots x_1')\mapsto
	\begin{cases}
	\pm\mu^1(m)\otimes m'\pm m\otimes \mu^1(m'),&\text{if }k=l=0\\
	\pm \mu(x_1,\dots,x_k,m)\otimes m',&\text{if }l=0\\
	\pm m\otimes \mu(m',x_l',\dots x_1'),&\text{if }k=0\\
	0,&\text{otherwise}
	\end{cases}
\end{equation}
We will often restrict our attention to categories and modules satisfying various properties. Recall:
\begin{defn}
An $A_\infty$-category $\cB$ over $\Lambda$ is called \textbf{proper} if $H^*(\cB(L,L'))$ is finite dimensional for all $L,L'$. Similarly, a left/right module, resp. bimodule $\fN$ is called \textbf{proper} if $H^*(\fN(L))$, resp. $H^*(\fN(L,L'))$ is finite dimensional. A left/right/bi-module is called \textbf{perfect} if it is quasi-isomorphic to a direct summand of an iterated cone of Yoneda modules (Yoneda bimodules in the case of bimodules). This condition is equivalent to the module being a compact object in the dg-category of modules (see \cite{kellerdg}). A category is called \textbf{(homologically) smooth} if the diagonal bimodule is perfect. See also \cite{koso}. 
\end{defn}
Given right module $\fN$ and left module $\fN'$, one can define the convolution $\fN\otimes_\cB  \fN'$ over $\cB$. Given by a bar construction, this is a chain complex over $\Lambda$. The underlying graded vector space is given by 
\begin{equation}
	\fN\otimes_{\cB}\fN'=\bigoplus \fN(L_p)\otimes \cB(L_{p-1},L_p) \otimes \dots \otimes\cB(L_0,L_1)\otimes \fN'(L_0)
\end{equation}
where the sum varies over all objects of $\cB$ and all $p\geq 0$. Differential is defined by applying $\mu_\fN$, $\mu_{\fN'}$ or $\mu_\cB$ to successive terms. For more details, see \cite{generation}. 

If $\fN$ is a bimodule, the convolution carries the structure of a left module. It assigns the graded vector space
\begin{equation}
	(\fN\otimes_{\cB}\fN')(L)=\bigoplus \fN(L_p,L)\otimes \cB(L_{p-1},L_p) \otimes \dots \otimes\cB(L_0,L_1)\otimes \fN'(L_0)
\end{equation}
to an object $L$, its differential is the same as above, and its higher structure maps are defined similarly. Similarly, when $\fN'$ is a bimodule $\fN\otimes_{\cB}\fN'$ is a right module, and when $\fN$ and $\fN'$ are bimodules, $\fN\otimes_{\cB}\fN'$ is a bimodule.
\begin{note}\label{note:nconvdiagisn}
Let $\cB$ denote the diagonal bimodule. Then, for any right module $\fN$, it is a standard fact that $\fN\otimes_{\cB} \cB\simeq \fN$. The $0^{th}$-map of the quasi-isomorphism $\fN\otimes_{\cB} \cB\to \fN$ is given by 
\begin{equation}\label{eq:nconvdiagisn}
(n\otimes x_p \dots x_1\otimes x )\mapsto 	\pm \mu_\fN(n|x_p, \dots ,x_1, x )
\end{equation}	
and the higher maps are similar. Analogous statements hold for left modules and bimodules. 
\end{note}
The following is easy to prove:
\begin{lem}\label{lem:hlconvhlequalshom}
$h_{L'}\otimes_\cB h^L\simeq \cB(L,L')$ as chain complexes. More generally, for any right module $\fN'$, $\fN'\otimes_{\cB} h^L  \simeq \fN'(L)$. Similarly, for any left module $\fN$, $h_{L'}\otimes_{\cB}\fN\simeq \fN(L')$ and for any bimodule $\fM$, $h_{L'}\otimes_{\cB}\fM\otimes_{\cB} h^L\simeq \fM(L,L')$.
\end{lem}
\begin{proof}
One can define the map $f:h_{L'}\otimes_\cB h^L\to \cB(L,L')$	that sends $x'\otimes\dots\otimes x$ to $\pm\mu_\cB(x',\dots,x)$. For simplicity, assume $\cB$ has one object, i.e. it is an $A_\infty$-algebra. Let $C$ denote the cone of $f$. $C$ is naturally filtered by length, and the $E_1$-page of the corresponding spectral sequence is equivalent to the bar complex of the ordinary graded algebra $H^*(\cB)$ (spread into degrees). Therefore, its $E_2$-page vanishes, and $C$ is acyclic. In other words, $f$ is a quasi-isomorphism.

The proof is the same when one of the Yoneda modules is replaced by right/left modules. The bimodule version can be proven by applying right/left modules in order. More precisely, if $\fM$ is a bimodule, $\fM\otimes_{\cB} h^L\simeq \fM(L,\cdot)$ as left modules. The quasi-isomorphism is the same as above at the $0^{th}$-level. The $q^{th}$ map of the module quasi-isomorphism is given by
\begin{equation}
 (x_1',\dots ,x_q'|m\otimes x_p \dots x_1\otimes x )\mapsto \pm\mu_\fM^{q|1|p+1} (x_1',\dots ,x_q'|m| x_p \dots x_1, x )
\end{equation}
Applying $h_{L'}\otimes_{\cB}(\cdot)$ to both sides, we obtain 
\begin{equation}
h_{L'}\otimes_{\cB}\fM\otimes_{\cB} h^L\simeq h_{L'}\otimes_{\cB}\fM(L,\cdot)\simeq \fM(L,L')
\end{equation}
\end{proof}
This immediately implies
\begin{cor}\label{cor:properconvperfisperf}
Assume $\fN$ is a proper right module (i.e. $H^*(\fN(L))$ is finite dimensional), and $\fM$ is a perfect bimodule (i.e. it is quasi-isomorphic to a direct summand of an iterated cone of Yoneda bimodules). Then, $\fN\otimes_{\cB}\fM$ is perfect. 
\end{cor}
\begin{proof}
It suffices to show this for $\fM=h^L\otimes_\Lambda h_{L'}$, a Yoneda bimodule. In this case, 
\begin{equation}
\fN\otimes_{\cB} (h^L\otimes_\Lambda h_{L'})\simeq (\fN\otimes_{\cB} h^L)\otimes_\Lambda h_{L'}	\simeq \fN(L)\otimes_\Lambda h_{L'}
\end{equation}
by Lemma \ref{lem:hlconvhlequalshom}. The complex $\fN(L)$ has finite dimensional cohomology, as $\fN$ is proper, and this completes the proof.
\end{proof}
\begin{rk}It is easy to see other variants of this corollary (i.e. for left modules etc.) hold.
\end{rk}
Recall that one can think of bimodules over an $A_\infty$-category as generalizations of functors (see for instance, \cite{kellerdg}). In particular, for any endo-functor $\Phi:\cB\to \cB$, there are two bimodules, ${_\Phi\cB}=\cB(\cdot, \Phi(\cdot))$, and $\cB_\Phi=\cB(\Phi(\cdot), \cdot)$. The bimodule ${_\Phi\cB}$ is defined by $(L',L)\mapsto \cB(L', \Phi(L))$, and its structure maps are given by 
\begin{equation}\label{eq:subphibstructures}
(x_1,\dots, x_k|m|x_l',\dots,x_1')\mapsto \sum \pm \mu_\cB (\Phi^{i_1}(x_1,\dots,x_{i_1}), \Phi^{i_2}(x_{i_1+1},\dots, x_{i_1+i_2}),\dots, \Phi^{i_j}(\dots,x_k), m,x_l',\dots,x_1'  )
\end{equation}
The sum varies over all $(i_1,\dots, i_j)$ such that $i_1+\dots+ i_j=k$. \eqref{eq:subphibstructures} explains why we put $\Phi$ as a left subscript to ${_\Phi\cB}$: to obtain it, one twists the left multiplication by $\Phi$. The bimodule $\cB_\Phi$ is defined similarly. 
\begin{lem}\label{lem:composedgraphbimodules}
${_\Psi\cB}\otimes_{\cB} {_\Phi\cB}\simeq {_{\Phi\circ\Psi}\cB}$.
\end{lem}
\begin{rk}
In case these bimodules are counter-intuitive to the reader, we present the following simplified picture: let $\cB$ be a linear category, with finitely many objects and $A=\bigoplus_{L,L'} \cB(L,L')$ be its total algebra. Let $e_L\in A$ denote the idempotent element corresponding to the unit of $L$. Then, the right Yoneda module of $L$ is given by $e_L A$. Assume $\Phi$ acts on $\cB$ strictly (i.e. without higher components). It induces an action on $A$ that sends $e_L$ to $e_{\Phi(L)}$. Let ${_\Phi A}$, denote the $A$-bimodule, with the same underlying vector space as $A$ and same right multiplication, but the left multiplication is given by $x.m=\Phi(m)a\in A$. $A_\Phi$ is defined similarly. It is easy to check that $e_LA\otimes_A {_\Phi A}\cong e_{\Phi(L)}A $, via an isomorphism that sends $e_La\otimes a'$ to $\Phi(e_La)a'=e_{\Phi(L)}\Phi(a)a'$. In other words, the action of ${_\Phi A}$ on right Yoneda modules via convolution coincide with the action of $\Phi$ (we will not prove $A_\infty$-version of this claim though). Lemma \ref{lem:composedgraphbimodules} becomes ${_\Psi A}\otimes_{A}  {_\Phi A}\cong {_{\Phi\circ\Psi}A}$, which is easy to check by hand. An isomorphism is given by $a\otimes a'\mapsto \Phi(a)a'$.
\end{rk}
\begin{proof}[Proof of Lemma \ref{lem:composedgraphbimodules}]
We define a bimodule homomorphism ${_\Psi\cB}\otimes_{\cB} {_\Phi\cB}\to {_{\Phi\circ\Psi}\cB}$. The $(0|1|l)$ component of the homomorphism is given by 
\begin{equation}
(m\otimes x_1''\otimes\dots x_p''\otimes m'|  x_l',\dots, x_1')\mapsto \sum\pm \mu_\cB(\Phi^{i_1}(m,\dots),\dots ,\Phi^{i_j}(\dots, x_p''),m',\dots x_1')  	
\end{equation}
To write its $(k|1|l)$ component, one first applies $\Psi$ to $(x_1,\dots x_k)$. More precisely, given
\begin{equation}\label{eq:beforepsiphi}
	(x_1,\dots x_k|m\otimes x_1''\otimes\dots x_p''\otimes m'|  x_l',\dots, x_1')
\end{equation}
one has $(\Psi^{i_1}(x_1,\dots),\dots,\Psi^{i_j}( \dots,x_k)|m\otimes x_1''\otimes\dots x_p''\otimes m'|  x_l',\dots, x_1')$ for each $i_1+\dots +i_j=k$. Then one applies $\Phi$, to the terms on the right of $m'$. For instance, if $\Phi$ and $\Psi$ has no higher maps, one obtains 
\begin{equation}
	\pm \mu_\cB(\Phi^1(\Psi^1(x_1)),\dots ,\Phi^1(\Psi^1(x_k)),\Phi^1(m),\Phi^1( x_1''),\dots ,\Phi^1(x_p''), m',  x_l',\dots, x_1' )
\end{equation}
It is easy to check this defines a morphism 
\begin{equation}\label{eq:psiphitocomposition}
{_\Psi\cB}\otimes_{\cB} {_\Phi\cB}\to {_{\Phi\circ\Psi}\cB}
\end{equation} 
of bimodules. To show \eqref{eq:psiphitocomposition} is a quasi-isomorphism observe that as a chain complex (i.e. without higher structure maps) the cone of \eqref{eq:psiphitocomposition} does not depend on the action of $\Psi$ on morphisms. More precisely, the complex one obtains by plugging $(L,L')$ to $cone(\eqref{eq:psiphitocomposition})$ is the same as the complex one obtains by plugging $(L,\Psi(L'))$ into \begin{equation}\label{eq:onephitocomposition}
cone({_{1_\cB}\cB}\otimes_{\cB} {_\Phi\cB}\to {_{\Phi\circ 1_\cB}\cB})	
\end{equation}
Therefore, without loss of generality one can assume $\Psi=1_\cB$. Clearly, ${_{1_\cB}\cB}\simeq \cB$, the diagonal bimodule, and \eqref{eq:onephitocomposition} is the same as the standard quasi-isomorphism $\cB\otimes_{\cB}\fN\to \fN$, for $\fN={_\Phi\cB}$ (see Note \ref{note:nconvdiagisn} for a version of this). 
%
\end{proof}
Note that we will later refer to the explicit quasi-isomorphism given in the proof of Lemma \ref{lem:composedgraphbimodules}.
\begin{rk}
One similarly has $\cB_\Phi\otimes_{\cB} \cB_\Psi\simeq \cB_{\Phi\circ\Psi}$, although we do not need this. In the heuristics above, this becomes $A_\Phi\otimes_{A} A_\Psi\cong A_{\Phi\circ\Psi}$, and an isomorphism is given by $a\otimes a'\mapsto a\Phi(a')$. 
\end{rk}

Under the quasi-isomorphism given in Lemma \ref{lem:hlconvhlequalshom}, the composition map
\begin{equation}
	hom(L',L'')\otimes hom(L,L')\to hom(L,L'')
\end{equation}
admits a simple description. Namely, we would like to construct a chain map 
\begin{equation}\label{eq:hlcomps}
	h_{L''}\otimes_{\cB} h^{L'}	\otimes_\Lambda h_{L'}\otimes_{\cB} h^{L}\to h_{L''}\otimes_{\cB} h^{L}
\end{equation}
For this, observe there exists a map of bimodules from $h^{L'}	\otimes_\Lambda h_{L'}$ to the diagonal bimodule $\cB$ given by 
\begin{equation}\label{eq:midcontraction}
	(x_1,\dots x_k|m\otimes m'|x_1',\dots, x_l')\to\pm \mu_\cB(x_1,\dots x_k,m,m',x_1',\dots, x_l')	
\end{equation}
Applying this to the middle part of (\ref{eq:hlcomps}), we obtain the composition map. More precisely,
\begin{lem}\label{lem:compositionbymiddlecontraction}
Under the quasi-isomorphisms given by Lemma \ref{lem:hlconvhlequalshom}, the composition map is homotopic to the composition
\begin{equation}
	h_{L''}\otimes_{\cB} h^{L'}	\otimes_\Lambda h_{L'}\otimes_{\cB} h^{L}\to h_{L''}\otimes_{\cB} \cB \otimes_{\cB} h^{L} \simeq h_{L''}\otimes_{\cB} h^{L}	
\end{equation}
where the first map is given by contracting $h^{L'}	\otimes_\Lambda h_{L'}$ under the map $h^{L'}	\otimes_\Lambda h_{L'}\to\cB$ and the second one is the standard quasi-isomorphism of a module with its convolution with the diagonal. More precisely, there exists a homotopy commutative diagram
\begin{equation}\label{eq:htpcommutativeinproductlemma}
\xymatrix{h_{L''}\otimes_{\cB} h^{L'}	\otimes_\Lambda h_{L'}\otimes_{\cB} h^{L}\ar[r]\ar[d] & h_{L''}\otimes_{\cB} \cB \otimes_{\cB} h^{L} \ar[r]& h_{L''}\otimes_{\cB} h^{L}\ar[d] \\ \cB(L',L'')\otimes_\Lambda \cB(L,L')\ar[r(1.55)]& & \cB(L,L'')   }	
\end{equation}
\end{lem}
\begin{rk}
Note that to contract by $h^{L'}\otimes_\Lambda h_{L'}\to\cB$ as in the first row of (\ref{eq:htpcommutativeinproductlemma}), one still uses higher components of this bimodule homomorphism. Also notice, there is an ambiguity in the map
\begin{equation}\label{eq:tripleproduct}
h_{L''}\otimes_{\cB} \cB \otimes_{\cB} h^{L} \xrightarrow{\simeq} h_{L''}\otimes_{\cB} h^{L}		
\end{equation}
Clearly, the chain complexes $(h_{L''}\otimes_{\cB} \cB) \otimes_{\cB} h^{L}$ and $h_{L''}\otimes_{\cB} (\cB \otimes_{\cB} h^{L})$ are identical, but the map (\ref{eq:tripleproduct}) is obtained by contracting via one of the quasi-isomorphisms $h_{L''}\otimes_{\cB} \cB\to h_{L''}$ and $\cB \otimes_{\cB} h^{L}\to h^L$. Two induced maps are actually homotopic, and the homotopy is given by the map from $h_{L''}\otimes_{\cB} \cB \otimes_{\cB} h^{L}$ to $h_{L''}\otimes_{\cB}  h^{L}$ that forgets the middle $\cB$ component. More precisely, the domain is spanned by strings of the form $x\otimes x_1\dots x_k\otimes b\otimes x_l'\otimes\dots x_1'\otimes x'$, where $b$ belongs to middle component of triple convolution. The domain is spanned by strings $x\otimes \dots \otimes x'$, and the homotopy between maps above is given by sending a string to the same string (with a Koszul sign), only forgetting which element belongs to middle component. During the proof, we will use the second one of these maps (\ref{eq:tripleproduct}). 
\end{rk}
\begin{proof}[Proof of Lemma \ref{lem:compositionbymiddlecontraction}]
We prove this lemma in two steps. First, consider the diagram
\begin{equation}
	\xymatrix{h_{L''}\otimes_{\cB} h^{L'}	\otimes_\Lambda h_{L'}\otimes_{\cB} h^{L}\ar[r]\ar[d] & h_{L''}\otimes_{\cB} \cB \otimes_{\cB} h^{L} \ar[r]& h_{L''}\otimes_{\cB} h^{L}\ar@{=}[d] \\ h_{L''}\otimes_{\cB} h^{L'}\otimes_\Lambda \cB(L,L')\ar[r(1.5)]& & h_{L''}\otimes_{\cB} h^{L}  }	
\end{equation}
The first line of this diagram is as described above and the left vertical arrow is obtained by applying the map in Lemma \ref{lem:hlconvhlequalshom} to $h_{L'}\otimes_{\cB} h^{L}$. The bottom horizontal arrow is obtained by applying $h_{L''}\otimes_{\cB}(\cdot)$ to the natural left module homomorphism $h^{L'}\otimes_\Lambda \cB(L,L')\to h^L$. 

We claim this diagram is homotopy commutative: the composition through upper right corner sends the string
\begin{equation}\label{eq:stringfourconv}
	x''\otimes\dots\otimes x_2'\otimes_\Lambda x_1'\otimes \dots \otimes x
\end{equation}
to a signed sum of $x''\otimes\dots \otimes \mu_\B(\dots,\mu_\cB(\dots, x_2',x_1',\dots),\dots ,x)$. Here, we apply $\mu_\cB$ twice, the first one contains $x_1',x_2'$, but not $x$, the second one contains $x$ but not $x''$ (recall we are using the map $h_{L''}\otimes_{\cB} \cB\otimes_{\cB} h^L\to h_{L''}\otimes_{\cB} h^L$ that is induced by $\cB\otimes_{\cB} h^L\to h^L$). Similarly, the composition through bottom left arrow sends the string (\ref{eq:stringfourconv}) to a signed sum of $x''\otimes \dots \otimes \mu_\cB(\dots,x_2',\mu_\cB(x_1',\dots, x) )$ (where, in particular, $x''$ is outside $\mu_\cB$). An homotopy is given by the map sending the string (\ref{eq:stringfourconv}) to a signed sum of $x''\otimes \dots \otimes \mu_\cB(\dots, x_2',x_1',\dots ,x)$. We apply the $A_\infty$-map to all sub-strings containing $x$, $x_1'$ and $x_2'$, but not $x''$. That this is a homotopy follows from $A_\infty$-equations for the substrings containing $x$ and $x_2'$, but not containing $x''$. 

Similarly, consider the diagram 
\begin{equation}
	\xymatrix{h_{L''}\otimes_{\cB} h^{L'}\otimes_\Lambda \cB(L,L')\ar[r]\ar[d]& h_{L''}\otimes_{\cB} h^{L}\ar[d] \\ \cB(L',L'')\otimes_\Lambda \cB(L,L')\ar[r]& \cB(L,L'') }
\end{equation}
The upper horizontal arrow is as above, the vertical arrows are obtained by application of the quasi-isomorphism of Lemma \ref{lem:hlconvhlequalshom}, and the bottom arrow is composition. This diagram is also homotopy commutative: the composition through upper right corner sends the string 
\begin{equation}\label{eq:stringtriple}
	x''\otimes\dots\otimes x'\otimes b
\end{equation}
to a signed sum of $\mu_\cB(x'',\dots \mu_\cB(\dots, x',b) )$ (i.e. the first product is applied to a sub-string containing $b,x'$, but not $x''$). Similarly, the composition through bottom left corner sends (\ref{eq:stringtriple}) to $\pm \mu_\cB(\mu_\cB(x'',\dots, x'),b)$. An homotopy is given by a map sending (\ref{eq:stringtriple}) to $\pm\mu_\cB(x'',\dots, x',b)$. 
It is easy to check this defines an homotopy. 

Combining these diagrams gives us the homotopy commutativity of (\ref{eq:htpcommutativeinproductlemma}). 
\end{proof}

\subsection{Reminders on Fukaya categories of monotone symplectic manifolds}
In this section, we recall basics of Fukaya categories on compact manifolds. Throughout the paper, $\Lambda=\bC((T^\bR))$ denote the Novikov field with complex coefficients and real exponents.  

Let $(M,\omega_M)$ be a compact monotone or negatively monotone symplectic manifold. Let $\{L_i\}$ be a collection of oriented, monotone, tautologically unobstructed Lagrangians. Assume $L_i$ are equipped with grading and spin structures. Without loss of generality, we can assume $L_i$ are pairwise transverse. One can define Fukaya category with objects $\{L_i\}$ by counting marked holomorphic discs. Most of our constructions are model independent, but for the compact case we prefer to use the version of Fukaya category presented in \cite{seidelbook} and \cite{sheridanfano}. 

In \cite{seidelbook}, one makes consistent choices of Floer and perturbation data varying over the moduli of discs with marked points. For instance, given $(L_i,L_j)$, one chooses a (possibly time dependent) pair $(H,J)$ of an Hamiltonian and an almost complex structure. The choice is made so that $\phi_H^1(L_i)\pitchfork L_j$, and one defines 
\begin{equation}
	hom(L_i,L_j)=CF(L_i,L_j)\cong \Lambda\langle \phi_H^1(L_i)\cap L_j \rangle 
\end{equation}
to be the vector space generated by time-$1$ chords from $L_i$ to $L_j$. Then $hom(L_i,L_j)$ is equipped with the standard Floer differential, i.e. 
\begin{equation}\label{eq:fukdiff}
\mu^1(x)=\sum \pm T^{E(u)}y
\end{equation}
where $x,y$ are generators of $CF(L_i,L_j)$. The sum varies over $y$ and the elements $u$ of $0$-dimensional component of moduli of pseudo-holomorphic strips (up to translation) asymptotic to $x$ and $y$ with boundary components on $L_i$ and $L_j$. $T$ is the Novikov parameter as above, and $E(u)$ denote the topological energy of the strips. One defines the topological energy of a strip $S=\bR_s\times [0,1]_t$ equipped with Floer data $(H,J)$ by 
\begin{equation}
E(u)=E^{top}(u)=\int_S u^* \omega_M-d(H.dt)
\end{equation}
More generally, given $(L_{i_0},L_{i_1},\dots , L_{i_q})$, and generators $x_j\in CF(L_{i_{j-1}}, L_{i_j})$, define 
\begin{equation}\label{eq:fukstructure}
	\mu^q(x_q,\dots, x_1)=\sum \pm T^{E(u)}.y	
\end{equation}
as $y$ runs over the generators of $CF(L_{i_1}, L_{i_q})$, $u$ runs over rigid marked pseudo-holomorphic discs with boundary components on various $L_{i_j}$ and asymptotic to $x_1,\dots, x_q,y$ near markings. Here, $E(u)$ still denotes the topological energy, defined similarly (see \cite{abousei} for instance). Thanks to spin structures, one can orient the relevant moduli of pseudo-holomorphic discs, and this determines the signs in (\ref{eq:fukdiff}) and (\ref{eq:fukstructure}). We will often omit the superscript $q$. We choose the perturbation data such that the topological energy is larger than the geometric energy. By standard compactness and gluing arguments, the operations $\mu^q$ satisfy $A_\infty$-relations; hence, we obtain an $A_\infty$ category.

The difference of geometric and topological energies is given by integrating curvature, see \cite[(5.12)]{seidellef1}, for instance. For later purposes, we also assume that the Floer data is chosen such that this difference is uniformly bounded over discs with fixed boundary conditions, and fixed inputs and outputs. 

Throughout the paper, we assume the Floer data for a pair $(L_i,L_j)$ such that $i\neq j$ has vanishing Hamiltonian term. This will simplify the application of Fukaya's trick.

Let $L\subset M$ be another tautologically unobstructed Lagrangian, and assume it is equipped with a brane structure (i.e. grading and spin structure). Without loss of generality assume it is transverse to all $L_i$. Extend the Fukaya category by adding $L$ and making further consistent choices of Floer and perturbation data. As above, we assume the Floer data for the pairs $(L,L_i)$ and $(L_i,L)$ have vanishing Hamiltonian terms. Then, one can define a left, resp. right $A_\infty$-module $h^L$, resp. $h_L$ over the Fukaya category spanned by $\{L_i\}$ by restricting the Yoneda modules. Concretely, 
\begin{equation}
	h^L(L_i):=CF(L,L_i)\text{ and }h_L(L_i):=CF(L_i,L)
\end{equation}
and the differential and higher structure maps of are induced by the $A_\infty$-operations on the Fukaya category with objects $\{L_i\}_i\cup\{L\}$.

More generally, if $L$ is a Lagrangian as above and $\xi_L$ is a $U_\Lambda$-local system on $L$, one can add the pair $(L,\xi_L)$ to the Fukaya category and define its left/right Yoneda modules. Recall that if $(L,\xi_{L})$ and $(L',\xi_{L'})$ are two such pairs, one defines the Floer chains $CF((L,\xi_{L}),(L',\xi_{L'}))=CF(L,L')$. To define the differential, one fixes a basepoint on $L$, resp. $L'$ as well as homotopy classes of paths on $L$, resp. $L'$ from the basepoint to the generators of $CF((L,\xi_{L}),(L',\xi_{L'}))$. Given a Floer strip $u$, one can define $[\partial_L u]\in H_1(L)$, resp. $[\partial_{L'} u]\in H_1(L')$, by concatenating the path to the part of the boundary of $u$ lying on $L$, resp. $L'$, and inverse of the path from the output to the basepoint. The Floer differential is defined by counting such strips weighted by $T^{E(u)}\xi_L^{[-\partial_L u]}\xi_{L'}^{[\partial_{L'} u]}$ (instead of $T^{E(u)}$). Here $\xi_L^{[-\partial_L u]}$ denotes the holonomy of unitary local system $\xi_L\in H^1(L,U_\Lambda)$ at $[\partial_L u]$, and similarly with $\xi_{L'}^{[-\partial_{L'} u]}$. The other $A_\infty$-maps are modified similarly.
\begin{note}\label{note:globallocal1}
Throughout the paper, we only work with globally defined local systems on $M$. Hence, we will choose a basepoint on $M$ and paths from this point to the generators (not necessarily on $L$ and $L'$). The definitions will be modified accordingly, without changing the quasi-isomorphism class of Floer complex and Yoneda modules. 
\end{note}
In the following, we will use the phrase \textbf{Lagrangian brane} to mean a tautologically unobstructed Lagrangian with fixed grading and spin structure, and we will assume it is equipped with a unitary local system, unless stated otherwise. However, we will often omit the local system $\xi_L$ from the notation. If we further twist $L=(L,\xi_L)$ by another local system $\xi'$, we denote the new pair by $(L,\xi')$ rather than $(L,\xi_L\otimes\xi')$ to avoid further complicating the notation. 

We will consider Fukaya category with a restricted set of Lagrangians. However, it will have the following property:
\begin{defn}
We say $\{L_i\}$ \textbf{split generate $L$}, if $L$ is quasi-isomorphic to a direct summand of a twisted complex in $\{L_i\}$ in the Fukaya category with objects $\{L_i\}\cup\{L\}$. We say $\{L_i\}$ \textbf{split generate the Fukaya category} if this holds for any tautologically unobstructed Lagrangian brane $L$.
\end{defn}
\begin{notation}
From now on assume $\{L_i\}$ is a finite set of Lagrangian branes that split generate the Fukaya category and let $\cF(M)$ denote the Fukaya category spanned by	$\{L_i\}$. We assume $L_i$ is equipped with the trivial local system for any $i$.
\end{notation}
One can ensure split generation by the following:
\begin{thm}\cite{generation}
Let $\cF$ denote the span of $\{L_i\}$. If the open-closed map $HH_*(\cF,\cF)\to QH^*(M)$ hits the identity, then $\{L_i\}$ split generate the Fukaya category.
\end{thm}
If there exists a collection of Lagrangians satisfying the condition of this theorem, then $M$ is called \textbf{non-degenerate}. By \cite{sheelthesis}, this implies homological smoothness of $\cF$.

Assumption \ref{assump:monotoneexact} also imposes that there exists a set of generators that satisfy a condition called \textbf{Bohr-Sommerfeld monotonicity}. We borrow this notion from \cite[Remark 4.1.4]{wehrheimwoodwardquilted1}. To define this condition, one needs to assume $\omega_M$ is integral. For simplicity assume $[\omega_M]=c_1(M)$ ($M$ is already assumed to be either monotone or negatively monotone). Fix a pre-quantum bundle $(\cL,\nabla)$, i.e. a complex line bundle $\cL$ with a connection $\nabla$ that has curvature $-2\pi i\omega_M$. A Lagrangian $L$ is \textbf{Bohr-Sommerfeld} if the flat line bundle $(\cL,\nabla)|_L$ has trivial holonomy (it is called \textbf{rational} if the holonomy group is finite). To define Bohr-Sommerfeld monotone, fix an isomorphism of line bundles $\cL\cong \cK^{-1}$, where $\cK$ is the canonical bundle. Then, $L$ is \textbf{Bohr-Sommerfeld monotone} if it is Bohr-Sommerfeld and the natural Maslov section of $\cK^{-1}|_L$ is homotopic to a flat section of $\cL|_L$ under this isomorphism. For more details, see \cite{wehrheimwoodwardquilted1}. The crucial implication for us is the following:
\begin{lem}\label{lem:finitediscs}\cite{wehrheimwoodwardquilted1}
If $\{L_i\}$ are Bohr-Sommerfeld monotone (with respect to same data), then there exists at most finitely many pseudo-holomorphic rigid discs with fixed boundary conditions and asymptotic conditions.	
\end{lem}
Hence, the coefficients of the $A_\infty$-products are finite. 
\begin{rk}
When $[\omega_M]$ is not equal but proportional to $c_1(M)$, one needs to fix an isomorphism $\cL^{\otimes l}\cong (\cK^{-1})^{\otimes k}$, and the Bohr-Sommerfeld monotonicity condition becomes the matching of the powers of above-mentioned sections of $\cK^{-1}|_L$ and $\cL|_L$. In \cite{wehrheimwoodwardquilted1}, the authors consider only the monotone case, and assume $k,l>0$. On the other hand, the proofs of statements implying finiteness (such as \cite[Lemma 4.1.5, Remark 4.2.2]{wehrheimwoodwardquilted1}) of the counts go through, even if one of $k$ and $l$ is negative. Hence, Lemma \ref{lem:finitediscs} is still true in the negatively monotone case.
\end{rk}
\begin{rk}
This notion can easily be generalized to include rational symplectic forms and (a finite set of) rational Lagrangians. The simplest way is to replace $\omega_M$ by $k\omega_M, k\in \bZ, k\gg 0$, to ensure it is integral, and fix a pre-quantum bundle for $k\omega_M$. If all $L_i$ are rational with respect to this bundle, one can make them Bohr-Sommerfeld by replacing $\omega_M$ by a multiple of itself again (and replacing the bundle by its corresponding power). One can still define Bohr-Sommerfeld monotonicity as in \cite{wehrheimwoodwardquilted1} and Lemma \ref{lem:finitediscs} holds.
\end{rk}
From now on, when considering compact $M$, we will assume $\cF(M)$ is spanned by a finite set $\{L_i\}$ of split-generating, Bohr-Sommerfeld monotone Lagrangians. In particular, the $A_\infty$-coefficients are all finite sums. 

\subsection{Reminders on wrapped Fukaya categories}
In this section, we remind basics of wrapped Fukaya categories very briefly. For more details, the reader should consult \cite{GPS1,GPS2}, whose conventions and definitions we use. Only note that we reverse (back) the composition conventions of \cite{GPS1}. For instance, we have compositions
\begin{equation}
	\mu^2:hom(L_1,L_2)\otimes hom (L_0,L_1)\to hom(L_0,L_2)
\end{equation}
rather than $\mu^2:hom(L_0,L_1)\otimes hom (L_1,L_2)\to hom(L_0,L_2)$.

Throughout this section assume $M$ is a Weinstein manifold, and let $\sigma\subset\partial_\infty M$ be a stop. The partially wrapped Fukaya category $\cW(M,\sigma)$ is an $A_\infty$-category with objects given by exact, cylindrical Lagrangian branes, and the hom-complexes have cohomology $HW(L,L')$. In \cite{GPS1}, they are constructed via localization: the authors define a directed $A_\infty$-category $\cO(M,\sigma)$, that include cofinal positive wrappings of each Lagrangian. For each positive isotopy $L\rightsquigarrow L^+$, one obtains a continuation element in $CF(L^+,L)=\cO(M,\sigma)(L^+,L)$. Then, $\cW(M,\sigma)$ is defined to be the localization at these continuation elements. When $\sigma=\emptyset$, we denote $\cW(M,\sigma)$ by $\cW(M)$. There is a natural functor $\cW(M,\sigma)\to \cW(M)$, called \textbf{stop removal}. This functor is a quotient functor when $\sigma$ is a Weinstein, or (almost) Legendrian stop. The quotient can be taken to be by a collection of discs (the linking discs of $\sigma$ when it is almost Legendrian, and the linking discs of the core of $\sigma$ when $\sigma$ is Weinstein).
\begin{exmp}
Assume $M$ is endowed with the structure of a Lefschetz fibration with Weinstein fibers (which is always possible by \cite[Theorem 1.10]{girouxpardon}), and let  $\sigma$ denote a smooth fiber of this fibration (pushed to contact boundary at infinity, so it is a stop). Then, $\cW(M,\sigma)$ is quasi-equivalent to Fukaya-Seidel category. By \cite[Corollary 1.14]{GPS2}, this category is generated by Lefschetz thimbles. 
\end{exmp}
Note that when defining the Floer differential and higher maps in the exact setting, the weights $T^{E(u)}$ of the discs are often ignored. However, as we are interested in non-exact deformations of Lagrangians, the relevant disc counts can be infinite; therefore, we assume $A_\infty$-products are defined using counts with weight $T^{E(u)}$ as above. On the other hand, if one does not add non-exact Lagrangians, $E(u)$ can be expressed in terms of well-defined actions of each generator. Therefore, $T^{E(u)}$ terms can be gotten rid of by rescaling the generators. As a result, this category is equivalent to the standard one (after base change to $\Lambda$).

The following is crucial:
\begin{fact}
When $\sigma$ is almost Legendrian or Weinstein, $\cW(M,\sigma)$ is a smooth category.	
\end{fact}
When $\sigma=\emptyset$, this follows from the non-degeneracy of $M$ and \cite{sheelthesis}. 

Observe that one can extend $\cW(M,\sigma)$ by adding pairs $(L,\xi)$, where $L$ is exact (and cylindrical at infinity) and $\xi$ is a $\Lambda^*$-local system. Without exactness, one would need to use unitary local systems.

Let $\tilde L$ be a compact, possibly non-exact, tautologically unobstructed Lagrangian. One can still define $h_{\tilde L}$ over $\cW(M,\sigma)$. Indeed, this module can be defined over $\cO(M,\sigma)$, and it is ``local'' with respect to continuation elements (i.e. these elements act invertibly on $h_{\tilde L}(L)$). As a result, it descends to $\cW(M,\sigma)$. As mentioned above, the stop removal functor $\cW(M,\sigma)\to \cW(M)$ is also a quotient by a collection of Lagrangian discs, and they can be chosen arbitrarily far in the cylindrical end. In particular, if $D$ is such a disc $h_{\tilde L}(D)=0$, and the module $h_{\tilde L}$ over $\cW(M,\sigma)$ descends to the right module over $\cW(M)$, which we also denoted by $h_{\tilde L}$.

\subsection{Families of objects}\label{subsec:familiesdeformations}
In this section, we will remind the notions of families of objects, mostly following \cite{flux} and \cite{owniteratespadic} (mostly \cite{flux} for the latter). For us, the parameter space for families will be either an affine variety over $\Lambda=\bC((T^\bR))$, or an affinoid domain over $\Lambda$. For our purposes, a deep knowledge of affinoid domains is not required. In practice, we will take the dual perspective and work with their ring of analytic functions. We remind basics of affinoid domains in Appendix \ref{sec:appendixaffinoidsemicont}. Relevant examples are the following:
\begin{exmp}\label{exmp:adicannulus}
Let $a<b\in\bR$ and consider the ring of series $\sum_{n\in\bZ} a_nz^n$, $a_n\in \Lambda$ such that $val_T(a_n)+n\nu\to\infty $ as $n\to\pm\infty$, for every $a\leq \nu\leq b$. In other words, this is the ring of functions that converge at $z\in\Lambda$ with valuation between $a$ and $b$. Denote this ring by $\Lambda\{z^\bZ \}_{[a,b]}$. We can think of $\Lambda\{z^\bZ \}_{[a,b]}$ as the ring of analytic functions of an annulus $S_{[a,b]}$. In other words, $S_{[a,b]}$ is the ``spectrum'' of this ring. The $\Lambda$-points of $S_{[a,b]}$ (i.e. the continuous algebra maps $\Lambda\{z^\bZ \}_{[a,b]}\to\Lambda$) are in correspondence with elements of $\Lambda$ with $T$-adic valuation between $a$ and $b$. One can see $S_{[a,b]}$ as an analytic subdomain of the affine variety $\bG_m:=Spec(\Lambda[z^\bZ])$.
\end{exmp}
\begin{exmp}[\cite{abouzaidicm}]\label{exmp:adicpolytope}
More generally, let $V$ be a finite rank lattice, and let $P\subset Hom(V,\bR)$ be a convex polytope defined by integral affine equations. Consider the ring of series $\sum_{v\in V} a_vz^v$, $a_v\in \Lambda$ such that $val_T(a_v)+\langle \boldsymbol{\nu},v \rangle\to \infty$ for all $\boldsymbol{\nu}\in P$. Denote this ring of series by $\Lambda\{z^V \}_P$. One can show that this ring is Noetherian (\cite[Remark 2.6]{abouzaidicm}, \cite[p.222]{boschguntherremmert}). We can think of $\Lambda\{z^V \}_P$ as the ring of analytic functions on an affinoid domain $S_P$. To describe the set of $\Lambda$-points consider the set of group homomorphisms $Hom(V,\Lambda^*)$, and the map $val_T:Hom(V,\Lambda^*)\to Hom(V,\bR)$. The $\Lambda$-points of $S_P$ are in correspondence with pre-image of $P$ under the map $val_T$. Conversely, one can define $\Lambda\{z^V \}_{P}$ to be the set of series in $z^v,v\in V$ that converge over the pre-image $val_T^{-1}(P)$. One can see $S_P$ as an analytic subdomain of the affine variety $Hom(V,\bG_m):=Spec(\Lambda[z^V])$. When $P=\{0\}$, we denote $S_P$ by $S_0$ and it can be identified with $Hom(V,U_\Lambda)$, where $U_\Lambda=S_{[0,0]}$. 
\end{exmp}
Often we will not distinguish between the set of $\Lambda$-points and the affinoid domain itself. 

Now we are ready to define families:
\begin{defn}
Let $\cB$ be an $A_\infty$-category over $\Lambda$ and let $S$ be a smooth affine variety, resp. affinoid domain over $\Lambda$ with ring of functions $\cO(S)$, resp. ring of analytic functions $\cO^{an}(S)$. A \textbf{family of right $A_\infty$-modules} over $\cB$ parametrized by $S$ is an assignment of a projective, graded $\cO(S)$-module, resp. $\cO^{an}(S)$-module
\begin{equation}
L\mapsto \fN(L)
\end{equation}
for all $L\in ob(\cB)$ and a family of $\cO(S)$-linear, resp. $\cO^{an}(S)$-linear, structure maps
\begin{equation}
\mu_{\fN}^{1|k}:\fN(L_{k})\otimes \cB(L_{k-1},L_k)\otimes \dots \otimes \cB(L_0,L_1) \to\fN(L_0)[1-k]
\end{equation} satisfying the right $A_\infty$-module equations. 
Families of left modules and bimodules are defined similarly, i.e. as left $A_\infty$-modules or $A_\infty$-bimodules over $\cB$ that carry $\cO(S)$, resp. $\cO^{an}(S)$-linear structures compatible with the structure maps. When the parameter space $S$ is an affine variety, we refer to the family as \textbf{algebraic}, whereas when it is an affinoid domain we call it \textbf{analytic}. 
\end{defn}
\begin{defn}
Given two families $\fN_1$, $\fN_2$, we define \textbf{a pre-morphism of families} to be an $\cO(S)$-linear, resp. $\cO^{an}(S)$-linear, $A_\infty$-module pre-morphism. In other words, a pre-morphism of families is a collection of $\cO(S)$-linear, resp. $\cO^{an}(S)$-linear, maps 
\begin{equation}
f^{1|k}:\fN_1(L_{k})\otimes \cB(L_{k-1},L_k)\otimes \dots \otimes \cB(L_0,L_1) \to\fN_2(L_0)[-k]
\end{equation}
The definition is similar for families of left-modules and bimodules. We denote the set of pre-morphisms by $hom_S(\fN_1,\fN_2)$ and its cohomology by $Hom(\fN_1,\fN_2)$. One can define a differential and composition of pre-morphisms of families analogously to pre-morphisms of $A_\infty$-modules (see \cite{seidelbook}). One can also define shifts, and cones of closed morphisms similar to ordinary $A_\infty$-modules. Hence, the families of right/left/bi-modules form a pre-triangulated dg category.
\end{defn}
\begin{exmp}
For any right $A_\infty$-module $N$ over $\cB$, and parameter space $S$, there is \textbf{the constant family} defined by $\fN(L)=N(L)\otimes\cO(S)$ (or $\fN(L)=N(L)\otimes\cO^{an}(S)$ in the case of an affinoid $S$). The structure maps are given by $\cO(S)$-linearly, resp. $\cO^{an}(S)$-linearly, extending the structure maps of $N$. We denote the constant family also by $\underline{N}$. One can similarly define constant families of left modules and bimodules. In particular, \textbf{a constant family of Yoneda modules/bimodules} is defined by letting $N$ to be a Yoneda module/bimodule. If instead of $\cO(S)$, resp. $\cO^{an}(S)$, one considers a vector bundle $E$ over $S$ (i.e. a finite rank projective module $E$ over $\cO(S)$, resp. $\cO^{an}(S)$), we call the resulting family $N\otimes E$ \textbf{a locally constant family} (when $N$ is Yoneda, \textbf{a locally constant family of Yoneda modules/bimodules}). 
\end{exmp}
\begin{defn}
A family of modules/bimodules is called \textbf{perfect} if it is quasi-isomorphic to a direct summand of an iterated cone of locally constant families of Yoneda modules/bimodules. A family is called \textbf{proper} if each $\fN(L)$ has finitely generated cohomology over $\cO(S)$, resp. $\cO^{an}(S)$. Perfect/proper families form triangulated subcategories of the dg category of families. 
\end{defn}
It is easy to see that a perfect family over a proper category is proper. Conversely:
\begin{lem}\label{lem:properimpliesperfect}
If $\cB$ is a smooth category, then any proper family over $\cB$ is perfect.	
\end{lem}
\begin{proof}
We prove this for families of right modules, other cases are similar. Let $\fN$ be a proper family of right modules. Smoothness of $\cB$ implies that the diagonal bimodule is a direct summand of an iterated cone of Yoneda bimodules $h^L\otimes_\Lambda h_{L'}$. Therefore, $\fN\simeq \fN\otimes_{\cB} \cB$ is quasi-isomorphic to a direct summand of an iterated cones of $\fN\otimes_\cB (h^L\otimes_\Lambda h_{L'})$, and it suffices to prove perfectness of the latter.
By Lemma \ref{lem:hlconvhlequalshom},
\begin{equation}
	\fN\otimes_{\cB} (h^L\otimes_\Lambda h_{L'})\simeq (\fN\otimes_{\cB} h^L)\otimes_\Lambda h_{L'} \simeq \fN(L) \otimes_\Lambda h_{L'}
\end{equation}
As $\fN$ is proper, $\fN(L)$ has finitely generated cohomology over $\cO(S)$ (hence, as $S$ is smooth, $\fN(L)$ is quasi-isomorphic to a finite complex of finite rank projective modules over $\cO(S)$). This implies $\fN(L) \otimes_\Lambda h_{L'}$ is perfect. 
\end{proof}
We will also need:
\begin{lem}\label{lem:perfconvperfisperf}
If $\cB$ is a proper category, the convolution of two perfect families over $\cB$ is perfect.
\end{lem}
\begin{proof}
The statement of the lemma holds for several variants such as convolution of two bimodules, a single bimodule and a right module, etc. The proof is the same as in Corollary \ref{cor:properconvperfisperf}. 
\end{proof}
\begin{note}\label{note:kprojrepl}
A family $\fN$ of right modules over $\cB$ parametrized by an affine variety $S$ can be seen as an $A_\infty$-functor $\cB^{op}\to \cC_{dg}(\cO(S))$, where $\cC_{dg}(\cO(S))$ is the dg category of complexes over $\cO(S)$. The category of such functors is not derived automatically; for instance, when $\cB=\Lambda$, the category of families is equivalent to $\cC_{dg}(\cO(S))$, and the quasi-isomorphisms are not invertible in this category. To deal with this issue, one needs to pass to derived category by inverting quasi-isomorphisms. One way to do this is to replace the underlying complexes of $\cO(S)$-modules by $K$-projective complexes (which is generalization of free/projective replacements of finite complexes to unbounded complexes). More precisely, by \cite{spaltenstein}, there is a $K$-projective replacement functor from $\cC_{dg}(\cO(S))$ to the subcategory of $K$-projective complexes, and one composes it with $\cB^{op}\to \cC_{dg}(\cO(S))$ to obtain a family $\fN^{kp}$ and a quasi-isomorphism $\fN^{kp}\to \fN$. One can also show $\fN^{kp}$ is essentially unique. In many of our constructions, the underlying complexes of $\cO(S)$-modules are finite free; hence, $K$-projective automatically, but even when this is not the case, one can apply functorial $K$-projective replacements to obtain one such. So as not to complicate notation further, we will keep such replacements implicit. 
This remark applies to families of left modules and bimodules, as well as families parametrized by affinoid domains. Also see \cite[Lemma 6.10]{ownpaperalg}. 
\end{note}
\begin{note}\label{note:kprojreplz2graded}
This paper is also concerned with monotone symplectic manifolds; therefore, we also work with $\bZ/2\bZ$-graded complexes. One can see a $\bZ/2\bZ$-graded complex as either a $2$-periodic unbounded complex or as a dg-module over the dga $\Lambda[w^\pm]$, $|w|=2$, $d(w)=0$. Similarly, a family of such complexes can be seen as a functor to $\cO(S)[w^\pm]$-modules ($|w|=2$, $d(w)=0$). Possibly the replacements in \cite{spaltenstein} can be assumed to satisfy $2$-periodicity as well by a slight modification of their argument. Alternatively, one can use fibrant/cofibrant replacement functors over the category of dg $\cO(S)[w^\pm]$-modules (see \cite[Prop 3.1]{kellerdg}). Thus, \Cref{note:kprojrepl} applies. 
\end{note}
\begin{note}\label{note:z2gradedmodify}
There are other minor modifications needed in $\bZ/2\bZ$-graded case. First, if $C$ is a $\bZ$-graded complex of $\cO(S)$-modules with finitely generated cohomology and if $S$ is smooth, then $C$ is quasi-isomorphic to a finite complex of finite rank projective modules. This may hold in this generality in the $\bZ/2\bZ$-graded case as well. 
On the other hand, the complexes we will encounter will already have this property; therefore, we will implicitly assume this in our arguments. For instance, when we construct a proper (hence, perfect) family of bimodules $\fM$, one can represent it as a twisted complex of locally constant families. As a result, by assuming $\cB$ is minimal, one can ensure that $\fM(L_i,L_j)$ is quasi-isomorphic to a twisted complex of finite rank projective modules. Similar conclusions hold for $(\fM\otimes_{\cB} \fM')(L_i,L_j)$, $(h_L\otimes_{\cB} \fM)(L_i)$, $(h_L\otimes_{\cB} \fM\otimes h^{L'})$, and the likes, thanks to twisted complex representation as above, and \Cref{lem:hlconvhlequalshom}. Observe that in this case, there is a compatible $\bZ$-grading on the finite rank projective replacement. This is important in the proofs of some statements such as \Cref{lem:pointwisevanishing} and \Cref{prop:semicontinuityoverbase}, as they use convergent spectral sequence arguments. 
\end{note}
We now explain what it means for a family $\fM$ to define a group action:
\begin{defn}
Let $S$ be an affine algebraic group over $\Lambda$, $\cB$ be an $A_\infty$-category and $\fM$ be a family of bimodules over $\cB$ parametrized by $S$. The family $\fM$ is called \textbf{group-like} if 
\begin{enumerate}
	\item $\fM|_{e_S}$ is quasi-equivalent to diagonal bimodule
	\item $\pi_2^*\fM\otimes_{\cB}^{rel} \pi_1^*\fM\simeq m^*\fM$, where $\pi_1$ and $\pi_2$ are the projection maps $S\times S\to S$ and $m:S\times S\to S$ is the group multiplication
\end{enumerate}
\end{defn}
Here, the quasi-isomorphism is a quasi-isomorphism of families over $S\times S$. The following implication explains the term group-like:
\begin{lem}
If $\fM$ is group-like, then $\fM|_{z_1}\otimes \fM|_{z_2}\simeq \fM|_{z_1+z_2}$ for every $z_1,z_2\in S$. In particular, each $\fM|_z$ is invertible with an inverse given by $\fM|_{z^{-1}}$. 
\end{lem}
One can also prove:
\begin{lem}\label{lem:grouplikeimpliesperfect}
If $\fM$ is group-like and $\cB$ is smooth, then $\fM$ is perfect.
\end{lem}
\begin{proof}
The proof is similar to perfectness of a single invertible bimodule: if $\fM$ is an invertible bimodule, then 	$(\cdot)\otimes_{\cB} \fM$ is an auto-equivalence on the category of bimodules. As perfect bimodules are the same as compact objects in the category of bimodules (see \cite{kellerdg}), auto-equivalences preserve perfectness. In particular, $\fM\simeq \cB\otimes_\cB \fM$ is perfect.

To apply this idea to families, one considers the functor $(\cdot)\otimes_{\cB}^{rel} \fM$, i.e. the convolution relative to $S$, acting on the families of bimodules parametrized by $S$. As above, this is an auto-equivalence, with a quasi-inverse given by $(\cdot)\otimes_{\cB}^{rel}\fM^-$, where $\fM^-$ is the family obtained by pulling $\fM$ back along $S\to S,z\mapsto z^{-1}$ (hence, $\fM^-|_z\simeq \fM|_{z^{-1}}$). One can similarly identify the compact objects in the category of families with perfect families to see that auto-equivalences preserve perfectness. 
Hence, as before $\fM\simeq \underline{\cB}\otimes_\cB^{rel} \fM$ is perfect. Here, $\underline{\cB}$ is the constant family of bimodules with fiber $\cB$, which is perfect as $\cB$ is a smooth category. 
\end{proof}


\section{The global algebraic family $\fM^M$ and group-like property}\label{sec:algfamilydefclass}
\subsection{Definition of the family and local behavior via Fukaya's trick}
Recall that we use $\cF(M)$ to denote the Fukaya category of $(M,\omega_M)$ spanned by a fixed finite set of Bohr-Sommerfeld monotone Lagrangians that generate the Fukaya category. In this section, we will define a family $\fM^M$ of bimodules over $\cF(M)$ parametrized by 
\begin{equation}
H^1(M,\bG_m):=Spec(\Lambda[z^{H_1(M)}])\cong \bG_m^{b_1(M)}
\end{equation}
where $H_1(M)$ denote $H_1(M,\bZ)$ modulo torsion. Fix a basepoint on $M$ and relative homotopy classes of paths from the basepoint to the generators of $\cF(M)$ (i.e. to a point on the corresponding Hamiltonian chord). 
Let $u$ denote a map from a disc with boundary marking points to $M$ such that the boundary maps to $\bigcup L_i$ and near marked points $u$ is asymptotic to generators of the Fukaya category. 
\begin{defn}
We define the family of bimodules $\fM^M$ as follows: To a pair $(L_i,L_j)$, it associates the ($\bZ/2\bZ$-)graded free $\Lambda[z^{H_1(M)}]$-module 
\begin{equation}\label{eq:bimoduleunderlying}
\fM^M(L_i,L_j):=\Lambda[z^{H_1(M)}]\otimes_\Lambda CF(L_i,L_j)
\end{equation}
Its differential is defined by the formula 
\begin{equation}\label{eq:bimodulediff}
\mu^{0|1|0}(x):= \sum \pm T^{E(u)}z^{[\partial_h u]}.y
\end{equation}
where the count is over the Floer strips with input $x$ and output $y$ up to translation. Here, $E(u)$ denotes the energy of $u$ and $[\partial_h u]\in H_1(M)$ denotes the class obtained by concatenating the path from the basepoint to $x$, the $u$-image of one side of the Floer strip and the inverse of the path from the basepoint to $y$. 
More generally, define the structure maps of $\fM^M$ by the formula
\begin{equation}\label{eq:bimodulestructure}
(x_1,\dots,x_e|x|x_d',\dots, x_1')\mapsto \sum\pm T^{E(u)}z^{[\partial_h u]}.y
\end{equation}
where the sum is over pseudo-holomorphic discs as in Figure \ref{figure:structurebimodule}. The class $[\partial_h u]\in H_1(M)$ is defined similarly, by concatenating the $u$-image of a path from $x$ to $y$ with fixed paths from the basepoint to $x$ and $y$. The blue middle line in Figure \ref{figure:structurebimodule} shows the path from $x$ to $y$ concatenated with fixed paths to obtain $[\partial_h u]$.
\end{defn}
\begin{figure}\centering
	\includegraphics[height=4 cm]{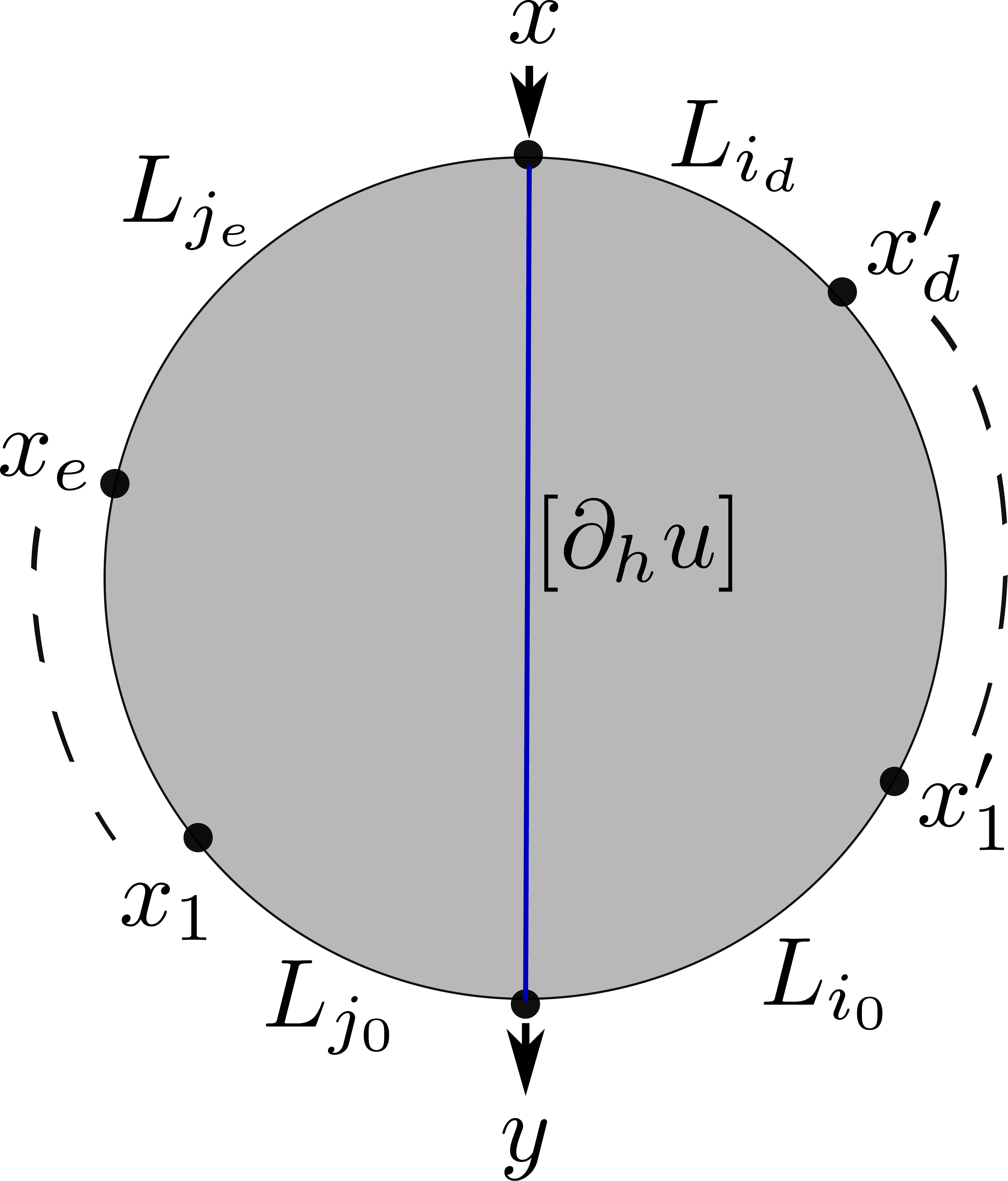}
	\caption{The disc counts defining $\fM^M$}
	\label{figure:structurebimodule}
\end{figure}
The sums defining the structure maps are finite, thanks to Bohr-Sommerfeld assumption; hence, $\fM^M$ is well-defined. 
\begin{rk}
Figure \ref{figure:structurebimodule} represents a count of pseudo-holomorphic discs, not quilted discs. The blue middle line is a pictorial representation of $[\partial_h u]$, not a seam. Also note that the signs in (\ref{eq:bimodulediff}) and (\ref{eq:bimodulestructure}) are determined by the orientation of the moduli of discs in the same way the signs in the defining equation for the diagonal bimodule do. Indeed, these equations are simply the deformations of the equations defining the diagonal bimodule. 
\end{rk}
Let $P\subset H^1(M,\bR)$ be a closed polytope:
\begin{defn}
Define the family $\fM_P^M$ of bimodules over $\cF(M)$ parametrized by the affinoid domain	$S_P\subset H^1(M,\bG_m)$ by replacing (\ref{eq:bimoduleunderlying}) by 
\begin{equation}\label{eq:Pbimoduleunderlying}
\fM_P^M(L_i,L_j):=\Lambda\{z^{H_1(M)}\}_P\otimes_\Lambda CF(L_i,L_j)
\end{equation}
The differential and the structure maps are defined by (\ref{eq:bimodulediff}) and (\ref{eq:bimodulestructure}).
\end{defn}
\begin{rk}
If one drops Bohr-Sommerfeld assumption, $\fM^M$ is not well-defined as (\ref{eq:bimodulediff}) and (\ref{eq:bimodulestructure}) may not converge over $H^1(M,\bG_m)$. On the other hand, if $P$ is a small neighborhood of $0$, then $\fM_P^M$ is well-defined. This follows from Fukaya's trick applied to diagonal. We will prove a simple variant of this.
\end{rk}
Given Bohr-Sommerfeld monotone $L$, one can similarly define a family of right modules deforming the right Yoneda module $h_L$. More precisely:
\begin{defn}\label{defn:hlalg}
Let $h_L^{alg}$ denote the family of right modules over $\cF(M)$ parametrized by $Spec(\Lambda[z^{H_1(M)}])$ that associates the free graded complex \begin{equation}\label{eq:rightalgunderlying}
h_L^{alg}(L_i):=\Lambda[z^{H_1(M)}]\otimes_\Lambda CF(L_i,L)
\end{equation} to each object $L_i$ and whose structure maps are defined by 
\begin{equation}\label{eq:rightalgstructure}
(x;x_e,\dots,x_1)\mapsto \sum\pm T^{E(u)}z^{[\partial_2 u]}.y
\end{equation}
where the sum varies over pseudo-holomorphic discs as in Figure \ref{figure:structureright}. The class $[\partial_2 u]\in H_1(M)$ is defined analogously, by concatenating the blue line from $x$ to $y$ in Figure \ref{figure:structureright} with the fixed path from the basepoint to generators $x$ and $y$.  Analogous to $\fM_P^M$, define $h_L^{alg}|_P$ by replacing (\ref{eq:rightalgunderlying}) by $\Lambda\{z^{H_1(M)}\}_P\otimes_\Lambda CF(L_i,L)$. 
\end{defn}
\begin{note}\label{note:localsystemsclass}
Recall that we assume a Lagrangian brane $L$ is equipped with a unitary local system $\xi_L$, mostly implicit in the notation. In the presence of a unitary local system $\xi_L$, one would normally need to modify (\ref{eq:rightalgstructure}) as $\sum\pm T^{E(u)}\xi_L^{[\partial_L u]}z^{[\partial_2 u]}.y$. As remarked before, we will only work with local systems defined over all $M$; hence, for notational convenience, we instead replace (\ref{eq:rightalgstructure}) by $\sum\pm T^{E(u)}\xi_L^{[\partial_2 u]}z^{[\partial_2 u]}.y$. The families that are defined by these two formulas are isomorphic: an isomorphism is given by rescaling the generators. See Note \ref{note:globallocal1}
\end{note}
Observe that $h_L^{alg}|_{z=1}$ is the same as right Yoneda module of the Lagrangian $L$. This family is well-defined thanks to the assumption that $L$ is Bohr-Sommerfeld monotone. However, as we will see, even without this assumption, it is well-defined over a small affinoid domain containing $H^1(M,U_\Lambda)\subset H^1(M,\bG_m)$. 

\begin{figure}\centering
	\includegraphics[height=4 cm]{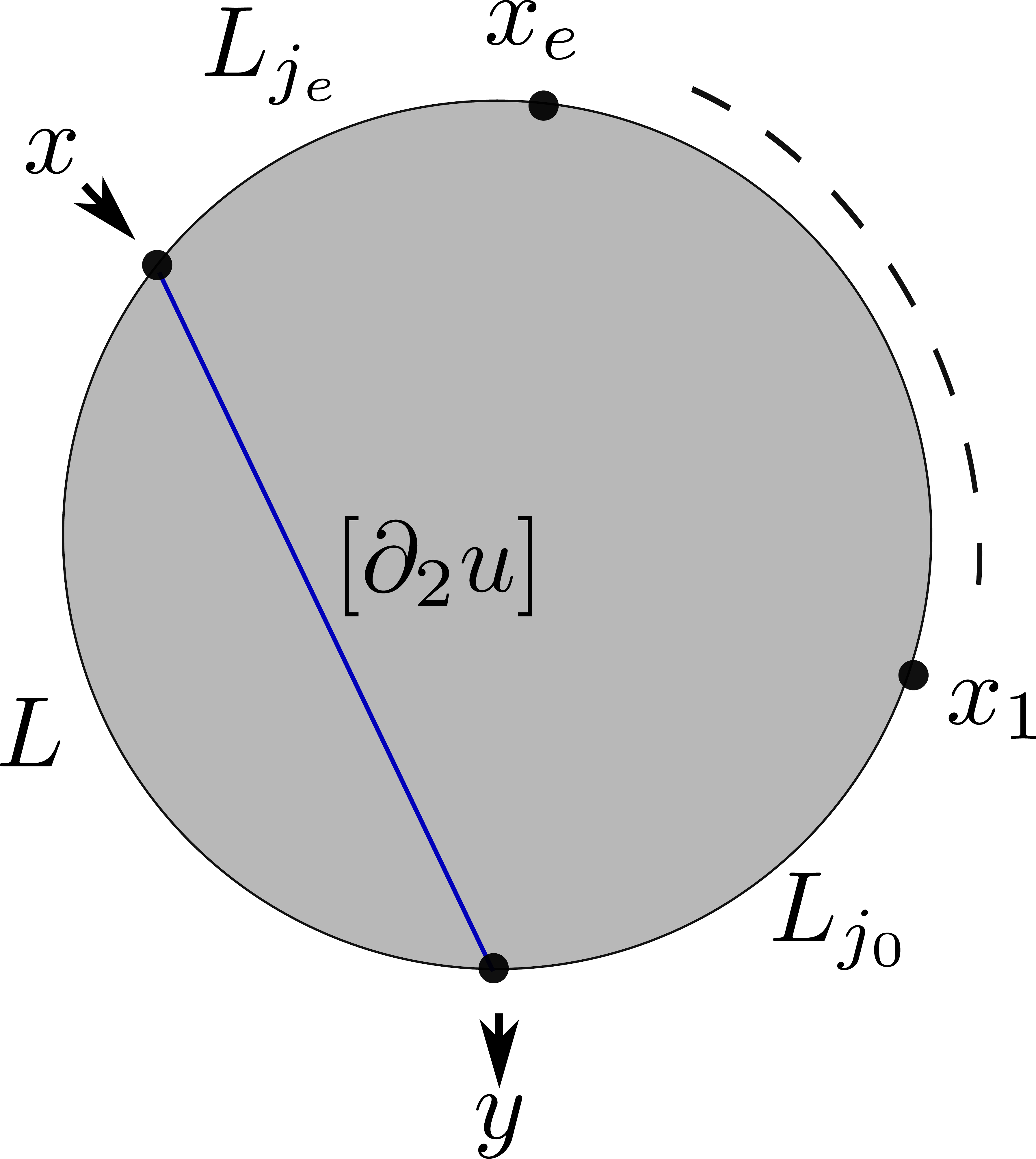}
	\caption{The disc counts defining $h_L^{alg}$}
	\label{figure:structureright}
\end{figure}

One can identify the set of isomorphism classes of $U_\Lambda$-local systems on $M$ with $H^1(M,U_\Lambda)\cong U_\Lambda^{b_1(M)}$ via their holonomy. Moreover, given $z_0\in H^1(M,U_\Lambda)\subset H^1(M,\bG_m)$, the module $h_{L}^{alg}|_{z=z_0}$ is quasi-isomorphic to $h_{(L,\xi_{z_0})}$, where $\xi_{z_0}$ denote the local system corresponding to $z_0$. 

Now, let $\tilde L$ denote a tautologically unobstructed Lagrangian brane, that is not necessarily Bohr-Sommerfeld monotone. By a small Hamiltonian perturbation assume $\tilde L\pitchfork L_i$ and extend the perturbation data to include boundary conditions on $\tilde L$ as well. We assume the Hamiltonian term of the pair $(L_i,\tilde L)$ vanishes as before. We define an analytic family similar to Definition \ref{defn:hlalg}:
\begin{defn}\label{defn:analyticdefolag}
The graded vector space underlying $h_{\tilde L}^{an}$ is given by
\begin{equation}
h_{\tilde L}^{an}(L_i)=CF(L_i,\tilde L)\otimes_\Lambda \Lambda\{z^{H_1(M)}\}_P
\end{equation}
where $P$ is a polytope containing $0$ in its interior. The differential and the structure maps are defined by (\ref{eq:rightalgstructure}). The domain $P$ will mostly be omitted from the notation, but when we want to emphasize, we will use the notation $h_{\tilde L}^{an}|_P$.
\end{defn}
\begin{rk}
	When $L$ is Bohr-Sommerfeld monotone $h_L^{alg}|_P$ and $h_L^{an}|_P$ coincide.		
\end{rk}
Observe that $h_{\tilde L}^{an}$ is well-defined if one allows $P$ to be $\{0\}$. Indeed, as observed, the restriction of this family to a point $z_0\in H^1(M,U_\Lambda)$ gives $h_{(L,\xi_{z_0})}$. The proof of Lemma \ref{lem:hlgeometricnearby} will show that one can let $P$ to be a small neighborhood of $0$:
%
\begin{lem}\label{lem:hlgeometricnearby}
For a small neighborhood $P$ of $0\in H^1(M,\bR)$, any restriction $h_{\tilde L}^{an}|_z$, $z\in S_P=val_T^{-1}(P)$ is quasi-isomorphic to the Yoneda module of a Lagrangian-local system pair. More precisely, one can write given $z\in P\subset H^1(M,\bG_m)$ as $z=T^{val_T(z)}z_0$, where $val_T(z)\in H^1(M,\bR)$ and $z_0\in H^1(M,U_\Lambda)$, and 
\begin{equation}
	h_{\tilde L}^{an}|_z\simeq h_{(\phi^1_\alpha(\tilde L),\xi_{z_0})}=:h_{\phi_z(\tilde L)}
\end{equation}
where $\alpha$ is a closed $1$-form representing $val_T(z)$.
\end{lem}
\begin{proof}
We prove this lemma using Fukaya's trick. First, choose a collection of closed $1$-forms on $M$, representing a basis of $H^1(M,\bR)$. This collection spans a subspace of the space of closed $1$-forms isomorphic to $H^1(M,\bR)$, let $\alpha(v)$ denote the $1$-form corresponding to $v\in H^1(M,\bR)$ and $\phi_v$ denote the time $1$-flow of $X_{\alpha(v)}$. The collection $\{\phi_v\}\subset Symp^0(M,\omega_M)$ does not define an action of $H^1(M,\bR)$ in the strict sense, but it defines an action up to Hamiltonian isotopy (i.e. $\phi_{v_1}\circ\phi_{v_2}$ is Hamiltonian isotopic to $\phi_{v_1+v_2}$). Moreover, $\phi_v$ is well defined up to Hamiltonian isotopy; hence, $h_{\phi_v(\tilde L)}$ is well-defined up to quasi-isomorphism. Observe that when $v$ is close to $0$, the intersection points $L_i\cap \phi_v(\tilde L)$ can be identified with $L_i\cap \tilde L$.

For a small neighborhood $P\subset H^1(M,\bR)$ of $0$, choose a smoothly varying family of diffeomorphisms $\{\psi_v:v\in P\}$ such that $\psi_v(\tilde L)=\phi_v(\tilde L)$ and $\psi_v(L_i)=L_i$. We choose $\psi_v$ with small support near $\tilde L$ so that it does not effect Hamiltonian chords between $L_i$. 

Choose perturbation data for discs with boundary on $(L_{i_0},L_{i_1},\dots, L_{i_k}, \phi_v(\tilde L))$ such that it is related to initial perturbation data for $(L_{i_0},L_{i_1},\dots, L_{i_k}, \tilde L)$ by push-forward along $\psi_v$. The almost complex structures are tame with respect to $(\psi_v)_*\omega_M$; hence, with respect to $\omega_M$, as long as $v$ is close to $0$ (strictly speaking, we allow the almost complex structures to vary over the discs, but one can choose families of almost complex structures in a way that tameness with respect to $(\psi_v)_*\omega_M$ holds uniformly for almost complex structures in these families, as long as $v$ belong to a small neighborhood of $0$). 
As $\psi_v$ is supported near $\tilde L$, the Hamiltonian chords between $L_i$ with respect to push-forward Hamiltonian (and $\omega_M$) remains unchanged. Note that push-forward still modifies the Floer data for the pair $(L_i,L_j)$, but we have an identification of the Floer chains, and the number of pseudo-holomorphic strips --with the same boundary conditions on $L_i$ and the same asymptotic conditions on the strip-like ends-- remains unchanged. This is analogous to invariance of such counts with no Hamiltonian terms and asymptotic to actual intersection points. The same applies to higher $A_\infty$-operations between $L_i$; hence, Fukaya category $\cF(M)$ spanned by $\{L_i\}$ remains unchanged, in the strict sense. Therefore, one can use the new data to define the module $h_{\phi_v(L)}$ over $\cF(M)$. 

$CF(L_i,\tilde L)$ can be identified with $CF(L_i,\phi_v(\tilde L))$ as graded vector spaces. One also has an identification of the Floer strips, that sends $u$ to $\psi_v\circ u$, but the energy changes. Similarly, pseudo-holomorphic discs with boundary on $(L_{i_0},L_{i_1},\dots, L_{i_k}, \tilde L)$ and $(L_{i_0},L_{i_1},\dots, L_{i_k}, \phi_v(\tilde L))$ (defined by the original data and the push-forward data) can be identified. The energy changes according to the formula
\begin{equation}\label{eq:energyidentity}
E(\psi_v\circ u)=E(u)+v([\partial_2 u])-g_v(x)+g_v(y)
\end{equation}
Here, both $E(\psi_v\circ u)$ and $E(u)$ denote topological energies, computed with respect to $\omega_M$ (and the data used to define respective discs). $v([\partial_2 u])$ denotes the evaluation of $[\partial_2 u]\in H_1(M)$ at $v\in H^1(M,\bR)$. $x$ and $y$ denote the module input and the output of the marked disc $u$ (that correspond to the inputs and outputs of $\psi_v\circ u$). $g_v(x)$, resp. $g_v(y)$ are real numbers that depend only on $(v,x)$, resp. $(v,y)$. We will define $g_v(x)$ and $g_v(y)$; however, we will not attempt to give explicit formulae for them. 

Assuming (\ref{eq:energyidentity}), we can conclude Lemma \ref{lem:hlgeometricnearby} as follows: for simplicity, assume $\tilde L$ carries no local system and let $z=T^vz_0$, where $v$ is close to $0$. For such small $v$, one can use the push-forward data to define $h_{\phi_z(\tilde L)}:=h_{\phi_v(\tilde L),\xi_{z_0}}$ as above, and identify the generators of $h_{\phi_z(\tilde L)}(L_i)$ with the generators of $h_{\tilde L}(L_i)=CF(L_i,\tilde L)$. Then, the differential of $h_{\phi_z(\tilde L)}(L_i)$ satisfies 
\begin{equation}
	\mu^1(T^{g_v(x)}x)=\sum \pm T^{E(\psi_v\circ u)+g_v(x)}\xi_{z_0}^{[\partial_2 u]}y= \sum \pm T^{E(u)+v([\partial_2 u])}\xi_{z_0}^{[\partial_2 u]}	T^{g_v(y)}y
\end{equation}
by (\ref{eq:energyidentity}). If we replace every generator $x$ by $x'=T^{g_v(x)}x$, this expression turns into 
\begin{equation}
	\mu^1(x')=\sum \pm T^{E(u)+v([\partial_2 u])}\xi_{z_0}^{[\partial_2 u]}	y'
\end{equation}
which is the same as (\ref{eq:rightalgstructure}) evaluated at $z=T^vz_0$. The same holds for the higher structure maps, and this shows this module is isomorphic to $h^{an}_{\tilde L}|_{z=T^vz_0}$. One adds an extra $\xi_{\tilde L}^{[\partial_2 u]}$-term when $\tilde L$ carries the local system $\xi_{\tilde L}$ (which we assume to be globally defined for simplicity). 

(\ref{eq:energyidentity}) is similar to the one given in \cite{abouzaidicm} and it follows from Stokes' theorem. To see why it holds, first notice that if we let $(H.\gamma, J)$ to be the perturbation data on a disc $S$, where $\gamma$ is a closed $1$-form on $S$ that vanish in tangent directions to boundary, its topological energy $E(u)=\int_S u^*\omega_M -d(u^*H.\gamma)$ differs from the symplectic area by $\int_{\partial S} u^*H.\gamma$ (where this expression involves integrals over Hamiltonian chords $S$ is asymptotic to as well, and only these matter as $\gamma|_{\partial S}=0$). When we let $u:S\to M$ evolve by $\psi_v$, this isotopy does not change the Hamiltonian chords between $L_i$, and the integral of $u^*H.\gamma$ over the chord remains the same as well. On the other hand, the Hamiltonian term associated to pair $(L_i,\tilde L)$ is zero, and remains zero under isotopy; therefore, $\int_{\partial S} u^*H.\gamma$ stays the same as $u$ (and $H$) evolves. In other words, the Hamiltonian part of topological energy is the same, and the difference of topological energies is equal to the difference of symplectic areas of $\psi_v\circ u$ and $u$. 

To calculate the symplectic area difference, we apply Stokes' theorem. As we apply isotopy $\psi_v$, all the boundary components, except the one mapping to $\tilde L$ vary over some $L_i$. Hence, the symplectic area traced by other components vanish, and the symplectic are difference of $u$ and $\psi_v(u)$ is the area traced by the boundary component on $\tilde L$ under the isotopy $\psi_{tv}, t\in[0,1]$. To compute the area traced by the part on $\tilde L$, fix a basepoint on $\tilde L$ and paths to all generators $\tilde L\cap L_i$ from this basepoint. Define $[\partial_{\tilde L} u]\in H_1(\tilde L)$ as before, i.e. as the concatenation of the fixed paths on $\tilde L$ with the $\tilde L$-part of the boundary of $u$. The area swiped by this loop is equal to $v([\partial_{\tilde L}u])$ (as $\phi_{tv}^{-1}\circ \psi_{tv}|{\tilde L}$ is homotopic to identity). 
It differs from the area swiped by the $\tilde L$-part of the boundary by the areas swiped by the fixed paths from the basepoint on $\tilde L$, i.e. it is of the form $n_v(y)-n_v(x)$, where $n_v(x),n_v(y)\in\bR$ only depend on the generators $x$, resp. $y$, and $v$. From this we can conclude the energy difference $E(\psi_v\circ u)-E(u)=v([\partial_{\tilde L} u])+n_v(y)-n_v(x)$. On the other hand, the difference
\begin{equation}\label{eq:loopdifference}
	v([\partial_2 u])-v([\partial_{\tilde L} u])
\end{equation}
is also of the form $m_v(y)-m_v(x)$. Indeed, for each generator $x$, we have a fixed path from the basepoint on $M$ to $x$ and another from the basepoint on $\tilde L$ to $x$. Concatenating them gives a path that only depend on $x$ and (\ref{eq:loopdifference}) can be seen as the difference of area swiped by paths between the basepoints that correspond to $x$ and $y$. Therefore, 
\begin{equation}
	E(\psi_v\circ u)-E(u)=v([\partial_2 u])+(m_v(y)+n_v(y))-(m_v(x)+n_v(x))
\end{equation}
Letting $g_v(x)=m_v(x)+n_v(x)$, we conclude the proof of (\ref{eq:energyidentity}).
\end{proof}
\begin{cor}\label{cor:hlanconverges}
$h_{\tilde L}^{an}$ is well-defined for small enough $P$ containing $0$ in its interior.	
\end{cor}
\begin{proof}
This follows from the proof of Lemma \ref{lem:hlgeometricnearby}. Namely, the proof shows that, if $v$ is small, then with the right choice of data, and after scaling the generators $h_{\phi_z(\tilde L)}$ has the same structure maps as $h_{\tilde L}^{an}$ evaluated at $z=T^vz_0$. the former is well-defined, showing the convergence of series defining $h_{\tilde L}^{an}|_{z=T^vz_0}$ as well. 
\end{proof}
\begin{note}\label{note:convergenceofmp}
For $\fM_P^M$ to be well defined, one needs (\ref{eq:bimodulediff}) and (\ref{eq:bimodulestructure}) to converge as long as $z\in S_P$. Previously, we relied on Assumption \ref{assump:monotoneexact} for this. However, one can apply ideas in the proof of Lemma \ref{lem:hlgeometricnearby} and Corollary \ref{cor:hlanconverges} to establish convergence without this assumption. The main additional difficulties are Hamiltonian terms, and possibly repeating boundary labels. More precisely, analogous to before, we would like to choose diffeomorphisms $\psi_v$ such that $L_j$ on the left hand side of Figure \ref{figure:structurebimodule} map to $\phi_v(L_j)$ and $L_i$ on the right hand side stay the same (setwise). This is not always possible as some $L_{j_l}$ and some $L_{i_k}$ may be the same. One can fix these problems simultaneously as follows: first use families of Hamiltonian diffeomorphisms parametrized by the discs to gauge the perturbation data and the boundary conditions so that the Hamiltonian terms for pairs $(L_{i_d},L_{j_e})$ and $(L_{i_0},L_{j_0})$ become zero, and $L_{i_k}$ becomes transverse to $L_{j_l}$ for all $k$ and $l$. One can use the same family of Hamiltonian diffeomorphisms to write a correspondence of solutions to perturbed Cauchy-Riemann equations such that the topological energies match. Then, one can apply the proof of Lemma \ref{lem:hlgeometricnearby} to show the convergence of (\ref{eq:bimodulediff}) and (\ref{eq:bimodulestructure}) for the new boundary conditions and perturbation data, and for small $val_T(z)$. However, the series for the new conditions is exactly the same as the old conditions; therefore, we obtain convergence.
\end{note}
For a Bohr-Sommerfeld monotone $L$, one can define a morphism of families
\begin{equation}\label{eq:maphlmtohl}
h_L\otimes_{\cF(M)}\fM^M\to h_L^{alg}
\end{equation}
by the formula
\begin{equation}\label{eq:formulamaphlmtohl}
(x\otimes x_1\otimes\dots \otimes x_e\otimes m;x_d',\dots x_1')\mapsto \sum\pm T^{E(u)}z^{[\partial_1u]}.y
\end{equation}
Here, $x\otimes x_1\otimes\dots \otimes x_e\otimes x$ is a generator of $h_L\otimes_{\cF(M)}\fM^M$. The sum ranges over pseudo-holomorphic curves with input $x, x_1,\dots , x_e, m,x_d',\dots, x_1'$ and output $y$. The class $[\partial_1u]\in H_1(M)$ is defined similar to $[\partial_hu]$, namely by concatenating the fixed path from the base point to $m$, the $u$-image of the path from the marked point mapping to $m$ to the marked point mapping to $y$, and inverse of the fixed map from the basepoint to $y$. We represent this class by the blue line in Figure \ref{figure:maphlmtohl}. It is easy to check this defines a map of bimodules.
\begin{figure}\centering
	\includegraphics[height=4 cm]{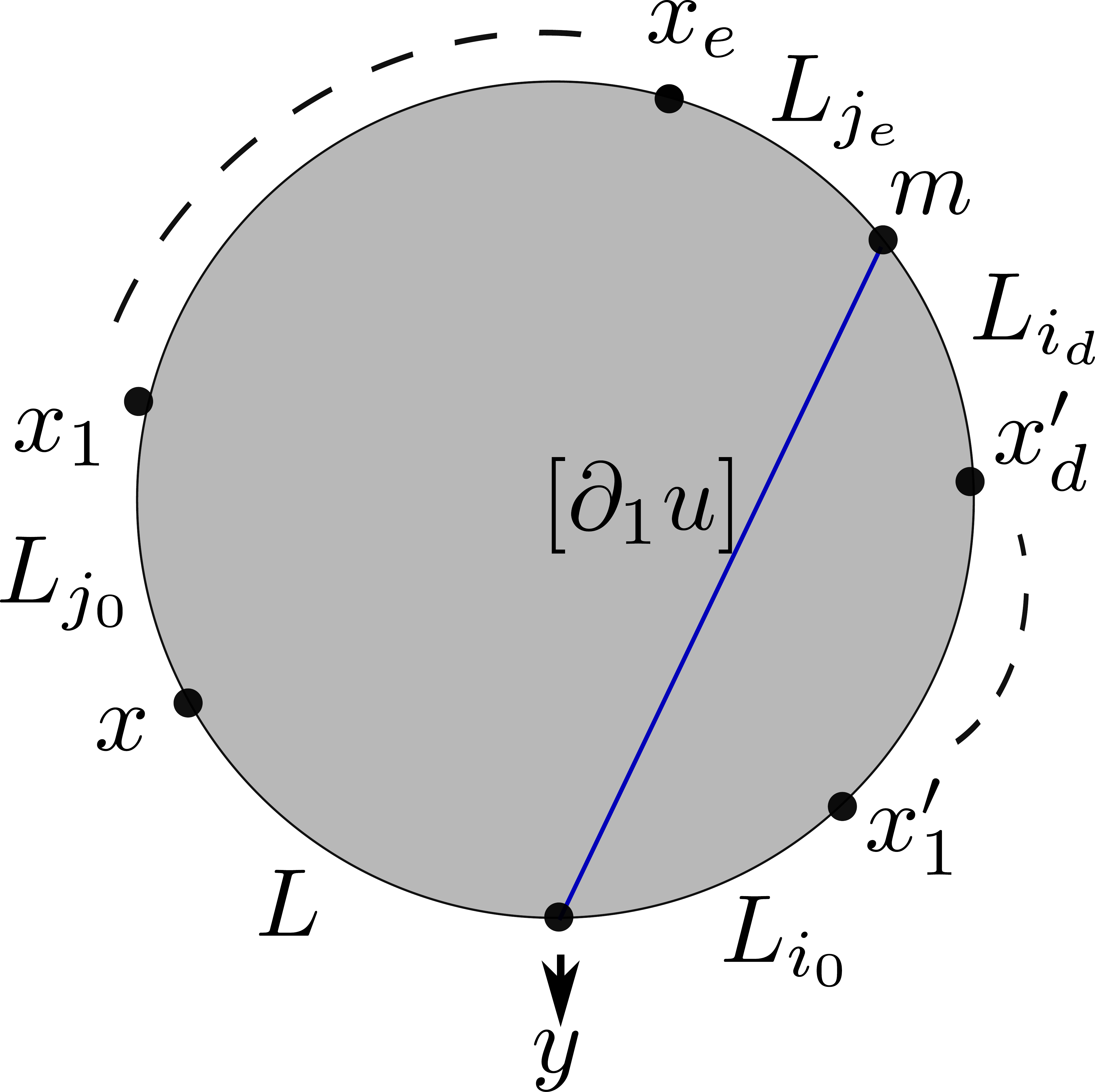}
	\caption{The disc counts defining (\ref{eq:maphlmtohl})}
	\label{figure:maphlmtohl}
\end{figure}
Analogously, one can define a map of families
\begin{equation}\label{eq:analyticmaphlmtohl}
h_{\tilde L}\otimes_{\cF(M)} \fM_P^M\to h_{\tilde L}^{an}
\end{equation}
for small $P\subset H^1(M,\bR)$ containing $0$. One still needs to check the convergence of series (\ref{eq:formulamaphlmtohl}) for small enough $P$. This can be shown using the same idea in Corollary \ref{cor:hlanconverges} and Note \ref{note:convergenceofmp}.

One can prove the following:
\begin{lem}\label{lem:hlconvmtohlsemicont}
The map (\ref{eq:analyticmaphlmtohl}) is a quasi-isomorphism of families for small $P$. Therefore, (\ref{eq:maphlmtohl}) is a quasi-isomorphism when restricted to a small domain $S_P$.
\end{lem}
\begin{proof}
Consider the cone of (\ref{eq:analyticmaphlmtohl}). The cone is acyclic at $z=1$, as this map at $z=1$ is a well-known quasi-isomorphism from the convolution of a module with the diagonal bimodule to the module itself, which is given in Note \ref{note:nconvdiagisn}. One can similarly check that the cone vanishes for $z\in H^1(M,U_\Lambda)$. 

To see this, extend $\cF(M)$ to include pairs $(L_i,\xi_z^{\otimes k})$, where $L_i\in \cF(M)$ and $k\in\bZ$, as well as pairs $(\tilde L,\xi_z^{\otimes k})$ (more precisely, triples $(L_i,\xi_z^{\otimes k},k)$, resp. $(L,\xi_z^{\otimes k},k)$, i.e. one remembers $k$). Then, $\xi_z$ acts on this new set of objects freely by $(L_i,\xi_z^{\otimes k},k)\mapsto (L_i,\xi_z^{\otimes k+1},k+1)$, etc. By choosing Floer and perturbation data to be equivariant under this free action, one obtains a strictly $\bZ$-equivariant category $\widetilde{\cF}(M)$, i.e. $\{\xi_z^{\otimes k}\}$ act on $\widetilde{\cF}(M)$ by strict auto-equivalences. Note that when defining Floer cohomology with local coefficients, we use the base point on $M$, and the paths from this base point to the generators, as in Note \ref{note:localsystemsclass} (as opposed to choosing separate base points and paths on the Lagrangians). Denote the action of $\xi_z$ on $\widetilde{\cF}(M)$ by $\Xi$. Consider the bimodule corresponding to the action of $\Xi$, i.e. 
\begin{equation}
	(L',L'')\mapsto hom(L',\Xi(L''))\cong hom(\Xi^{-1}(L'),L'')
\end{equation}
Denote this bimodule by $\cM_\Xi$. Then for purely algebraic reasons, 
\begin{equation}\label{eq:hlxitohxil}
	h_{\tilde L}\otimes_{\widetilde{\cF}(M)} \cM_\Xi\simeq h_{\Xi(\tilde L)}
\end{equation}
More precisely, one can write a map 
\begin{equation}\label{eq:hlxitohxilexpanded}
	h_{\tilde L}\otimes_{\widetilde{\cF}(M)} \cM_\Xi=hom(\cdot,\tilde L)\otimes_{\widetilde{\cF}(M)} hom(\Xi^{-1}(\cdot),\cdot)\to hom(\Xi^{-1}(\cdot),\tilde L)
\end{equation}
using the same formula (\ref{eq:nconvdiagisn}) (and the analogous higher maps). That this is an equivalence can be proven in the same way, or deduced from Note \ref{note:nconvdiagisn} directly. It is clear that $hom(\Xi^{-1}(\cdot),\tilde L)\simeq h_{\Xi(\tilde L)}$. Moreover, as $\cF(M)$ generates $\widetilde{\cF}(M)$, the base of the convolution can be replaced by $\cF(M)$. The restriction of $\cM_\Xi$ to $\cF(M)$ is $\fM^M_P|_z$ and the restriction of $h_{\Xi(\tilde L)}$ to $\cF(M)$ is $h_{\tilde L}^{an}|_z$. Combining these we obtain
\begin{equation}
	h_{\tilde L}\otimes_{\cF(M)} \fM^M_P|_z\simeq 	h_{\tilde L}\otimes_{\cF(M)} \cM_\Xi\simeq h_{\Xi(\tilde L)}\simeq h_{\tilde L}^{an}|_z
\end{equation}
It is easy to check that this quasi-isomorphism is the same as (\ref{eq:analyticmaphlmtohl}) restricted to $z$. 
In other words, the cone of (\ref{eq:analyticmaphlmtohl}) vanishes at $z\in H^1(M,U_\Lambda)$.

As $\cF(M)$ is smooth and proper, the proper families $\fM_P^M$ and $h_{\tilde L}^{an}$ are perfect by Lemma \ref{lem:properimpliesperfect}. By Lemma \ref{lem:perfconvperfisperf}, one can conclude that the convolution $h_{\tilde L}\otimes_{\cF(M)} \fM_P^M$ and the cone of (\ref{eq:analyticmaphlmtohl}) are also perfect families. Therefore, by semi-continuity (i.e. by \Cref{prop:semicontinuityoverbase}), the cone vanishes in a small neighborhood of $H^1(M,U_\Lambda)$.
\end{proof}

\subsection{Group-like property}\label{subsec:grouplike}
To show that $\fM^M$ is group-like, we write a morphism 
\begin{equation}\label{eq:grouplikemap}
\pi_2^*\fM^M\otimes_{\cF(M)}^{rel} \pi_1^*\fM^M\to m^*\fM^M
\end{equation}
of families analogous to (\ref{eq:maphlmtohl}). The defining formula for the morphism is 
\begin{equation}\label{eq:grouplikeformula}
(x_1,\dots,x_e|m_2\otimes \dots\otimes  m_1|x_d',\dots,x_1')\mapsto \sum \pm T^{E(u)}z_1^{[\partial_1u]}z_2^{[\partial_2u]}.y
\end{equation}
Here, $z_1,z_2$ are coordinates corresponding to first and second component, $m_2\otimes \dots\otimes  m_1$ is a generator of  $\pi_2^*\fM^M\otimes_{\cF(M)}^{rel} \pi_1^*\fM^M$ and the sum ranges over pseudo-holomorphic discs as in Figure \ref{figure:grouplikeformula}. The classes $[\partial_1u]$ and $[\partial_2u]$ are defined analogously, namely by concatenating the $u$-image of paths from $m_1$, resp. $m_2$, to $y$ with the fixed paths from the base point. It is easy to check that this defines a morphism of families of bimodules. 
\begin{figure}\centering
	\includegraphics[height=4 cm]{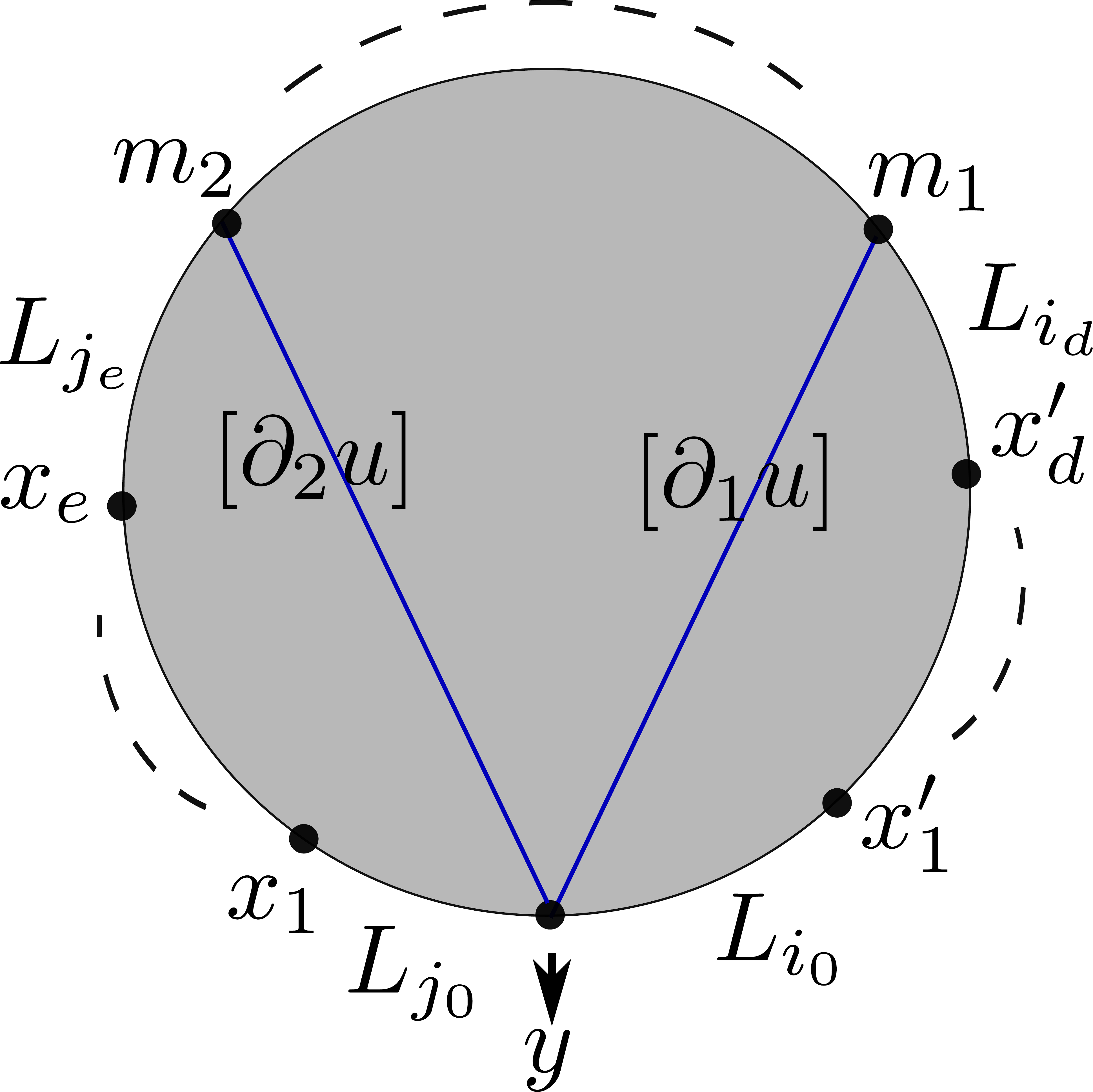}
	\caption{The disc counts defining (\ref{eq:grouplikemap}) }
	\label{figure:grouplikeformula}
\end{figure}
Our goal is to prove
\begin{prop}\label{prop:grouplike1}
	The map (\ref{eq:grouplikemap}) is a quasi-isomorphism; hence, $\fM^M$ is group-like.	
\end{prop}
As $\cF(M)$ is spanned by Bohr-Sommerfeld monotone Lagrangians, we can expand it to contain pairs $(L,\xi_{L})$, where $L\in ob(\cF(M))$, and $\xi_L$ is a --possibly non-unitary-- $\Lambda^*$-local system of rank $1$. More precisely, one can define $A_\infty$-structure via the same formulae, i.e. by multiplying summand in (\ref{eq:fukdiff}) and (\ref{eq:fukstructure}) by $\xi_L^{[\partial_L]}$, for all $L$ to which boundary components of the holomorphic disc map. Without Assumption \ref{assump:monotoneexact}, these sums may no longer be well-defined; however, by Lemma \ref{lem:finitediscs}, our assumptions imply finiteness of these sums. As before, for convenience, we only consider local systems that extend to all $M$ (with a fixed extension), and instead of choosing a new basepoint for each Lagrangian, and paths from this basepoint to the Hamiltonian chords, we use the ones already chosen on $M$. 
\begin{rk}
Contrary to before, $h_{\tilde L}$ may not be well-defined for another compact Lagrangian $\tilde L$, as the convergence of defining formulae is not guaranteed.
\end{rk}
Denote this expanded $A_\infty$-category over $\Lambda$ by $\widetilde{\cF}(M)$ (we have used this notation previously with a similar, but different meaning). In particular, its set of objects are pairs $(L,\xi_L)$, where $\xi_L$ is a $\Lambda^*$-local system on \emph{$M$}. We assume the Floer and perturbation data does not depend on the local systems on Lagrangians. The family $\fM^M$ extends to a family-- still denoted by $\fM^M$-- over $\widetilde{\cF}(M)$ in the obvious way, and the same formulae \eqref{eq:grouplikeformula} defines a map 
\begin{equation}\label{eq:extendedgrouplikemap}
	\pi_2^*\fM^M\otimes_{\widetilde{\cF}(M)}^{rel} \pi_1^*\fM^M\to m^*\fM^M
\end{equation}
similar to \eqref{eq:grouplikemap}. 

$\Lambda^*$-local systems on $M$ correspond to $\Lambda$-points of $H^1(M,\bG_m)$. As extension of unitary case before, denote the rank $1$, $\Lambda^*$-local system on $M$ by $\xi_z$. By the assumption on Floer and perturbation data, for each $z$, one has a strict auto-equivalence $\Phi_z$ of $\widetilde{\cF}(M)$ that sends $(L,\xi_{L})$ to $(L,\xi_L\otimes\xi_z )$. This should not be confused with $\phi_z$, defined before, although our eventual goal is to show they act on the Fukaya category in the same way. It is immediate that the bimodule $\cB(\cdot, \Phi_z(\cdot))=_{\Phi_z}\cB$ is the same as $\fM^M|_z$ (generators identify and the structure maps are defined by the same count). Moreover, the map 
\begin{equation}
	\fM^M|_{z_2}\otimes_{\widetilde{\cF}(M)} \fM^M|_{z_1}\to \fM^M|_{z_1z_2}	
\end{equation}
obtained by restricting \eqref{eq:extendedgrouplikemap} coincide with the map \eqref{eq:psiphitocomposition} for $\Phi=\Phi_{z_1}$ and $\Psi=\Phi_{z_2}$. Indeed, after one identifies the generators, to apply \eqref{eq:psiphitocomposition} to a cochain 
\begin{equation}
	(x_1,\dots x_k|m\otimes x_1''\otimes\dots x_p''\otimes m'|  x_l',\dots, x_1')
\end{equation}
one first applies $\Psi=\Phi_{z_2}$ to $x_1,\dots x_k$. This does not change the corresponding chord, but modifies the structure maps that these are plugged into. Then one applies $\Phi=\Phi_{z_1}$ to $x_1,\dots x_k,m, x_1'',\dots, x_p''$ and the final effect is a count of discs as in \Cref{figure:grouplikeformula}, where $L_{j_0},\dots, L_{j_e}$ components are modified by $\xi_{z_2}$ and the components on the left of $m_1$ (i.e. those other than $L_{i_0},\dots, L_{i_d}$) are modified by $\xi_{z_1}$ ($L_{j_0},\dots, L_{j_e}$ components are modified twice; hence, modified by $\xi_{z_1z_2}$, but this is not relevant here). The relative homotopy class of the boundary in the counterclockwise direction from $m'$, resp. $m$, input to output is $[\partial_1u]$, resp. $[\partial_2u]$. Hence, the counts defining \eqref{eq:psiphitocomposition} and \eqref{eq:extendedgrouplikemap} agree in this specific case. 

Combining this with \Cref{lem:composedgraphbimodules}, we see that the restriction of \eqref{eq:extendedgrouplikemap} to $(z_1,z_2)$ is a quasi-isomorphism at every closed point $(z_1,z_2)$. 

To conclude the proof of Proposition \ref{prop:grouplike1}, we need the following lemma (c.f. Proposition \ref{prop:semicontinuityoverbase}):
\begin{lem}\label{lem:pointwisevanishing}
Let $C$ be a bounded complex of coherent sheaves over a smooth affine variety $V$ over $\Lambda$ such that for every closed point $x\in V$, $Li_x^*C$ is acyclic. Then $C$ is acyclic. 
\end{lem}
\begin{proof}
Denote the hypercohomology groups of $C$ by $\cH^q$. We show $\cH^q=0$ for all $q$. First, given a coherent sheaf $\scrF$ on $V$ (i.e. a module over $\cO(V)$), if the ordinary restriction $i^*_x\scrF=0$ for a closed point $x$, then $Li_x^*\scrF=0$. To see this, first observe by Nakayama's lemma $i^*_x\scrF=0$ implies $\scrF$ is zero on the local ring of $x$ (or on a small neighborhood $U$ of $x$). This clearly implies $Li_x^*\scrF=0$. 

Given closed point, there is a spectral sequence with $E_2$-page $L_pi_x^*\cH^{-q}\Rightarrow L_{p+q}i_x^*C=\cH^{-p-q}(Li_x^*C)=0$, with differentials $L_pi_x^*\cH^{-q}\to L_{p-2}i_x^*\cH^{-q-1}$. By the observation above, given $q$, if $i_x^*\cH^{-q}=L_0i_x^*\cH^{-q}=0$, then $L_pi_x^*\cH^{-q}=0$ for all $p$. Assume $L_0i_x^*\cH^{-q_0}\neq 0$, and $q_0$ is the minimal such. Then, $E_2$-page is $0$ below the line $q=q_0$, and $L_0i_x^*\cH^{-q_0}$ survives to infinity page. Since, the spectral sequence converges to $0$, $i_x^*\cH^{-q_0}=L_0i_x^*\cH^{-q_0}=0$, which is a contradiction. Therefore, $L_0i_x^*\cH^{-q}=0$ for all $q$ (and for all $x$ as $x$ was arbitrary).

A coherent sheaf that vanishes at every closed point is $0$. 
Hence, $\cH^q=0$ for all $q$.
\end{proof}
\begin{proof}[Proof of Proposition \ref{prop:grouplike1}]
$\widetilde{\cF}(M)$ is a proper category expanding the smooth category $\cF(M)$. As any proper module over a smooth category is perfect (this is a simpler version of Lemma \ref{lem:properimpliesperfect}), Yoneda modules of objects of $\widetilde{\cF}(M)$ are perfect over $\cF(M)$. In particular, $\widetilde{\cF}(M)$ is split generated by $\cF(M)$, and 
\begin{equation}
	\pi_2^*\fM^M\otimes_{\cF(M)}^{rel} \pi_1^*\fM^M\simeq \pi_2^*\fM^M\otimes_{\widetilde{\cF}(M)}^{rel} \pi_1^*\fM^M
\end{equation}
Moreover, under this identification, the maps \eqref{eq:grouplikemap} and \eqref{eq:extendedgrouplikemap} coincide, and $cone(\ref{eq:grouplikemap})\simeq cone(\ref{eq:extendedgrouplikemap})$ (as families of $\cF(M)$-bimodules, or as families of $\widetilde{\cF}(M)$-bimodules if one considers $cone(\ref{eq:grouplikemap})$ as a family of $\widetilde{\cF}(M)$-bimodules). 

As observed above, 	\eqref{eq:extendedgrouplikemap} is a quasi-isomorphism at every closed point $(z_1,z_2)$; hence, $cone(\ref{eq:extendedgrouplikemap})|_{(z_1,z_2)}\simeq 0$. The families $\pi_1^*\fM^M$ and $\pi_2^*\fM^M$ are proper; hence, by Lemma \ref{lem:properimpliesperfect}, they are perfect. This implies $\pi_2^*\fM^M\otimes_{\cF(M)}^{rel} \pi_1^*\fM^M$ is proper by Lemma \ref{lem:perfconvperfisperf}. Therefore,  $cone(\ref{eq:extendedgrouplikemap})$ is proper, and $cone(\ref{eq:extendedgrouplikemap})(L,L')$ has coherent, bounded cohomology over $H^1(M,\bG_m)\times H^1(M,\bG_m)$, for any pair $(L,L')$ of objects. In particular, $cone(\ref{eq:extendedgrouplikemap})(L,L')$ is quasi-isomorphic to a bounded complex of coherent sheaves $C$, and by Lemma \ref{lem:pointwisevanishing}, it is acyclic. Since this holds for every $(L,L')$, $cone(\ref{eq:extendedgrouplikemap})\simeq 0$. This concludes the proof.
\end{proof}
We will combine Lemma \ref{lem:hlgeometricnearby}, Lemma \ref{lem:hlconvmtohlsemicont} and Proposition \ref{prop:grouplike1}, to conclude that the family $\fM^M$ is ``geometric'', i.e. $\fM^M|_z$ represents $\phi_z$.

\subsection{The adjustments for the Weinstein case}\label{subsec:exactcase}
In this section, we explain how to extend the results of the previous chapter to exact symplectic manifolds. We will restrict ourselves to Weinstein manifolds, although the results presumably generalize to non-degenerate Liouville manifolds with slightly different techniques. We need to explain how to construct the families, and establish \Cref{lem:hlgeometricnearby}, \Cref{lem:hlconvmtohlsemicont} and \Cref{prop:grouplike1} in this setting. Throughout this section, assume $M$ is Weinstein. 
Therefore, $\cW(M)$ is a smooth category, as stated. Also, throughout this section, let $G$ denote the algebraic group $H^1(M,\bG_m)=Spec(\Lambda[z^{H_1(M)}])$ (where $H_1(M)$ means first homology modulo torsion). 

The family $\fM^M$ can be constructed as before. More precisely, one can define the family $\fM^M$ exactly in the same way on $\cO(M,\sigma)$. Then, if $C$ denotes the set of continuation elements, then the right localization $\fM^M_{C^{-1}}$ (or the left localization $_{C^{-1}}\fM^M$) descends to a family of bimodules over $\cW(M,\sigma)$. This can be shown using the techniques of \cite{GPS1}. One can also define $\fM^M$ in the same way for the models of $\cW(M)$ given in \cite{abousei,generation}. Similarly, given compact, tautologically unobstructed Lagrangian brane $\tilde L$, one can define $h_{\tilde L}^{an}$ exactly as before. Analogous to \Cref{lem:hlgeometricnearby} we have 
\begin{lem}\label{lem:hlgeometricnearbywrapped}
For small $P\subset H^1(M,\bR)$, $h_{\tilde L}^{an}$ is well-defined and given $z\in S_P$, $h_{\tilde L}^{an}|_z\simeq h_{\phi_z(\tilde L)}$.
\end{lem} 
Hence, $h_{\tilde L}^{an}$ is an analytic family of right modules parametrized by $S_P$. The proof of \Cref{lem:hlgeometricnearby} applies verbatim, we include a separate statement for cross referencing purposes.

The main obstruction to establish \Cref{lem:hlconvmtohlsemicont} and \Cref{prop:grouplike1} is that the family $\fM^M$ is not necessarily proper; therefore, we cannot apply \Cref{lem:properimpliesperfect} to show perfectness. Moreover, one needs the properness of both sides of \eqref{eq:grouplikemap} to apply \Cref{lem:pointwisevanishing} to conclude \eqref{eq:grouplikemap} is a quasi-isomorphism. This does not hold. Indeed, in the presence of a stop $\fstop$ such that $\cW(M,\fstop)$ is proper, the proofs of \Cref{lem:hlconvmtohlsemicont} and \Cref{prop:grouplike1} works verbatim. 

We use this fact to deduce the same result on $\cW(M)$ as follows: by \cite{girouxpardon}, one can endow $M$ with the structure of a Lefschetz fibration. Let $\fstop$ denote the corresponding fiber (pushed to infinity, so it is a stop). In particular, $\cW(M,\fstop)$ is smooth and proper and there exists a quotient map $\cW(M,\fstop)\to \cW(M,\fstop)/\cD\simeq \cW(M)$ (this is an example of smooth categorical compactification as in \cite{efimovhfp}, also see \cite{ownwithcotefiltrationgrowth}). Here, $\cD$ can be taken to be a finite collection of Lagrangian discs (the linking discs of the core of $\fstop$). 
There exists an algebraic family $\fM^M_{\cW(M,\fstop)}$ of bimodules over $\cW(M,\fstop)$, defined exactly in the same way, and an analytic family $h_{\tilde L}^{an}$. Note that the family $h_{\tilde L}^{an}$ over $\cW(M,\fstop)$ can be obtained from the family $h_{\tilde L}^{an}$ over $\cW(M)$ by restriction along the stop removal functor $\cW(M,\fstop)\to\cW(M)$. This is true as $\tilde L$ is compact. In particular, $h_{\tilde L}^{an}(D)$ is acyclic for all $D\in \cD$.

Observe $\fM^M_{\cW(M,\fstop)}$ can seen as the deformation of the diagonal bimodule obtained by endowing one Lagrangian with ``the universal rank-$1$'' local system (these local systems are parametrized by $G$). Consider the right quotient $\fM^M_{\cW(M,\fstop)}/\cD$, which is a family of $\cW(M,\fstop)$-$\cW(M)$-bimodules. See \cite{GPS1} for a review of quotients and localization (and note the slight change in composition conventions). Similarly, this family is the deformation of $\cW(M,\fstop)/\cD$ (as a $\cW(M,\fstop)$-$\cW(M)$-bimodule) obtained by endowing one input with the universal rank-$1$ local system on $M$. By definition of localization, $(\cW(M,\fstop)/\cD)(D,\cdot)$ and $(\fM^M_{\cW(M,\fstop)}/\cD)(D,\cdot)$ are acyclic for every $D\in\cD$, and $\fM^M_{\cW(M,\fstop)}/\cD$ can be identified with $\fM^M$. The latter is a priori a family of $\cW(M)$-$\cW(M)$-bimodules, but in the last sentence, it is considered as a family of $\cW(M,\fstop)$-$\cW(M)$-bimodules. It follows that $(\cW(M,\fstop)/\cD)(\cdot,D)$ and its deformation $(\fM^M_{\cW(M,\fstop)}/\cD)(\cdot, D)$ are acyclic too (for each $D\in\cD$). In particular, they both descend to $\cW(M,\fstop)/\cD$-$\cW(M,\fstop)/\cD$-bimodules. 

Similarly, $\cD\backslash\fM^M_{\cW(M,\fstop)}$ descends to a family of $\cW(M,\fstop)/\cD$-$\cW(M,\fstop)/\cD$-bimodules. Therefore, one has 
\begin{equation}
\cD\backslash\fM^M_{\cW(M,\fstop)}\simeq \cD\backslash\fM^M_{\cW(M,\fstop)}/\cD\simeq \fM^M_{\cW(M,\fstop)}/\cD
\end{equation}
and this is also quasi-isomorphic to $\fM^M$ as remarked. Now we are ready to prove
\begin{lem}\label{lem:hlconvmtohlsemicontwrapped}
There exists $P\subset H^1(M,\bR)$ such that $h_{\tilde L}\otimes_{\cW(M)}\fM^M|_{S_P}\simeq h_{\tilde L}^{an}$.
\end{lem}
\begin{proof}
As remarked, this statement holds over $\cW(M,\fstop)$, i.e. $h_{\tilde L}\otimes_{\cW(M,\fstop)}\fM^M_{\cW(M,\fstop)}|_{S_P}\simeq h_{\tilde L}^{an}$. Take the right quotient by $\cD$ on both sides. The left hand side becomes $h_{\tilde L}\otimes_{\cW(M,\fstop)}(\fM^M_{\cW(M,\fstop)}/\cD)|_{S_P}$. By remarks above $(\fM^M_{\cW(M,\fstop)}/\cD)|_{S_P}\simeq (\cD\backslash \fM^M_{\cW(M,\fstop)}/\cD)|_{S_P}$ is a family of $\cW(M,\fstop)/\cD$-$\cW(M,\fstop)/\cD$-bimodules, and similarly for $h_{\tilde L}$. Hence, by \cite[Lemma 3.15]{GPS1},
\begin{equation}
h_{\tilde L}\otimes_{\cW(M,\fstop)}(\fM^M_{\cW(M,\fstop)}/\cD)|_{S_P}\simeq h_{\tilde L}\otimes_{(\cW(M,\fstop)/\cD)}(\fM^M_{\cW(M,\fstop)}/\cD)|_{S_P}\simeq h_{\tilde L}\otimes_{\cW(M)}\fM^M|_{S_P}
\end{equation}
Similarly, the right hand side becomes $h_{\tilde L}^{an}/\cD$. On the other hand, $h_{\tilde L}^{an}$ descend to the same named family over $\cW(M)$; therefore, $h_{\tilde L}^{an}\simeq h_{\tilde L}^{an}/\cD$. This finishes the proof.   
\end{proof}
Using similar methods, we have:
\begin{prop}\label{prop:grouplikewrapped}
The family $\fM^M$ over $\cW(M)$ is group-like. 
\end{prop}
\begin{proof}
As remarked above, the proof of \Cref{prop:grouplike1} works verbatim and one has 
\begin{equation}
	\pi_2^*\fM^M_{\cW(M,\fstop)}\otimes_{\cW(M,\fstop)}^{rel} \pi_1^*\fM^M_{\cW(M,\fstop)}\simeq m^*\fM^M_{\cW(M,\fstop)}
\end{equation}
where the quasi-isomorphism is defined as in \eqref{eq:grouplikemap}. Take the quotient by $\cD$ on both sides. The left hand side becomes 
$
		\pi_2^*(\cD\backslash\fM^M_{\cW(M,\fstop)})\otimes_{\cW(M,\fstop)}^{rel} \pi_1^*(\fM^M_{\cW(M,\fstop)}/\cD)
$. As above, both sides are bimodules over $\cW(M,\fstop)/\cD$, and one can apply \cite[Lemma 3.15]{GPS1} to show 
\begin{equation}
			\pi_2^*(\cD\backslash\fM^M_{\cW(M,\fstop)})\otimes_{\cW(M,\fstop)}^{rel} \pi_1^*(\fM^M_{\cW(M,\fstop)}/\cD)\simeq 		\pi_2^*(\cD\backslash\fM^M_{\cW(M,\fstop)})\otimes_{(\cW(M,\fstop)/\cD)}^{rel} \pi_1^*(\fM^M_{\cW(M,\fstop)}/\cD)\simeq \atop
					\pi_2^*\fM^M\otimes_{\cW(M)}^{rel} \pi_1^*\fM^M
\end{equation}
Similarly, the right hand side becomes $m^*(\cD\backslash \fM^M_{\cW(M,\fstop)}/\cD)\simeq m^*\fM^M$. Therefore, 
\begin{equation}
						\pi_2^*\fM^M\otimes_{\cW(M)}^{rel} \pi_1^*\fM^M\simeq m^*\fM^M
\end{equation}
This finishes the proof. 
\end{proof}
\Cref{prop:grouplikewrapped}, together with \Cref{lem:grouplikeimpliesperfect} imply:
\begin{cor}\label{lem:perfectintheexactcase}
The family $\fM^M$ is perfect.
\end{cor}

\section{Algebraicity of the Floer homology sheaf}\label{sec:algfloersheaf}
In this section, we prove the main abstract \namecref{thm:mainabstract}, \Cref{thm:mainabstract}, by combining Lemma \ref{lem:hlgeometricnearby}, Lemma \ref{lem:hlconvmtohlsemicont} and Proposition \ref{prop:grouplike1} (resp. Lemma \ref{lem:hlgeometricnearbywrapped}, Lemma \ref{lem:hlconvmtohlsemicontwrapped} and Proposition \ref{prop:grouplikewrapped}). We use this to deduce \Cref{thm:algsheaf}. To avoid repetition, we only show the proof for compact $M$. The proof is identical in the exact case. Recall the statements first:
\begingroup
\def\thethm{\ref*{thm:algsheaf}}
\begin{thm}
Let $L,L'\subset M$ be tautologically unobstructed closed Lagrangian branes. Then, there exists a finite complex of algebraic coherent sheaves over $H^1(M,\bG_m)$ whose restriction at $z\in H^1(M,\bG_m)$ has cohomology isomorphic to $HF(L,\phi_z(L') )$. 
\end{thm}
\addtocounter{thm}{-1}
\endgroup
\begingroup
\def\thethm{\ref*{thm:mainabstract}}
\begin{thm}
Given tautologically unobstructed Lagrangian brane $\tilde{L}$ 
\begin{equation}
	h_{\tilde{L}}\otimes_{\cF(M)} \fM^M|_z\simeq h_{\phi_z(\tilde{L})}
\end{equation}	
for all $z\in H^1(M,\bG_m)$.
\end{thm}
\addtocounter{thm}{-1}
\endgroup
The algebraic sheaf mentioned in \Cref{thm:algsheaf} is quasi-isomorphic to $h_{L'}\otimes_{\cF(M)} \fM^M \otimes_{\cF(M)} h^L$. The algebraicity of will follow from algebraicity of the family $\fM^M$, and one point of the proof is its coherence. On the other hand, to relate it to Floer homology, we combine \Cref{thm:mainabstract} with \Cref{lem:hlconvhlequalshom}.
\begin{proof}[Proof of \Cref{thm:mainabstract}]
For simplicity, assume $z=T^{[\alpha]}$, for some closed $1$-form $\alpha$, and for cleaner notation denote $\phi_{T^{t[\alpha]}}$ by $\phi^t_\alpha$, as before. Then, $\phi^t_{\alpha}(\tilde{L}), t\in[0,1]$ defines a Lagrangian isotopy from $\tilde{L}$ to $\phi_z(\tilde{L})$. By Lemma \ref{lem:hlgeometricnearby} and Lemma \ref{lem:hlconvmtohlsemicont}, given $\phi^t_\alpha(\tilde{L})$, there exists a small $P_t\subset H^1(M,\bR)$ such that 
\begin{equation}
	h_{\phi^t_\alpha(\tilde{L})}\otimes_{\cF(M)} \fM^M|_{T^{[\beta]}}\simeq h_{\phi^t_{\alpha+\beta}(\tilde{L})}
\end{equation}
for every $[\beta]\in P_t$ (i.e. $z'=T^{[\beta]}\in S_{P_t}$). The compact interval $[0,val_T(z)]=[0,[\alpha]]\subset H^1(M,\bR)$ is covered by the interiors of $t[\alpha]+P_t$; thus, it admits a finite subcover, and there exists $0=t_0<t_1<\dots <t_p=1$ such that every $[t_{i-1}[\alpha],t_i[\alpha]]$ is contained in one of $t[\alpha]+P_t$. 
Hence, \begin{equation}\label{eq:inproofhlconvsmall}
	h_{\phi^t_\alpha(\tilde{L})}\otimes_{\cF(M)} \fM^M|_{T^{(t_{i-1}-t)[\alpha]}}\simeq h_{\phi^{t_{i-1}}_\alpha(\tilde{L})}
\end{equation}
and the same with $t_{i-1}$ replaced by $t_i$. (\ref{eq:inproofhlconvsmall}) also implies
\begin{equation}
	h_{\phi^{t_{i-1}}_\alpha(\tilde{L})}\otimes_{\cF(M)} \fM^M|_{T^{(t-t_{i-1})[\alpha]}}\simeq h_{\phi^{t}_\alpha(\tilde{L})}	
\end{equation}
by the group-like property of $\fM^M$. Hence,
\begin{align}
	h_{\phi^{t_i}_\alpha(\tilde{L})}\simeq 	h_{\phi^t_\alpha(\tilde{L})}\otimes_{\cF(M)} \fM^M|_{T^{(t_i-t)[\alpha]}}\simeq \\	h_{\phi^{t_{i-1}}_\alpha(\tilde{L})}\otimes_{\cF(M)} \fM^M|_{T^{(t-t_{i-1})[\alpha]}} \otimes_{\cF(M)} \fM^M|_{T^{(t_i-t)[\alpha]}}\simeq \\
	h_{\phi^{t_{i-1}}_\alpha(\tilde{L})}\otimes_{\cF(M)} \fM^M|_{T^{(t_i-t_{i-1})[\alpha]}} 
\end{align}
where the last identity also follows from the group-like property. Therefore,
\begin{equation}
h_{\phi^{t_p}_\alpha(\tilde{L})}\simeq h_{\phi^{t_0}_\alpha(\tilde{L})} \otimes_{\cF(M)} \fM^M|_{T^{(t_1-t_0)[\alpha]}}\otimes_{\cF(M)}  \dots \otimes_{\cF(M)} \fM^M|_{T^{(t_p-t_{p-1})[\alpha]}} \simeq \atop
h_{\phi^{t_0}_\alpha(\tilde{L})} \otimes_{\cF(M)} \fM^M|_{T^{(t_p-t_0)[\alpha]}}
\end{equation}
The last quasi-isomorphism is again due to group-like property. This is what the \namecref{thm:mainabstract} asserts, and it finishes the proof.
\end{proof}
\begin{rk}\label{rk:convleft}
It is also true that\begin{equation}
		\fM^M|_z \otimes_{\cF(M)} h^{\tilde{L}}\simeq h^{\phi_{-z}(\tilde{L})}
\end{equation}	which follows similarly. More precisely, one needs left module versions of Lemma \ref{lem:hlgeometricnearby}, Lemma \ref{lem:hlconvmtohlsemicont}, Lemma \ref{lem:hlgeometricnearbywrapped} and Lemma \ref{lem:hlconvmtohlsemicontwrapped}, whose proofs are very similar to given versions.
\end{rk}
\begin{proof}[Proof of Theorem \ref{thm:algsheaf}]
$\fM^M$ is a perfect family of bimodules (directly implied by \Cref{lem:properimpliesperfect} in the compact case, for the exact case, see \Cref{lem:perfectintheexactcase}).

As $h_{L'}$ and $h^L$ are proper modules, the complex $h_{L'}\otimes_{\cF(M)} \fM^M \otimes_{\cF(M)} h^L$ is quasi-isomorphic to a finite complex of coherent sheaves (it is algebraic by construction, it can be seen as a complex of $\cO(H^1(M,\bG_m))=\Lambda[z^{H_1(M)}]$-modules). 

By \Cref{thm:mainabstract}, $h_{L'}\otimes_{\cF(M)} \fM^M|_z$ is quasi-isomorphic to $h_{\phi_z(L')}$. Thus, the restriction of $h_{L'}\otimes_{\cF(M)} \fM^M \otimes_{\cF(M)} h^L$ at $z\in H^1(M,\bG_m)$ is quasi-isomorphic to 
\begin{equation}
h_{\phi_z(L')}\otimes_{\cF(M)} h^L\simeq CF(L,\phi_z(L'))	
\end{equation}
The last quasi-isomorphism follows from \Cref{lem:hlconvhlequalshom}. This completes the proof.
\end{proof}
Corollary \ref{cor:zariskiopenrank} and Corollary \ref{cor:rankstratification} now follows from generalities on complexes of coherent sheaves. More precisely, let $(C,d)$ be a finite, algebraic complex of vector bundles $H^1(M,\bG_m)$ (i.e. a finite complex of projective modules over $\Lambda[z^{H_1(M)}]$) that is quasi-isomorphic to $h_{L'}\otimes_{\cF(M)} \fM^M \otimes_{\cF(M)} h^L$. Then, at every $z\in H^1(M,\bG_m)$, the rank of $H^*(C|_z,d_z)$ is equal to $dim (C)-2rank(d_z)$. One can assume $C$ is graded free, and by trivializing the it, one can see $d$ as a matrix with coefficients in $\Lambda[z^{H_1(M)}]$. The locus where $rank(d_z)>\frac{1}{2}(dim(C)-k)$ is a Zariski open set, and the locus of points where $dim (C)-2rank(d_z)\geq k$ is its complement. This proves Corollary \ref{cor:rankstratification}. The minimal rank locus forms the Zariski open set that is mentioned in Corollary \ref{cor:zariskiopenrank}. 

To prove Corollary \ref{cor:realiterates}, we have to restrict this algebraic sheaf to ``the real line'' $\{T^{t[\alpha]}:t\in\bR \}\subset H^1(M,\bG_m)$. Consider the ring 
\begin{equation}
	\Lambda[z^\bR]=\{\sum a_rz^r:a_r\in\Lambda,\text{ the sum is finite} \}
\end{equation}
and base change $(C,d)$ to this ring, along the map $	\Lambda[z^{H_1(M)}]\to \Lambda[z^\bR]$ that sends $z^{[C]}$ to $z^{\alpha([C])}$. We obtain a finite complex of free $\Lambda[z^\bR]$-modules whose restriction to ``$z=T^t$'' (i.e. its base change along $\Lambda[z^\bR]\to\Lambda$, $z^r\mapsto T^{tr}$) is quasi-isomorphic to $CF(L,\phi_\alpha^t(L'))$. Denote this complex by $(C',d')$. We claim the rank of cohomology of this complex restricted to $z=T^t$ is constant in $t$, with finitely many exceptions. We first need:
\begin{lem}\label{lem:finitezeroes}
Let $f\in \Lambda[z^\bR]$ be non-zero. Then $f(T^t)=0$ holds only for finitely many $t\in \bR$.
\end{lem}
\begin{proof}
Let $f(z)=a_{r_1}z^{r_1}+\dots+ a_{r_k}z^{r_k}$, where $a_{r_i}\neq 0$ and $r_1<\dots< r_k$. Given $t$, $val_T(a_{r_i}T^{tr_i})=val_T(a_{r_i})+tr_i$. Thus, when ``$z=T^t$'' is plugged in for $t\gg 0$, $a_{r_1}z^{r_1}$ term has strictly smaller valuation than the other terms, and it is non-zero; hence, $f(T^t)\neq 0$. Similarly, for $t\ll 0$, $a_{r_k}z^{r_k}$ term has strictly smaller valuation and $f(T^t)\neq 0$. In other words, the zeroes of $t\mapsto f(T^t)$ form a bounded subset of $\bR$. 

Now, assume $t_0$ be a zero of this function, i.e. $f(T^{t_0})=0$. Let $I\subset\{1,\dots, k\}$ denote the set of $i$ such that $a_{r_i}T^{t_0r_i}$ terms has minimal valuation. There is a positive valuation gap between $a_{r_i}T^{t_0r_i}$ terms for $i\in I$ and the remaining terms, and this property persists for small deformations of $t_0$. In other words, there exists an $\epsilon>0$ such that for any $t\in (t_0-\epsilon,t_0+\epsilon)$, $a_{r_i}T^{tr_i}$ has also strictly smaller valuation than $a_{r_j}T^{tr_j}$, if $i\in I$, $j\not\in I$. On the other hand, if $t\neq t_0$, then $val_T(a_{r_i}T^{tr_i})$ is no longer equal to $val_T(a_{r_j}T^{tr_j})$, for $i,j\in I$, $i\neq j$ (as $r_i\neq r_j$). As these terms have pairwise different and smaller valuation than those for $i\not\in I$, $f(T^t)\neq 0$. Thus, $t_0$ is the only zero in $(t_0-\epsilon, t_0+\epsilon)$, and the set of zeroes is also discrete. As it is also bounded, this set is finite.
\end{proof}
The set of $t\in \bR$, where cohomology of the complex $(C'|_{T^t}, d'_{T^t})$ has rank larger than minimal (in $t$) is the same as the set of $t\in\bR$ such that $rank(d'_{T^t})$ is less than the maximal possible (among other $rank(d'_{T^t})$, not only less than full rank). This set is given by vanishing of finitely many minors of the matrix corresponding to $d'$. Hence, by Lemma \ref{lem:finitezeroes}, this set is finite (as its compliment is non-empty). This implies the minimal rank locus of the cohomology of $(C'|_{T^t}, d'_{T^t})$ is a cofinite subset of $\bR$, and this finishes the proof of \Cref{cor:realiterates}.

\section{Algebraic stabilizers and the flux groups of Lagrangians}\label{sec:stabilizer}
We now turn to the proof of:
\begingroup
\def\thethm{\ref*{thm:stabilizer}}
\begin{thm} Let $L$ be a tautologically unobstructed, closed Lagrangian brane in $M$. Then the set of $z\in H^1(M,\bG_m)$ such that $L$ is stably isomorphic to $\phi_z(L)$ form an algebraic subtorus of $H^1(M,\bG_m)$ whose Lie algebra is given by the kernel of the map $H^1(M,\Lambda)\to H^1(L,\Lambda)$.
\end{thm}
\addtocounter{thm}{-1}
\endgroup
The idea for algebraicity of this set is simple: $L$ and $\phi_z(L)$ are stably isomorphic if and only if the composition maps
\begin{align}\label{eq:firstcomposition}	HF(\phi_z(L),L)\otimes HF(L,\phi_z(L))\to HF(L,L)\\
\label{eq:secondcomposition}	HF(L,\phi_z(L))\otimes HF(\phi_z(L),L)\to HF(\phi_z(L),\phi_z(L))
\end{align}
both hit the respective units. We have shown $HF(L,\phi_z(L))$ can be recovered as the the cohomology of an algebraic sheaf at $z$, and the same holds for $HF(\phi_z(L),L)$ and $HF(\phi_z(L),\phi_z(L))$. One can also show the composition map is a map of algebraic sheaves. Therefore, the locus of $z$ such that these maps hit the unit form a constructible set in Zariski topology. But, this set is also a subgroup; therefore, it is a Zariski closed subgroup. 

Now, we would like to elaborate on the steps. First, (\ref{eq:secondcomposition}) hits the unit if and only if the map
\begin{equation}\label{eq:secondcompositionmodified}
		HF(\phi_{z^{-1}}(L),L)\otimes HF(L,\phi_{z^{-1}}(L))\to HF(L,L)
\end{equation}
hits the unit. (\ref{eq:secondcompositionmodified}) can be obtained from (\ref{eq:firstcomposition}) by a parameter change $z\mapsto z^{-1}$. Therefore, if the locus of points hitting the unit under (\ref{eq:firstcomposition}) is constructible, then the same holds for the locus corresponding to (\ref{eq:secondcompositionmodified}), and the intersection of two loci is still constructible. Hence, it suffices to show that the locus of points such that (\ref{eq:firstcomposition}) hits the unit is constructible. 

We have shown that $HF(L,\phi_z(L))$ is the cohomology of the restriction of the complex $h_L\otimes_{\cF(M)} \fM^M\otimes_{\cF(M)} h^L$ to $z\in H^1(M,\bG_m)$. Similarly, if one defines $\fM^{M,-}$ to be the family over $H^1(M,\bG_m)$ by replacing $z$ in (\ref{eq:bimodulediff}) and (\ref{eq:bimodulestructure}) by $z^{-1}$ (in other words, by a parameter change $z\mapsto z^{-1}$), then the cohomology of the complex $h_L\otimes_{\cF(M)} \fM^{M,-}\otimes_{\cF(M)} h^L$ at $z$ is $HF(\phi_z(L),L)$. Indeed, Remark \ref{rk:convleft} implies that 
\begin{equation}
h^{\phi_z(L)}\simeq \fM^{M,-}|_z\otimes_{\cF(M)} h^L	
\end{equation}
Recall from Lemma \ref{lem:compositionbymiddlecontraction} that the composition 
\begin{equation}
	h_L\otimes_{\cF(M)} h^{\phi_z(L)}\otimes h_{\phi_z(L)}\otimes_{\cF(M)} h^L\simeq CF(\phi_z(L),L)\otimes CF(L,\phi_z(L))\to\atop CF(L,L)\simeq h_L\otimes_{\cF(M)} h^L\simeq h_L\otimes_{\cF(M)} \cF(M)\otimes_{\cF(M)} h^L
\end{equation}
is obtained (up to homotopy) by contracting the middle components $h^{\phi_z(L)}\otimes h_{\phi_z(L)}$ via the bimodule map
\begin{equation}\label{eq:hphiconvhphi}
h^{\phi_z(L)}\otimes h_{\phi_z(L)}\to \cF(M)
\end{equation}
given by (\ref{eq:midcontraction}) (here $\cF(M)$ denotes the diagonal bimodule). The left hand side of \eqref{eq:hphiconvhphi} is quasi-isomorphic to
\begin{equation}\label{eq:mhlhlm}
\fM^{M,-}|_z\otimes_{\cF(M)} h^L	\otimes h_L\otimes_{\cF(M)} \fM^M|_z	
\end{equation}
As a result, \eqref{eq:hphiconvhphi} is equivalent to a map from \eqref{eq:mhlhlm} to $\cF(M)$. We have
\begin{lem}\label{lem:compositionforbimodulesatz}
The map from (\ref{eq:mhlhlm}) to $\cF(M)$ obtained by the composition of 
\begin{equation}\label{eq:multiplebimodulecomposition}
\fM^{M,-}|_z\otimes_{\cF(M)} h^L	\otimes h_L\otimes_{\cF(M)} \fM^M|_z\to \fM^{M,-}|_z\otimes_{\cF(M)} \cF(M) \otimes_{\cF(M)} \fM^M|_z\simeq\atop \fM^{M,-}|_z\otimes_{\cF(M)} \fM^M|_z \to \cF(M)			
\end{equation}
where the arrow in the first line is given by applying the natural bimodule map $h^L\otimes h_L\to \cF(M)$, and the arrow in the second line is given by applying (\ref{eq:grouplikemap}) at $z_1=z,z_2=z^{-1}$ (i.e. the composition map $\fM^M|_{z^{-1}}\otimes_{\cF(M)} \fM^M|_z\to \cF(M)$).
\end{lem}
We skip the proof of this lemma. It is based on chasing the isomorphism constructed in the proof of \Cref{thm:mainabstract} (hence, also in \Cref{lem:hlgeometricnearby} and \Cref{lem:hlconvmtohlsemicont}). 

Therefore, one can write a family version of (\ref{eq:firstcomposition}) as the composition
\begin{equation}\label{eq:familybimodulecomposition}
	(\fM^{M,-}\otimes_{\cF(M)} h^L	\otimes h_L)\otimes_{\cF(M)}^{rel} \fM^M\to (\fM^{M,-}\otimes_{\cF(M)} \cF(M)) \otimes_{\cF(M)}^{rel} \fM^M\simeq\atop \fM^{M,-}\otimes_{\cF(M)}^{rel} \fM^M \to \underline{\cF(M)}			
\end{equation}	
Here the superscript $rel$ indicates the tensor product is taken relative to $H^1(M,\bG_m)$, and $\underline{\cF(M)}$ denotes the constant family of bimodules parametrized by $H^1(M,\bG_m)$ that restricts to $\cF(M)$ at points (we put the underline once to avoid confusion, but we will mostly omit this from the notation). Hence, (\ref{eq:familybimodulecomposition})  involves families over $H^1(M,\bG_m)$ only (i.e. not over $H^1(M,\bG_m)\times H^1(M,\bG_m)$ etc.). The family $\fM^{M,-}\otimes_{\cF(M)}^{rel} \fM^M$ can be obtained by restricting $\pi_2^*\fM^M\otimes_{\cF(M)}^{rel} \pi_1^*\fM^M$ to the subvariety $\{z_1=z,z_2=z^{-1}\}\subset H^1(M,\bG_m)\times H^1(M,\bG_m)$, and the bottom map in (\ref{eq:familybimodulecomposition}) is the restriction of (\ref{eq:grouplikemap}) to this subvariety. 

Hence we have a morphism of families
\begin{equation}\label{eq:familybimodulecomposedcomposition}
	\fM^{M,-}\otimes_{\cF(M)} h^L	\otimes h_L\otimes_{\cF(M)} \fM^M\to  \underline{\cF(M)}			
\end{equation}
that restrict to (\ref{eq:multiplebimodulecomposition}) at $z\in H^1(M,\bG_m)$ (we omitted the superscript $rel$ as there was not a unique place it can be placed). One can define a (necessarily algebraic) map 
\begin{equation}\label{eq:familychaincomposition}
	h_L\otimes_{\cF(M)} \fM^{M,-}\otimes_{\cF(M)} h^L	\otimes h_L\otimes_{\cF(M)} \fM^M\otimes_{\cF(M)} h^L
	\to  \underline{h_L	\otimes_{\cF(M)} h^L}
\end{equation}
of chain complexes over $H^1(M,\bG_m)$ by contracting middle components via (\ref{eq:familybimodulecomposedcomposition}). The restriction of this map to $z\in H^1(M,\bG_m)$ is homotopic to composition map (\ref{eq:firstcomposition}) by Lemma (\ref{lem:compositionforbimodulesatz}) and the preceding remarks. The complex $\underline{h_L	\otimes_{\cF(M)} h^L}\simeq CF(L,L)\otimes \cO(H^1(M,\bG_m))$ admits a natural closed section $\textbf{1}$ corresponding to unit. Then, the map (\ref{eq:firstcomposition}) hits the unit if and only if the restriction of the map (\ref{eq:familychaincomposition}) at $z$ hits $\textbf{1}|_z $. Let $s$ denote the image of $\textbf{1}\in \underline{h_L	\otimes_{\cF(M)} h^L}$ considered as an element of the cone of (\ref{eq:familychaincomposition}). Then, (\ref{eq:familychaincomposition}) at $z$ hits $\textbf{1}|_z $ if and only if $s|_z$ vanishes in cohomology of the $cone(\ref{eq:familychaincomposition})$. The following lemma shows that this condition defines a constructible set:
\begin{lem}\label{lem:constructible}
Let $S$ be a smooth affine variety over $\Lambda$ and let $(C,d)$ be a finite complex of finite rank projective modules over $\cO(S)$. Assume $s\in C$ satisfy $d(s)=0$. Then the locus of points in $S$ such that $s|_x$ vanishes in $H^*(C|_x)$ form a constructible set.
\end{lem}
\begin{proof}
Assume without loss of generality that $s\in C^1$. Also, as $C$ is locally free, for simplicity we can assume it is free. Trivialize $C$ to consider $d_0:C^0\to C^1$ as a matrix. Then $s|_x\in Im(d_0|_x)$ if and only if the rank of the matrix $d_0$ is the same as the rank of the matrix $[d_0,s]$ at $x$. Here, $[d_0,s]$ is the matrix obtained by considering $s$ as a column matrix and placing on the right of $d_0$. The locus of points where these ranks are the same and $k$ is the set of points where all minors of $[d_0,s]$ of size $(k+1)$-vanish, but at least one minor of $d_0$ of size $k$ does not vanish. Hence, this set is locally closed, and the union of such sets over all $k$ is constructible. 
\end{proof}
We apply Lemma \ref{lem:constructible} to a finite projective replacement of the cohomologically proper complex $cone(\ref{eq:familychaincomposition})$. As remarked the locus of points where $(\ref{eq:secondcomposition})$ (equivalently (\ref{eq:secondcompositionmodified})) hits the unit can be obtained by the change of variable $z\mapsto z^{-1}$; hence, this set is also constructible. The intersection of two constructible sets is constructible. This intersection locus is the set of $z$ such that $L$ and $\phi_z(L)$ are stably isomorphic.
\begin{rk}
From a scheme theoretic perspective, we are considering the underlying reduced varieties of the locally closed subsets forming this intersection locus. 
\end{rk}
\begin{rk}
In the exact case, $\fM^M$ is not necessarily proper, but it is still perfect. This, together with compactness of $L$ (hence, properness of $h_L,h^L$) implies that \eqref{eq:familybimodulecomposedcomposition} is a proper bimodule, and \eqref{eq:multiplebimodulecomposition} is cohomologically finite. 
\end{rk}
Hence, we have shown that the locus of $z\in H^1(M,\bG_m)$ such that $L$ and $\phi_z(L)$ are stably isomorphic is a constructible set. It is clearly closed under multiplication and taking inverses; therefore, this locus is an abstract subgroup (to see this better, one has to consider the initial definition of stably isomorphic in terms of each object embedding into a direct sum of copies of the other). The following lemma shows that it is a closed subgroup:
\begin{lem}
Let $G$ be an affine algebraic group over $\Lambda$	and let $H\subset G$ be a (reduced) constructible set that form an abstract subgroup. Then, $H$ is an algebraic subgroup of $G$.
\end{lem}
\begin{proof}
$H$ is a finite union of locally closed sets $S_1,\dots ,S_l$, assume without loss of generality these sets are all irreducible, relatively closed in $H$ and none of them is contained in the union of others. As $H$ is closed under multiplication by its elements, for any $S_i$ and $x\in H$, $xS_i\subset H$. Hence, $xS_i\subset S_j$ for some $j$. If $xS_i\neq S_j$, this implies $S_i\neq x^{-1}S_j$, and the relatively closed, irreducible set $x^{-1}S_j$ is contained in one of $S_k$, $k\neq i$. This contradicts the initial assumption that $S_i\not\subset S_k$. Therefore, $xS_i=S_j$. In other words, the components $\{S_i\}$ are closed under multiplication by the elements of $H$. 

A generic point of $H$ belongs to exactly one of $S_i$ and is smooth. Any translation of this point by an element of $H$ satisfy the same property; hence, every point of $H$ is smooth and belongs to exactly one of $S_i$. In particular, these components $S_i$ are pairwise disjoint. Homogeneity of $H$ implies they are all translates of each other. Assume $S_1$ is the component such that $1\in S_1$. Being the connected component of the identity, $S_1$ is an abstract subgroup of $H$. 

Hence, we only need to check the locally closed, reduced, abstract subgroup $S_1$ is actually closed. The Zariski closure $\overline{S_1}$ is a closed, connected subgroup of $G$, and $S_1$ is an open subgroup of $\overline{S_1}$, which implies $S_1= \overline{S_1}$.
\end{proof}
This completes part of the proof of Theorem \ref{thm:stabilizer}. Namely, the locus of $z$ such that $L$ and $\phi_z(L)$ are stably isomorphic form a closed subgroup. Therefore, it is an extension of a subtorus of $H^1(M,\bG_m)$ by a finite set. Denote this locus by $H$ from now on. We claimed that the Lie algebra of $H$ is given by the kernel of $H^1(M,\Lambda)\to H^1(L,\Lambda)$. This kernel can obtained from the kernel of $H^1(M,\bZ)\to H^1(L,\bZ)$ by $\Lambda$-linear extension. Kernel of $H^1(M,\bZ)\to H^1(L,\bZ)$ is a primitive sublattice of $H^1(M,\bZ)$ ($H^1(M,\bZ)$ is naturally isomorphic to the cocharacter lattice of $H^1(M,\bG_m)$). Hence, it defines a subtorus of $H$ which we denote by $K$. To complete the proof of Theorem \ref{thm:stabilizer}, we need to show $K=H$. First,
\begin{lem}\label{lem:kneqh}
If $K\neq H$, then there exists $z\in H\setminus K$ such that $z\in H^1(M,U_\Lambda)$ (hence, it corresponds to a unitary local system on $M$ that is non-trivial on $L$).
\end{lem}
\begin{proof}
As $K$ is connected, it is a subgroup of $H^o\subset H$, the identity component. First assume $K\subsetneq H^o$. Then, there exists a primitive cocharacter $v\in Lie(H)$ such that $\Lambda. v\cap Lie(K)=\{0\}$. Then, $v$ defines a subgroup $\bG_m\cong K'\subset H^o$ such that $K\cap K'=\{1\}$. In particular, $K'\cap H^1(M,U_\Lambda)\cong U_\Lambda$ is not fully contained in $K$. This implies the claim. 

If $K=H^o\subsetneq H$, then there exists $z\in H\setminus H^o$, and $k\in \bN$ such that $z^k\in H^o=K$. In particular, $val_T(z^k)=k val_T(z)\in val_T(K)=ker (H^1(M,\bR)\to H^1(L,\bR))$. 
This implies $val_T(z)\in val_T(K)$ and there exists a $z_0\in K$ such that $val_T(z_0)=val_T(z)$. Then, $zz_0^{-1}$ is an element of $(H\setminus K)\cap H^1(M,U_\Lambda)$. 

Such an element corresponds to a unitary local system, and as it is not in $K$, the restriction of this local system to $L$ is non-trivial.
\end{proof}
The following lemma will let us complete the proof of Theorem \ref{thm:stabilizer}:
\begin{lem}\label{lem:localsystofukaya}
Let $L$ be a tautologically unobstructed, closed Lagrangian brane (endowed with the trivial unitary local system). Then there exists a functor
\begin{equation}
Loc_L(U_\Lambda)\to H^*(\cF(M))	
\end{equation}	
(resp. to $H^*(\cW(M))$), where $Loc_L(U_\Lambda)$ is the category of $U_\Lambda$-local systems on $L$, that sends $\xi $ to $(L,\xi)$. Moreover, this functor is fully faithful. 
\end{lem}
\begin{proof}
The functor can be defined using a version of the PSS map. Namely, at the morphism level, define it to be the obvious map from $RHom(\xi_1,\xi_2)=H^*(\xi_2\otimes \xi_1^{-1})$ to $HF((L,\xi_1),(L,\xi_2))=HF(L,(L,\xi_2\otimes \xi_1^{-1}))\cong H^*(\xi_2\otimes \xi_1^{-1})$. One needs to check compatibility of this with the product (only $\mu^2$), which holds as the PSS map is an algebra map. 
The fully faithfulness is clear.
\end{proof}
\begin{cor}\label{cor:loctofukaya}
If $(L,\xi_1)$ and $(L,\xi_2)$ are stably isomorphic, than $\xi_1\cong\xi_2$.	
\end{cor}
\begin{proof}
The functor above immediately extends to direct sums of unitary local systems. As it is fully faithful, stable isomorphism assumption on 	$(L,\xi_1)$ and $(L,\xi_2)$ imply the same for $\xi_1$ and $\xi_2$, i.e. the composition
\begin{equation}\label{eq:complocalsystems}
	RHom(\xi_2,\xi_1)\otimes RHom(\xi_1,\xi_2)\to RHom(\xi_1,\xi_1)
\end{equation}
hits the identity, and similarly with $\xi_1,\xi_2$ swapped. Regardless of Fukaya category being $\bZ$-graded or not, the two spaces in (\ref{eq:complocalsystems}) are $\bZ$-graded, they have no negative degree components, and (\ref{eq:complocalsystems}) preserves grading. Therefore, $RHom^0(\xi_1,\xi_2)=H^0(\xi_2\otimes \xi_1^{-1})\neq 0$. Since $\xi_1$ and $\xi_2$ are of rank $1$, this implies $\xi_1\cong\xi_2$ in the case $L$ is connected. Observe that \eqref{eq:complocalsystems} hitting the identity has the stronger implication that for every connected component of $L$, $H^0(\xi_2\otimes \xi_1^{-1})$ has a section that is non-zero on that component; hence, the local systems are isomorphic over any component. This finishes the proof in the case $L$ is disconnected. 
\end{proof}
\begin{proof}[Conclusion of the proof of Theorem \ref{thm:stabilizer}]
We only need to show $K=H$. If $K\neq H$, Lemma \ref{lem:kneqh} implies there exists an element $z\in H\cap H^1(M,U_\Lambda)$ such that the corresponding local system $\xi_z$ is non-trivial on $L$. On the other hand, $z\in H$ implies $L$ and $\phi_z(L)$ are stably isomorphic, and by Corollary \ref{cor:loctofukaya}, $\xi_z|_L=1$. This contradiction shows that initial assumption was wrong and $K=H$. This completes the proof of Theorem \ref{thm:stabilizer}. 
\end{proof}
Now we can show
\begin{proof}[Proof of Corollary \ref{cor:vanishingflux}]
Let $[\alpha]$ be the flux of an isotopy from $1$ to $\phi\in Symp^0(M,\omega)$. By Banyaga's theorem, $\phi_\alpha^1$ is Hamiltonian isotopic to $\phi$. Hence, if we let $z=T^{[\alpha]}$, $\phi_z(L)$ and $L$ are stably isomorphic. By Theorem \ref{thm:stabilizer}, $z$ belongs to subtorus with Lie algebra given by $ker (H^1(M,\Lambda)\to H^1(L,\Lambda))$. In particular, $[\alpha]=val_T(z)\in ker (H^1(M,\bR)\to H^1(L,\bR))$. In other words, $[\alpha|_L]=0$. 
\end{proof}
\begin{proof}[Proof of Corollary \ref{cor:lagrvanishingflux}]
If $H^1(M,\bR)\to H^1(L,\bR)$ is surjective, than any Lagrangian isotopy from $L$ can be generated by a global symplectic isotopy. Hence, applying Corollary \ref{cor:vanishingflux}, we conclude the proof. 	
\end{proof}
We have always considered embedded Lagrangians, but we did not assume the Lagrangians are connected. Before ending the section, we prove the following lemma, which shows that under this extra assumption, the notions of isomorphic and stably isomorphic coincide:
\begin{lem}
If $L$ and $L'$ are connected and stably isomorphic, then they are isomorphic.		
\end{lem}
\begin{proof}
Stably isomorphic implies that
\begin{equation}
\mu^2:HF(L',L)\otimes HF(L,L')\to HF(L,L)\simeq H^*(L,\Lambda)	
\end{equation}
hits the unit. Therefore, there exists closed morphisms $f_i, g_i$ such that $\sum \mu^2(f_i,g_i)=1$. In particular, there exists an $i$ such that $\mu^2(f_i,g_i)$ has non-vanishing $H^0(L,\Lambda)\cong \Lambda$ component. In the $\bZ/2\bZ$-graded case, there may be other terms of even degree, but non-vanishing $H^0(L,\Lambda)$-component implies $\mu^2(f_i,g_i)$ is invertible in the algebra $H^*(L,\Lambda)$. Hence, by composing $f_i$ with an invertible element of $HF(L,L)\cong H^*(L,\Lambda)$ on the left, we can assume $\mu^2(f_i,g_i)=1$. Hence, $g_i$ is a split embedding from $L$ to $L'$. Similarly, there exists a split embedding $h$ from $L'$ to $L$. Thus, $\mu^2(h,g_i)$, resp. $\mu^2(g_i,h)$ are split embeddings from $L$ to itself, resp. $L'$ to itself. Any such split embedding is invertible, as the algebras $H^*(L,\Lambda)$, resp. $H^*(L',\Lambda)$ has no non-trivial idempotents. Invertibility of $\mu^2(h,g_i)$ and $\mu^2(g_i,h)$ imply invertibility of both $h$ and $g_i$. 
\end{proof}

\section{Affine torus charts and the rationality of the mirror}\label{sec:mirror}
In this section, we prove the following statement:
\begingroup
\def\thethm{\ref*{thm:rationalmirror}}
\begin{thm}
Assume $M$ is Weinstein, $L$, as above, is a Lagrangian torus, and $H^1(M,\bR)\to H^1(L,\bR)$ is surjective. Assume $M$ is mirror dual to a projective or affine variety $X$ over $\Lambda$, in the sense that the wrapped Fukaya category is $\bZ$-graded and derived equivalent to $D^b Coh(X)$. Further assume, the equivalence maps $L$ to a sky-scraper sheaf of a smooth point $x\in X$. Then $x$ lies in a Zariski chart isomorphic to $H^1(L,\bG_m)\cong \bG_m^{b_1(L)}$. In particular, its irreducible component is rational. Under the given equivalence, other points in this chart correspond to Lagrangian tori (possibly equipped with unitary local systems) inside $M$.	
\end{thm}
\addtocounter{thm}{-1}
\endgroup
The idea is to construct ``charts'' in the moduli of objects/modules. The family $h_L\otimes_{\cW(M)} \fM^M$ is ``versal'' in the sense that infinitesimally it deforms in every possible direction. Using Theorem \ref{thm:stabilizer}, we will actually show that by restricting this family to a subvariety of the base, we obtain a family of modules that can be thought as a chart in the moduli of objects.

More precisely, let $H\subset H^1(M,\bG_m)$ be the subtorus of points $z$ such that $L$ and $\phi_z(L)$ are stably isomorphic (i.e. the subtorus with Lie algebra $ker(H^1(M,\Lambda)\to H^1(L,\Lambda))$ and cocharacter lattice $ker(H^1(M,\bZ)\to H^1(L,\bZ))$). By choosing a complement to the cocharacter lattice of $H$, we obtain a subtorus $J\subset H^1(M,\bG_m)$ such that $H^1(M,\bG_m)=H\oplus J$. Hence, the induced maps $J\to H^1(L,\bG_m)$ and $Lie(J)\to H^1(L,\Lambda)$ are isomorphisms. Consider the family 
\begin{equation}
\fN_L:=	h_L\otimes_{\cW(M)} \fM^M|_J
\end{equation}
parametrized by $J$. Heuristically, this family can be seen as an algebraic map from $J$ to the moduli of objects/modules. Also, by \Cref{thm:mainabstract}, one still has $\fN_L|_z\simeq h_{\phi_z(L)}$. We will show that $\fN_L$ is injective on $\Lambda$-points and an immersion, in an appropriate sense. Indeed, 
\begin{lem}\label{lem:injectivenl}
If $z_1,z_2\in J(\Lambda)$ are such that $\fN_L|_{z_1}\simeq \fN_L|_{z_2}$, then $z_1=z_2$.
\end{lem}
\begin{proof}
If $\fN_L|_{z_1}\simeq \fN_L|_{z_2}$, then 	$\phi_{z_1}(L)$ and $\phi_{z_2}(L)$ are stably isomorphic. Hence, $L$ is stably isomorphic to $\phi_{z_2^{-1}z_1}(L)$ and by \Cref{thm:stabilizer}, $z_2^{-1}z_1\in H$. On the other hand, $z_2^{-1}z_1\in J$ as well, and $H\cap J=\{1\}$. This implies $z_1=z_2$.
\end{proof}
Without loss of generality, we can assume $L$ is exact. Indeed, the surjectivity of $H^1(M,\bR)\to H^1(L,\bR)$ implies that there exists a closed $1$-form $\alpha$ on $M$ whose restriction to $L$ agree with the Liouville form. Hence, $\phi_\alpha^{-1}(L)=\phi_{T^{-[\alpha]}}(L)$ is exact. We can also assume $T^{[\alpha]}\in J$. Then, the families $h_{L}\otimes_{\cW(M)} \fM^M|_J$ and 	$h_{\phi_{T^{-[\alpha]}}(L)}\otimes_{\cW(M)} \fM^M|_J$ are related by the change of parameter $z\mapsto T^{[\alpha]}z$, thanks to group-like property. From now on assume $L$ is exact. 

Let $z\in J(\Lambda)$. Recall that a tangent vector at $z$ can be identified with an element $z_\nu\in J(\Lambda[\epsilon]/(\epsilon^2))$ that specialize to $z$ at $\epsilon=0$. A map is an immersion if its differential is injective on tangent vectors. Our immersion statement is 
\begin{lem}\label{lem:immersivenl}
If $z_{\nu_1}$ and $z_{\nu_2}$ are two tangent vectors at $z$, and $\fN_L|_{z_{\nu_1}}$ and $\fN_L|_{z_{\nu_2}}$ are quasi-isomorphic, then $z_{\nu_1}=z_{\nu_2}$.
\end{lem}
Here, $\fN_L|_{z_{\nu_1}}$ and $\fN_L|_{z_{\nu_2}}$ are families of modules over $\cW(M)$ parametrized by $Spec(\Lambda[\epsilon]/(\epsilon^2))$. Denote $\Lambda[\epsilon]/(\epsilon^2)$ by $R$.
\begin{proof}
As $L$ is exact, it can be considered as an object of $\cW(M)$ and $\fN_L$ is quasi-isomorphic to $\fM^M(\cdot,L)$. In the language of \Cref{sec:algfamilydefclass}, this is the same as $h_L^{alg}|_J$ (see \Cref{defn:hlalg}). In particular, $\fN_L|_{z_{\nu_i}}\simeq h^{alg}_L|_{z_{\nu_i}}$. Let $\xi_{z_{\nu_i}}$ denote the rank $1$, $R=\Lambda[\epsilon]/(\epsilon^2)$ local system on $M$ corresponding to $z_{\nu_i}$. Then $h^{alg}_L|_{z_{\nu_i}}$ is represented as ``the Yoneda module of $(L,\xi_{z_{\nu_i}}|_L)$''. We will use the Yoneda lemma to identify $hom(h^{alg}_L|_{z_{\nu_i}},h^{alg}_L|_{z_{\nu_j}})$ with $CF((L,\xi_{z_{\nu_i}}|_L),(L,\xi_{z_{\nu_j}}|_L))$ (and hence with $RHom(\xi_{z_{\nu_i}}|_L,\xi_{z_{\nu_j}}|_L)$ in the category of $R$-linear local systems over $L$). 

To make this precise, one can construct the $R$-linear wrapped Fukaya category of Lagrangians equipped with rank $1$ $R$-linear local systems. Consider this category spanned by the objects $ob(\cW(M))\sqcup \{(L,\xi_{z_{\nu_i}}|_L):i=1,2\}$, and denote it by $\widetilde{\cW}(M)$. Explicitly, the $A_\infty$-subcategory of $\widetilde{\cW}(M)$ spanned of the objects of $\cW(M)$ is obtained by base change $\cW(M)\otimes_\Lambda R$, but once the other two objects are added, the $A_\infty$-structure maps are obtained by counting discs whose weights are twisted by the local systems. Consider the Yoneda modules $h_{(L,\xi_{z_{\nu_i}}|_L)}$ and restrict them to the subcategory $\cW(M)\otimes_\Lambda R$. The category of modules over this subcategory can be identified with families of modules over $\cW(M)$ parametrized by $Spec(R)$. In particular, $h_{(L,\xi_{z_{\nu_i}}|_L)}$ corresponds to $h^{alg}_L|_{z_{\nu_i}}$, and the homomorphisms $hom(h^{alg}_L|_{z_{\nu_i}},h^{alg}_L|_{z_{\nu_j}})$ of families is quasi-isomorphic to  $hom(h_{(L,\xi_{z_{\nu_i}}|_L)},h_{(L,\xi_{z_{\nu_j}}|_L)})$. By Yoneda lemma, the latter is quasi-isomorphic to 
\begin{equation}\label{eq:tildewcflocal}
\widetilde{\cW}(M)((L,\xi_{z_{\nu_i}}|_L),(L,\xi_{z_{\nu_j}}|_L))\simeq CF((L,\xi_{z_{\nu_i}}|_L),(L,\xi_{z_{\nu_j}}|_L))\simeq R\Gamma_L((\xi_{z_{\nu_j}}\otimes \xi_{z_{\nu_i}}^{-1})|_L)
\end{equation}
If $z_{\nu_1}\neq z_{\nu_2}$, then $z_{\nu_2}z_{\nu_1}^{-1}\in J(R)\setminus\{1\}$, which does not intersect $H$. Therefore, $z_{\nu_2}z_{\nu_1}^{-1}\not\in H$, and the restriction $\xi_d:=(\xi_{z_{\nu_2}}\otimes \xi_{z_{\nu_1}}^{-1})|_L$ is a non trivial rank $1$ $R$-local system on $L$. Since, $z=z_{\nu_1}|_{\epsilon=0}= z_{\nu_2}|_{\epsilon=0}$, on the other hand, the $\Lambda=R/(\epsilon)$-local system obtained from $\xi_d$ by setting $\epsilon=0$ is trivial. In other words, $\xi_d$ has monodromy in $1+\epsilon R\subset R^*$, and this monodromy is non-trivial. In particular, $R^0\Gamma_L(\xi_d)=\Gamma_L(\xi_d)=\epsilon R$. 
On the other hand, if one puts $i=j$ in \eqref{eq:tildewcflocal}, one obtains $R$ in degree $0$. This suffices to show that $h^{alg}_L|_{z_{\nu_1}}$ and $h^{alg}_L|_{z_{\nu_2}}$ are different, which finishes the proof.
\end{proof}
\begin{note}
Observe that the only property of the ring $R$ that we need here is that the fixed set of any $R$-linear non-trivial action of $\pi_1(L)$ on $R$ is different from $R$ (as an $R$-module). More precisely, we distinguished a non-trivial $R$-local system on $L$ of rank $1$ from the trivial one by using the fact that the global sections are different. This is equivalent to the given condition. This property holds for any field extension of $\Lambda$, any commutative Artinian local ring, as well as for any integral domain. The conclusion is, for any two different elements $z_1,z_2\in J(R)$, $\fN_L|_{z_1}$ and $\fN_L|_{z_2}$ are different (as families of modules parametrized by $Spec(R)$). 
\end{note}
Next, we discuss technical properties of this family. 
\begin{lem}\label{lem:propernl}
The family $\fN_L$ is proper, i.e. $\fN_L(L')$ is cohomologically finitely generated over the base $J\subset H^1(M,\Lambda)$, for any $L'\in \cW(M)$.
\end{lem}
\begin{proof}
As $L$ is exact, we can assume $L$ is an object of $\cW(M)$. By \Cref{lem:hlconvhlequalshom}
\begin{equation}
	\fN_L(L')\simeq \fN_L \otimes_{\cW(M)} h^{L'} =h_L\otimes_{\cW(M)} \fM^M|_J \otimes_{\cW(M)} h^{L'} =\fM^M(L',L)|_J
\end{equation}
$\fM^M(L',L)$ is by definition a trivial vector bundle over $H^1(M,\bG_m)$ of finite rank; hence, so is its restriction to $J$.
\end{proof}
We also have:
\begin{lem}\label{lem:perfectnl}
$\fN_L$	is a perfect family, i.e. it can be represented as a twisted complex of constant families. Also, for any $z$, $Hom(\fN_L|_z,\fN_L|_z):=H^*(hom(\fN_L|_z,\fN_L|_z))$ vanishes in negative degree, and it is one dimensional at degree $0$.
\end{lem}
\begin{proof}
Perfectness follow from properness, see \cite[Lemma 3.15, Lemma 3.16]{owniteratespadic} for some versions (these statements does not use properness of the category). Alternatively, consider the functor 
\begin{equation}
	\fN\mapsto \fN\otimes_{\cW(M)}^{rel}\fM^M|_J
\end{equation}
on the families of right modules over $J$. Group-like property implies that this functor is an equivalence with inverse given by 
\begin{equation}
	\fN\mapsto \fN\otimes_{\cW(M)}^{rel}\fM^{M,-}|_J
\end{equation}
Hence, it maps compact objects of the category of families to compact objects. It is standard that the perfect modules coincide with the compact objects (see \cite{kellerdg}), and this still holds for families of modules. Therefore, $\fN_L$, the image of the constant family $\underline{h_L}$ is perfect.

Properties of hom-sets also follow from this auto-equivalence of families, namely $hom(\underline{h_L},\underline{h_L})$ is quasi-isomorphic to $hom(h_L,h_L)\otimes \cO(J)\simeq CF(L,L)\otimes \cO(J)$. This property is preserved under the equivalences of families.
\end{proof}
\begin{note}
This claim holds for points $z$ over other rings as well, not only for $\Lambda$-points.	
\end{note}
The properties of $\fN_L$ shown in \Cref{lem:injectivenl}, \Cref{lem:immersivenl}, \Cref{lem:propernl}, and \Cref{lem:perfectnl} are intrinsic to this family, and are preserved under the mirror equivalence.  

We use $Coh(X)$ to denote an $A_\infty$-enhancement of $D^b Coh(X)$, which is derived equivalent (more precisely Morita equivalent) to $\cW(M)$ by assumption (and by the uniqueness of dg enhancements, see \cite{orlovlunts}). Using the Morita equivalence, $\fN_L$ can be translated to a family over $Coh(X)$, i.e. an $A_\infty$-functor 
\begin{equation}
	Coh(X)^{op}\to \cO(J)^{mod}\simeq Coh(J)
\end{equation}
By \cite{indcoh}, one can see this functor (or rather its extension to $IndCoh(X)$) as a Fourier-Mukai transform
\begin{equation}
\scrF\mapsto	Rq_*\derivedsheafhom (p^*\scrF, \cK)	
\end{equation}
where $\cK\in IndCoh(X\times J)$, $\derivedsheafhom$ denote the local-hom and $p:X\times J\to X$, $q:X\times J\to J$ are the first and second projections. 
\begin{rk}
To obtain a functor $Coh(X)\to \cO(J)^{mod}\simeq Coh(J)$ and a Fourier-Mukai transform of the form $Rq_*(p^*\scrF\otimes^L  \cK)$ in its more typical form, one has to consider the family $\fM^M|_J\otimes_{\cW(M)} h^L$ instead. Everything else works in the same way; hence, we stick to the family $\fN_L$.
\end{rk}
One can construct the Fourier-Mukai kernel without reference to \cite{indcoh}, as the proof of the following lemma implies:
\begin{lem}\label{lem:coherenceK}
$\cK$ is quasi-isomorphic to a bounded complex of coherent sheaves over $X\times J$.
\end{lem}
\begin{proof}
By Lemma \ref{lem:perfectnl}, the corresponding family over $Coh(X)$ can be represented by a twisted complex of constant families. Any constant family $\underline{\scrF}$, where $\scrF\in Coh(X)$, is represented by the kernel $p^*\scrF$. Hence, the lemma follows.	
\end{proof}
The functor category is quasi-isomorphic to $IndCoh(X\times J)(\simeq IndCoh(X)\otimes IndCoh(J))$. Therefore, one can compute $hom(\fN_L,\fN_L)$ as $RHom_{X\times J}(\cK,\cK)$, which identifies with $Rq_*\derivedsheafhom(\cK,\cK)$, the $J$-relative homomorphisms of $\cK$, as $J$ is affine.

The second part of Lemma \ref{lem:perfectnl} translates as:
\begin{lem}
For any point $z$ of $J$, $RHom_X(\cK|_z,\cK|_z)\simeq H^*(L)$. In particular, it vanishes in negative degree, and it is one dimensional at degree $0$.	
\end{lem}
In other words, $\cK$ can be seen as morphism from the scheme $J$ to the moduli of complexes of sheaves satisfying this property (more precisely, it is a natural transformation from the functor of points of $J$ to the moduli functor of complexes with vanishing negative degree $(-1)$ self-Ext, and one dimensional degree $0$ self-Ext).

First assume $X$ is projective. Then, by \cite{inabamoduli}, this moduli space is represented by an algebraic space $Splcpx_X$. Therefore, the natural transformation above can be seen as a map $J\to Splcpx_X$. Moreover, $\fN|_{z=1}\simeq h_L$ corresponds to $\cO_x$, the sky-scraper sheaf of $x\in X$, by the assumption on mirror equivalence. Hence, $J$ maps to the component of $Splcpx_X$ consisting of sky-scraper sheaves. This component is clearly isomorphic to $X$, in other words we obtain a map $J\to X$.

Now assume $X$ is affine, and let $X\subset \bar X$ be a compactification to a projective scheme. By \Cref{lem:coherenceK}, $\cK\in D^bCoh(X\times J)$. We have 
\begin{lem}
The push-forward $\bar\cK$ of $\cK$ to $\bar X\times J$ is still in $D^b Coh(\bar X\times J)$, and for any $z\in J$, $Ext_{\bar X}(\bar\cK|_z,\bar\cK|_z)$ is the same as $Ext_X(\cK|_z,\cK|_z)$.
\end{lem}
In particular, $\bar\cK$ satisfies the same self-Ext conditions and defines a map $J\to Splcpx_{\bar{X}}$.
\begin{proof}
It suffices to show that the support of $\cK$ is proper over $J$. By \Cref{lem:propernl}, for any $\scrF\in Coh(X)$, $Rq_*(\derivedsheafhom_X(p^*\scrF,\cK) )$ is a bounded complex of coherent sheaves over $J$. This in particular, holds for $\scrF=\cO_X$. Therefore, $Rq_*\cK$ in in $D^bCoh(J)$. In particular, for any point $z$ of $J$, $R\Gamma(X,\cK|_z)$ is finite dimensional. As $X$ is affine, this implies $\cK|_z$ has proper support in $X$, and the $J$-relative version also holds, i.e. $\cK$ has proper support over $J$. 
\end{proof}
Hence, we obtain a map $J\to Splcpx_{\bar{X}}$, and by the same reasoning, this actually comes from a map $J\to \bar X$. On the other hand, $\cK$ is pushed forward from $X\times J$, its support is also there. Therefore, the map $J\to \bar X$ has image in $X$.

In summary, we obtain a map from the affine torus $J$ to $X$. \Cref{lem:injectivenl} implies this map is injective on $\Lambda$-points, and \Cref{lem:immersivenl} implies it is an immersion. Note that $J$ and $X$ have the same dimension ($b_1(L)$, for $X$ this can be seen using $Ext^1(\cO_x,\cO_x)\cong H^1(L,\Lambda)$); thus, the map is an isomorphism on tangent spaces. From this one can conclude that $J$ actually lands in the smooth locus of $X$. Therefore, we have an open immersion $J\to X$. 


This proves the first part of Theorem \ref{thm:rationalmirror}. On the other hand, the module $\fN_L|_z\simeq h_{\phi_z(L)}$ is shown to be mirror to a point in this chart. Therefore, other points in this chart are mirror to Lagrangians isotopic to $L$.

\appendix
\section{Affinoid domains and semi-continuity}\label{sec:appendixaffinoidsemicont}
In this \namecref{sec:appendixaffinoidsemicont}, we review the basics of affinoid domains and prove a necessary semi-continuity result. In practice, we only need examples $S_P$ defined in Section \ref{subsec:familiesdeformations}; however, we start in higher generality here. Here too, we mostly take the dual point of view and work with rings of analytic functions. As before, let $\Lambda=\bC((T^\bR))$ and $val_T$ denote the $T$-adic valuation. Let $|x|_T:=e^{-val_T(x)}$ denote the $T$-adic norm. 
\begin{defn}\label{defn:tatealgebra}\cite{kedlayaoverconvergent}
Let $\rho:=(\rho_1,\dots, \rho_m)\in \bR^m_+$. Define \textbf{the Tate algebra $\Lambda\langle x_1,\dots, x_m\rangle_\rho$} to be the ring of series that converge over the polydisc of radius $\rho$, i.e. the set of series $\sum a_Ix^I$, $I\in \bN^m, a_I\in \Lambda$, such that \begin{equation}
	val_T(a_I)+\langle I,-\log \rho\rangle=val_T(a_I)-\sum_{i=1}^m I_i\log\rho_i\to \infty
\end{equation} 
as $I\to \infty$. This condition is equivalent to $|a_I|_T\rho^I\to 0$, there $\rho^I=\rho_1^{I_1}\dots\rho_m^{I_m}$. Notice that $\Lambda\langle x_1,\dots, x_m\rangle_\rho$ is endowed with the norm 
\begin{equation}
	\bigg| \sum a_Ix^I \bigg|:=\max\{|a_I|_T\rho^I:I\in\bN^m\}
\end{equation}
\end{defn}
\begin{rk}
The term Tate algebra is mostly used when $\rho=(1,\dots,1)$, i.e. for the set of series that converge over the closed polydisc of radius $1$. Observe that changing $\rho$ defines isomorphic rings. Namely, the map $x_i\to T^{\log \rho_i}x_i$ defines an isomorphism from $\Lambda\langle x_1,\dots, x_m\rangle_{(1,\dots, 1)}$ to $\Lambda\langle x_1,\dots, x_m\rangle_\rho$. We drop the subscript when $\rho=(1,\dots, 1)$.
\end{rk}
\begin{rk}
In Definition \ref{defn:tatealgebra}, one can replace the ground field by another non-Archimedean normed algebra. 
\end{rk}
\begin{defn}
	An \textbf{affinoid algebra} is a quotient of the Tate algebra by a closed ideal. 	
\end{defn}
One can show that Tate algebras and more general affinoid algebras are Noetherian, see \cite[p.222]{boschguntherremmert}. One can define the dual notion ``affinoid domain $Sp(A)$'', see \cite{bosch}, \cite{boschguntherremmert} for more details. The set of $\Lambda$-points of the spectrum of Tate algebra $\Lambda\langle x_1,\dots,x_m\rangle_\rho$ is the set of elements $(a_1,\dots, a_m)\in \Lambda^m$ such that $|a_i|_T\leq \rho_i$. We will formally work with affinoid algebras, but urge the reader to keep the geometric picture in mind. Throughout the appendix, we will emphasize if a point $x\in Sp(A)$ is a $\Lambda$-point, but for the rest of the paper we drop the adjective $\Lambda$, assuming this implicitly.

Given affinoid algebra $A$, and $f,g\in A$, without common zeroes in $Sp(A)$, one can define subdomains called \textbf{Weierstrass domains}, \textbf{Laurent domains} and \textbf{rational domains}. The $\Lambda$-points of the Weierstrass domain is given by those points of $Sp(A)$ such that $|f(x)|_T\leq 1$. Similarly, Laurent domains are given by those satisfying $|f(x)|_T\leq 1$ and $|g(x)|_T\geq 1$, and rational domains are given by those satisfying $|f(x)|_T\leq |g(x)|_T$ (only the last one requires no common zero assumption). These correspond to affinoid algebras $A\langle w\rangle/(w-f)$, $A\langle w_1,w_2 \rangle/(w_1-f,1-w_2g)$, and $A\langle w\rangle/(f-wg)$, respectively. Morally, these algebras can be thought as ``$A\langle f\rangle$'', ``$A\langle f, g^{-1}\rangle$'', and ``$A\langle \frac{f}{g}\rangle$'', respectively. For more details see \cite{bosch}, \cite{boschguntherremmert}. Note that these notions are slightly more general. For instance, for a finite number of functions $f_i,g_j$, the subdomain $\{|f_i(x)|_T\leq 1, |g_j(x)|_T\geq 1 \}$ is still called a Laurent domain. One can also define subdomains by iterating these constructions. These subdomains are called \textbf{subdomains of special type}.

Note that the affinoid subdomains are used to define the Grothendieck topology on affinoid domains, so they are more comparable to open subsets than the closed immersions. We will only use subdomains of special type or subdomains obtained by iterating these constructions; therefore, a subdomain will mean one of these. We refer the reader to \cite{bosch}, \cite{boschguntherremmert} for throughout discussions of affinoid subdomains. 

Next, we expand \Cref{exmp:adicannulus} and \Cref{exmp:adicpolytope} given in Section \ref{subsec:familiesdeformations} and show how they can be realized as affinoid spaces:
\begin{exmp}
Recall that for $a<b\in\bR$, $\Lambda\{z^\bZ \}_{[a,b]}$ denotes the ring of series $\sum_{n\in\bZ} a_nz^n$, $a_n\in \Lambda$ such that $val_T(a_n)+n\nu\to\infty $ as $n\to\pm\infty$, for every $a\leq \nu\leq b$. In other words, this is the ring of functions that converge at $z\in\Lambda$ with valuation between $a$ and $b$. The ring $\Lambda\{z^\bZ \}_{[a,b]}$ can be identified with the ring of functions on a Laurent subdomain of $Sp(\Lambda\langle z\rangle_{e^{-a}} )$, namely for $f(z)=T^{-a}z$, $g(z)=T^{-b}z$. The condition $|f(z)|_T\leq 1$ is equivalent to $val_T(z)\geq a$ and the condition $|g(z)|_T\geq 1$ is equivalent to $val_T(z)\leq b$. In other words, 
\begin{equation}
	\Lambda\{z^\bZ \}_{[a,b]}=\Lambda\langle z\rangle_{e^{-a}}\langle w_1,w_2 \rangle/(w_1-T^{-a}z,1-w_2T^{-b}z)
\end{equation}
We denoted the spectrum of $\Lambda\{z^\bZ \}_{[a,b]}$ by $S_{[a,b]}$, and its $\Lambda$-points are in correspondence with the elements of $\Lambda$ with $T$-adic valuation between $a$ and $b$. 
\end{exmp}
\begin{exmp}[\cite{abouzaidicm}]
More generally, let $V$ be a finite rank lattice, and let $P\subset Hom(V,\bR)$ be a convex, compact polytope defined by integral affine equations. In other words, there exists $v_1,\dots ,v_k\in V$ and $a_1,\dots ,a_k\in \bR$ such that $P$ is the set of points $\boldsymbol{\nu}$ satisfying $\langle\boldsymbol{\nu},v_i \rangle\leq a_i$. Let $\Lambda\{z^V \}_P$ denote the ring of series $\sum_{v\in V} a_vz^v$, $a_v\in \Lambda$ such that $val_T(a_v)+\langle \boldsymbol{\nu},v \rangle\to \infty$ for all $\boldsymbol{\nu}\in P$. As mentioned, if $S_P$ denote the spectrum of $\Lambda\{z^V \}_P$, then the set of $\Lambda$-points of $S_P$ is given by $val_T^{-1}(P)$, where $val_T:Hom(V,\Lambda^*)\to Hom(V,\bR)$ is extended from the valuation. 

By choosing coordinates on $V$, we obtain a positive cone $V^+$ on $V$ (the set of vectors with non-negative coordinates), and for each $v\in V$ a decomposition $v=v^+-v^-$, where $v^+ ,v^-\in V^+$ (concretely we split $v$ into its negative and positive parts). Consider the affinoid domain $\Lambda\langle z^{V^+}\rangle_\rho$ (where components $\rho_j\gg 0$). Let $v_c\in V^+$ denote an element with all positive components ($v_c$ can be taken to be $(1,1,\dots,1)$ in the chosen basis). Also choose a small $\epsilon>0$. To obtain $S_P$ from the spectrum of this domain, one puts the conditions
\begin{equation}\label{eq:geometriccuts}
\langle val_T(z),v_c \rangle\geq \epsilon \text{ and }
\langle val_T(z),v_i \rangle\leq  a_i \text{ for }i=1,\dots , k	
\end{equation} 
To explain these geometrically, recall that $S_P\subset Spec(\Lambda [z^V])\subset Spec(\Lambda [z^{V^+}])\cong \bA^{rk(V)}_\Lambda$. As $P$ is compact, $S_P$ is also in $Sp(\Lambda\langle z^{V^+}\rangle_\rho)$ for large $\rho$. Observe that the first condition in \eqref{eq:geometriccuts} cuts out the part of $Sp(\Lambda\langle z^{V^+}\rangle_\rho)\subset \bA^{rk(V)}_\Lambda$ that is close to the zero. Other conditions cut the polytope $P$. Note that, alternatively one could consider the variables (corresponding to basis elements of $V$) and add a condition like the first one in \eqref{eq:geometriccuts}. As we choose $v_c$ with all positive components and cut out further, this does not make a difference. 

Algebraically, one first takes the Weierstrass subdomain with algebra $\Lambda\langle z^{V^+}\rangle_\rho\langle w_0\rangle/ (w_0-T^{-\epsilon} z^{v_c})$ of $\Lambda\langle z^{V^+}\rangle_\rho$. This inverts all $z^v, v\in V^+$ in the process, so this domain is contained in $Spec(\Lambda [z^V])\cong \bG_{m,\Lambda}^{rk(V)}$. Then one takes the (iterated) Laurent subdomain with algebra
\begin{equation}
	\Lambda\langle z^{V^+}\rangle_\rho\langle w_0,w_1\dots w_k\rangle/ (w_0-T^{-\epsilon} z^{v_c}, 1-w_iT^{-a_i}z^{v_i}|i=1,\dots ,k)
\end{equation}
of the latter ($z^{v_i}$ may not belong to the former algebra). Alternative is to take an iterated rational subdomain 
\begin{equation}
		\Lambda\langle z^{V^+}\rangle_\rho\langle w_0,w_1'\dots w_k'\rangle/ (w_0-T^{-\epsilon} z^{v_c}, T^{a_i}z^{v_i^-}-w_i'z^{v_i^+}|i=1,\dots ,k)
\end{equation}
(recall $z^{v_i^\pm}\in \Lambda\langle z^{V^+}\rangle_\rho$). The equivalence of these is due to $val_T(z^{v_i^+})-val_T(z^{v_i^-})=\langle val_T(z),v_i \rangle$; hence, $val_T(z^{v_i^+})\leq val_T(T^{a_i}z^{v_i^-})$ iff $val_T(T^{-a_i}z^{v_i})\leq 0$. In short, the outcome does not change, in both cases one obtains $\Lambda\{z^V\}_P$.
\end{exmp}
Our next goal is to prove a semi-continuity statement for the rank of chain complexes over affinoid domains. More precisely:
\begin{prop}\label{prop:semicontinuityoverbase}
Let $(C,d)$ be a chain complex over $\Lambda\{z^V\}_P=\cO^{an}(S_P)$ with finitely generated cohomology, where $P$ is a polytope as above which also contains $0$ in its interior. Assume the derived restriction of $(C,d)$ to any $\Lambda$-point of $S_0$ is acyclic. Then, there exists a smaller polytope $P'\subset P$ such that $0\in int (P')$ and the derived restriction of $(C,d)$ to $\Lambda\{z^V\}_{P'}=\cO^{an}(S_{P'})$ is acyclic.
\end{prop}
First recall:
\begin{fact}[Weak Nullstellensatz for affinoid algebras, {\cite[p. 263, Corollary 5]{boschguntherremmert}}, {\cite[Fact 3.1.2.3]{Temkin2015}}]\label{fact:nullstellensatzforaffinoids}
If $A$ is an affinoid algebra over $\Lambda$, i.e. a quotient of a Tate algebra, and $f_1,\dots, f_k\in A$ have no common zeroes among $\Lambda$-points of $Sp(A)$, then $f_1,\dots,f_k$ generate the unit ideal of $A$.
\end{fact}
\cite{boschguntherremmert} state this for Tate algebras only, but the statement we want follows immediately from this. \cite{Temkin2015} states the existence of a point over a finite extension of the ground field; however, as $\Lambda$ is algebraically closed, this gives us what we want.

We will also need:
\begin{lem}\label{lem:projectiveresolution}
If $S$ is a smooth affinoid domain, then every chain complex of $\cO^{an}(S)$-modules with finitely generated cohomology is quasi-isomorphic to a finite complex of finite rank projective $\cO^{an}(S)$-modules. 
\end{lem}
\begin{proof}
The regular Noetherian ring $\cO^{an}(S)$ has finite Krull dimension; hence, also finite projective dimension. More precisely, since it is regular the Krull dimension agrees with the projective dimension at every maximal ideal, and therefore by \cite[(5.92) Theorem]{lamlecturesmodulesrings}, the projective dimension of $\cO^{an}(S)$ is also bounded. In particular, every finitely generated module over $\cO^{an}(S)$ has a finite projective resolution, and this implies the statement of \Cref{lem:projectiveresolution} immediately.
\end{proof}
\begin{note}
By \cite[Prop 6.5]{kedlayaoverconvergent}, finitely generated modules over $\Lambda\langle x_1,\dots, x_m\rangle_\rho$ has finite free resolutions, which implies the same for the complexes as in \Cref{lem:projectiveresolution}. We believe, the same holds for its affinoid subdomains such as $S_P$, but we do not need this. Indeed, to show this one only needs to be able extend coherent sheaves over $S_P$ to $Sp(\Lambda\langle x_1,\dots, x_m\rangle_\rho)$, as then one can take a free resolution of the extension and then restrict back. 
\end{note}
\begin{note}\label{note:affinoidvectorbundle}
Given affinoid domain $A$, a finite rank projective $A$-module $E$ is locally free in the affinoid topology too. Indeed, if $E$ is projective, there exists $u_i,g_i\in A, i=1,\dots, m$ such that $\sum u_ig_i=1$ and $E_{g_i}$ is free over $A_{g_i}$ (hence, $E_{u_ig_i}$ is free over $A_{u_ig_i}$). By replacing $g_i$ by $u_ig_i$, we can assume $\sum g_i=1$. This implies the Laurent domains $\{|g_i|\geq 1\}$ cover $Sp(A)$. Moreover, $A\langle w\rangle/(1-wg_i)$ is an extension of $A_{g_i}$; thus, the base change of $E$ to $A\langle w\rangle/(1-wg_i)$ is also free. Also observe that one can choose the (co)framing of $E_{g_i}$ to be defined over $A$. In other words, one can choose a map $E\to A^{rk(E)}$ that become an isomorphism over the localization $A_{g_i}$. Therefore, one can also choose the (co)framing of the base change of $E$ to $A\langle w\rangle/(1-wg_i)$ well-defined over $A$. 
\end{note}

\begin{lem}\label{lem:functiondeformsboundedaway}
Let $Sp(A)$	be an affinoid subdomain of $S_P$ (hence, $Sp(A)\cap S_{P'}$ is a subdomain of both $Sp(A)$ and $S_{P'}$). Let $f\in A$ be such that $f$ is non-zero on the $\Lambda$-points of $Sp(A)\cap S_0$. Then, there exists a small $P'\subset P$ that contain $0$ in its interior such that restriction of $f$ to $Sp(A)\cap S_{P'}$ is invertible and bounded away from $0$. More precisely, there exists a small $\epsilon>0$ such that $|f(x)|_T\geq \epsilon$ for all $x\in Sp(A)\cap S_{P'}$.
\end{lem}
\begin{proof}
Let $A_P$ denote the ring of functions on $Sp(A)\cap S_P$. By \Cref{fact:nullstellensatzforaffinoids}, the restriction of $f$ to the affinoid domain $Sp(A)\cap S_0$ is invertible, i.e. there is an analytic function $u\in A_0=\cO^{an}(Sp(A)\cap S_0)$ such that $fu=1$. As $|u(x)|_T$ is bounded above, this implies $|f(x)|_T$ bounded away from $0$ on $Sp(A)\cap S_0$.
	
The rest of the argument is topological, and for this, it is convenient to use the language of Berkovich spaces. See \cite{Temkin2015} or \cite{bakerintroberkovich} for an introduction. In summary, the Berkovich spectrum $Sp_b(A)$ of the normed algebra $A$ is the set of multiplicative semi-norms $|\cdot|$ on $A$ that is bounded above by the norm of $A$ and that extend the norm $|\cdot|_T$ on $\Lambda$. It refines the spectrum of maximal ideals $Sp(A)$ and the semi-norm corresponding to a $\Lambda$-point $x$ is given by $g\mapsto |g(x)|_T$. It also has a natural topology, which is the coarsest one making $|\cdot|\mapsto |g|$ continuous for all $g\in A$. This space is quasi-compact and Hausdorff. In particular, for $u\in A_0$ above, $|\cdot|\mapsto |u|$ has bounded image, and there exists an $a>0$ such that $|f|\geq a$ for all $|\cdot|\in Sp_b(A_0)$. The set of $|\cdot|\in Sp_b(A_P)$ such that $|f|>a/2$ is an open neighborhood of $Sp_b(A_0)$. 

Let $S_P^{berk}:=Sp_b(\Lambda\{z^V\}_P)$ and observe that $val_T$ extends to a continuous map $trop:S_P^{berk}\to P\subset Hom(V,\bR)$. Explicitly, one sends a semi-norm $|\cdot|$ to $(v\mapsto -\log|z^v|)\in Hom(V,\bR)$. Also note that $S_P^{berk}$ is homeomorphic to the product of $S_0^{berk}$ and $P$. Indeed, given $\nu\in P$, and $g(z)$, let $g(T^\nu z)$ denote the function obtained by replacing $z^v$ by $T^{\langle \nu ,v\rangle}z^v$. Given semi-norm $|\cdot|$ define $tr_\nu |\cdot|$ by $g\mapsto |g(T^\nu z)|$. It is easy to see that $trop(tr_\nu|\cdot|)=\langle \nu ,v\rangle+trop(|\cdot|)$. Hence, $P\times S_0^{berk}\to S_P^{berk}, (\nu,|\cdot|)\mapsto tr_\nu|\cdot|$ is an homoemorphism with the obvious inverse.
%
%

Therefore, $S_P^{berk}\to P$ is a fibration, and there exists a small neighborhood $0\in P'\subset P$ such that the neighborhood of $Sp_b(A_0)\subset Sp_b(A_P)\subset S_P^{berk}$ that consists of $|\cdot|\in Sp_b(A_P)$ satisfying $|f|>a/2$ contains $trop^{-1}(P')\cap Sp_b(A_P)$. In other words, $|f|>a/2$ holds on $trop^{-1}(P')\cap Sp_b(A_P)$, and this implies the desired result. 
%
%
\end{proof}
Another affinoid version of a classical statement is the following:
\begin{lem}\label{lem:fingenvanisheverywhere}
If a finitely generated module $M$ over an affinoid domain $A$ restricts to $0$ at every $\Lambda$ point of $Sp(A)$, then $M=0$.
\end{lem}
\begin{proof}
Assume $M\neq 0$, and consider the proper ideal $ann_A(M)$. This ideal is finitely generated, and \Cref{fact:nullstellensatzforaffinoids} implies it is contained in the maximal ideal of a $\Lambda$-point $x$. Combining this with Nakayama lemma implies the restriction of $M$ to $x$ is also non-zero. 
\end{proof}
\begin{proof}[Proof of \Cref{prop:semicontinuityoverbase}]
By \Cref{lem:projectiveresolution}, $(C,d)$ is quasi-isomorphic to a complex of finite rank projective $\Lambda\{z^V\}_P$-modules, and without loss of generality, assume $(C,d)$ itself is such. Choose $g_i$ as in \Cref{note:affinoidvectorbundle}, i.e. satisfying $\sum g_i=1$ and such that $C$ is (degreewise) free on the Laurent domains $\{|g_i(x)|_T\geq 1\}$ (and indeed even over the localization at $g_i$). 

One can show that the dimension of $H(C|_x,d|_x)$ is equal to $rank (C)-2rank(d|_x)$. By assumption, the complex is acyclic at every $\Lambda$-point $x\in Hom(V,U_\Lambda)=S_0(\Lambda)$; thus, $rank(d|_x)$ is maximal and equal to $r:=rank(C)/2$. Choose a trivialization for $C$ over each $\{|g_i(x)|_T\geq 1\}$ that extends to a map $\Lambda\{z^V\}_P^{rk(C)}\to C$ on $S_P$ (possibly as a non-trivialization, see \Cref{note:affinoidvectorbundle}). Hence, one can see $d$ as an upper triangular matrix and consider its set of minors of size $r$. The minors may not be defined over $S_P$; however, by multiplying them with $g_i^k,k\gg 0$, we can ensure this holds (here it is important that the trivialization is defined over the localization at $g_i$ rather than just the Laurent domain). Let $f_{ij}$ denote this set of (modified) minors. 

For any $\Lambda$ point $x\in S_0(\Lambda)$, some $g_i(x)\neq 0$, and as $rank(d|_x)=r$, some $f_{ij}\neq 0$; hence, $g_i(x)f_{ij}(x)\neq 0$. Thus, $g_if_{ij}\in \Lambda\{z^V\}_P=\cO^{an}(S_P)$ has no common zeroes on $S_0$. As a result, by \Cref{fact:nullstellensatzforaffinoids} again, there exists $u_{ij}\in \Lambda\{z^V \}_0=\cO^{an}(S_0)$ such that $\sum u_{ij}g_if_{ij}=1$. Each $u_{ij}$ is a series in $z^v$ that does not necessarily converge outside $S_0$. However, we can replace each $u_{ij}$ by a sufficiently large finite partial sum $u_{ij}'$, so that $u:=\sum u_{ij}'g_if_{ij}$ is close to $1$ over $S_0$. More precisely, the sup-norm of $1-u$ is small and thus, $|u(x)|_T$ is bounded away from $0$ for $x\in S_0$. Being finite sums $u_{ij}'$ are defined over $S_P$; thus, the same holds for $u$. On the other hand, by \Cref{lem:functiondeformsboundedaway}, for a small neighborhood $P'\subset P$ of $0$, $|u(x)|_T$ is bounded away from $0$ for $x\in S_{P'}(\Lambda)$. Say $|u(x)|_T\geq a>0$.

Therefore, for every $x\in S_{P'}(\Lambda)$, there exists $i,j$ such that $|u_{ij}'(x)g_i(x)f_{ij}(x)|_T\geq a$. In particular, $g_i(x)\neq 0$ and the cofactor corresponding to $f_{ij}$ is also non-vanishing at $x$. In other words, $d|_x$ has rank $r$, and $(C|_x,d|_x)$ is acyclic, for every $\Lambda$-point $x\in S_{P'}(\Lambda)$. The spectral sequence argument in \Cref{lem:pointwisevanishing} shows that the restriction of $H(C,d)$ to $x$ vanishes. This, combined with \Cref{lem:fingenvanisheverywhere}, implies \Cref{prop:semicontinuityoverbase}.
\end{proof}	

\bibliographystyle{alpha}
\bibliography{bibliogeneral}	

\end{document}